\documentclass[12pt,reqno]{amsart}
\usepackage{amsthm,amsmath,amsfonts,amssymb}
\usepackage[colorlinks,
            linkcolor=blue,
            anchorcolor=blue,
            citecolor=blue]{hyperref}
\usepackage{mathtools}
\allowdisplaybreaks[1]
\usepackage{commath}
\usepackage{float}
\usepackage{graphicx}
\usepackage{epstopdf}
\usepackage{setspace}
\usepackage{cite}
\usepackage{bm}
\usepackage[noabbrev,capitalise]{cleveref}
\usepackage{enumitem}
\numberwithin{equation}{section}

\newcommand{\vp}{\varphi}
\usepackage[hmarginratio=1:1]{geometry}      
\newcommand{\defs}{\coloneqq}
\DeclareMathOperator{\tr}{\mathrm{Trace}}

\makeatletter

\newcommand{\Rmnum}[1]{\expandafter\@slowromancap\romannumeral#1@}
\makeatother

\theoremstyle{plain}
\newtheorem{thm}{Theorem}[section]
\newtheorem{prop}{Proposition}[section]
\newtheorem{lem}{Lemma}[section]
\newtheorem{rmk}{Remark}[section]
\newcommand{\Real}{\mathbb R}

\newcommand{\mcw}{\mathcal{W}}

\newcommand{\mbe}{\mathbf e}

\newcommand{\mbx}{\mathbf x}
\newcommand{\mby}{\mathbf y}
\newcommand{\mbz}{\mathbf z}

\begin{document}
\title[Unsteady potential flows near a corner]{Local well-posedness of unsteady potential flows near a space corner of right angle} 
\author{Beixiang Fang}

\author{Wei Xiang}

\author{Feng Xiao}

\address{B.X. Fang: School of Mathematical Science, MOE-LSC, and SHL-MAC, Shanghai Jiao
Tong University, Shanghai 200240, China}

\email{\texttt{bxfang@sjtu.edu.cn}}

\address{W. Xiang: Department of Mathematics, City University of Hong Kong, Hong Kong, China}

\email{\texttt{weixiang@cityu.edu.hk}}

\address{F. Xiao: Key Laboratory of Computing and Stochastic Mathematics (Ministry of Education), School of Mathematics and Statistics, Hunan Normal University, Changsha, Hunan 410081, P. R. China
}

\email{\textbf{xf@hunnu.edu.cn}}
\keywords{Local well-posedness; Potential Flow Equations; Corner Singularity; Co-normal Boundary Condition; Quasi-linear Hyperbolic Equation}

\subjclass[2000]{35L65, 35B35, 35L72, 35L04}
\date{\today}

\begin{abstract}
	
	In this paper we are concerned with the local well-posedness of the unsteady potential flows near a space corner of right angle, which could be formulated as an initial-boundary value problem of a hyperbolic equation of second order in a cornered-space domain.
	The corner singularity is the key difficulty in establishing the local well-posedness of the problem.
	Moreover, the boundary conditions on both edges of the corner angle are of Neumann-type and fail to satisfy the linear stability condition, which makes it more difficult to establish \emph{a priori} estimates on the boundary terms in the analysis.
	In this paper, extension methods will be updated to deal with the corner singularity, and, based on a key observation that the boundary operators are co-normal, new techniques will be developed to control the boundary terms.

\end{abstract}

\maketitle
\tableofcontents{}

\section{Introduction}\label{sec:1}

In this paper we are concerned with the local existence of the unsteady inviscid compressible flows near a corner of right angle.
Physically, compressible flows near corners can be observed easily and continuously.
However, within the best extent of our knowledge, the mathematical analysis for such phenomena is far away from satisfied.
In this paper, we are going to study this problem and trying to establish a mathematical theory on the local well-posedness of the unsteady potential flow near a 2-D corner of right angle.

In this paper, the inviscid compressible flow is assumed to be isentropic and irrotational such that its motion can be governed by the following 2-D unsteady potential flow equations:
\begin{equation}\label{1.1}
\left\{
\begin{array}{ll}
\partial_t\rho+\mbox{div}(\rho\nabla\Phi)=0\qquad\qquad\qquad (\mbox{Conservation\ of\ mass})\\
\partial_t\Phi+\frac12|\nabla\Phi|^2+\imath(\rho)=B_0\ \quad\qquad (\mbox{Bernoulli's\ law})
\end{array}
\right.
\end{equation}
where $\nabla\defs(\partial_{x_1},\partial_{x_2})^\top$ is the gradient operator with respect to the spatial variable $\mathbf{x}\defs(x_1,x_2)$ and $t$ is the time variable. 
Moreover, $\rho$ is the density, and $\Phi$ is the velocity potential, \emph{i.e.}, the gradient $\nabla\Phi$ is the fluid velocity. 
The fluid is assumed to be a polytropic gas such that the enthalpy $\imath (\rho)=\displaystyle\frac{\rho^{\gamma-1}-1}{\gamma-1}$, where $\gamma>1$ is the adiabatic exponent.
Finally, $c=\sqrt{\rho^{\gamma-1}}$ is the sonic speed, and $B_0$ is the Bernoulli's constant.

The Bernoulli's law, the second equation of \eqref{1.1}, implies that the density $\rho$ can be expressed as a function with respect to $(\partial_{t}\Phi,\nabla\Phi)$:
\begin{equation}\label{eq:density}
\rho = \varrho(\partial_t\Phi, \nabla\Phi;\gamma,B_0)\coloneqq (1 + (\gamma-1)(B_{0} - \partial_t\Phi - \frac{1}{2}|\nabla\Phi|^2))^{\frac{1}{\gamma-1}}.
\end{equation}
Substituting \eqref{eq:density} into the first equation of \eqref{1.1}, we deduce that the velocity potential function $\Phi$ satisfies the following equation of second order:
\begin{equation}\label{eq:potential-equations}
\partial_{tt}\Phi+2\sum_{i=1}^2\partial_{x_i}\Phi\partial_{tx_i}\Phi-\sum_{i,j=1}^2(\delta_{ij}c^2-\partial_{x_i}\Phi\partial_{x_i})\partial_{x_ix_j}\Phi=0,
\end{equation}
where $\delta_{ij}$ is the Kronecker delta.
Let $a_{00}\defs 1$ and
\begin{align}
	&a_{ij}=a_{ji}\defs-c^2\delta_{ij}+\partial_{x_i}\Phi\partial_{x_j}\Phi, \quad&&\mbox{for\ } i,j\ge 1,\\
	&a_{0j}=a_{j0}\defs\partial_{x_j}\Phi,\quad&&\mbox{for\ } j=1,2.
\end{align}
Then equation \eqref{eq:potential-equations} can be denoted by
\begin{align}
	\sum_{i,j=0}^2a_{ij}\partial_{x_ix_j}\Phi=0.\label{1.6}
\end{align}

It is well-known that, as $ \rho \not=0 $, namely, vacuum does not appear in the flow, the equation \eqref{eq:potential-equations} ( or \eqref{1.6} ) is of hyperbolic type.

\begin{figure}[H]
	\centering
	\includegraphics[scale=0.7]{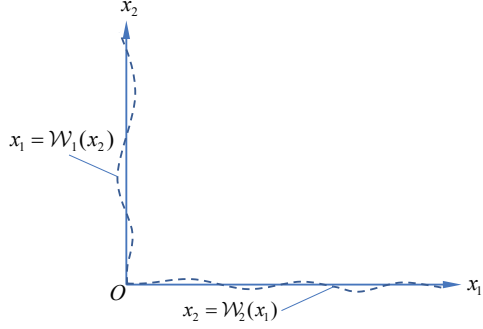}
	\caption[Cornered-space domain with slightly curved boundaries]{A cornered-space domain with slightly curved boundaries.}
	\label{fig:perturbation}
\end{figure}
Let $\Gamma_{w_1}$ and $\Gamma_{w_2}$ be two curves defined as
\[
\begin{split}
	\Gamma_{w_1} &=\{(x_1,x_2)\in\mathbb{R}^2;{\ }x_1=\mcw_1(x_2),{ }x_2>0\},\\
	\Gamma_{w_2} &=\{(x_1,x_2)\in\mathbb{R}^2;{\ }x_1>0,{\ }x_2=\mcw_2(x_1)\}.
\end{split}
\]
The following two assumptions will be imposed on these two curves: 
\begin{enumerate}
	\item[(\textbf{A1})] $\Gamma_{w_1}\cap\Gamma_{w_2}=\{(0,0)\}$.
	
	\item[(\textbf{A2})] $\mcw_i^{\prime}(0)=\mcw_i^{''}(0)=\mcw_i^{'''}(0)=0$ for $i=1,2$, here and after the superscripts $^\prime$, $^{\prime\prime}$ and $^{\prime\prime\prime}$ stand for the derivative of corresponding variable of the first, second and the third order.
\end{enumerate}
Obviously, $\Gamma_{w_1}$ and $\Gamma_{w_2}$ are perpendicular to each other at the origin $ (0,0) $.
Denote the cornered-space domain (see Figure \ref{fig:perturbation}) bounded by $\Gamma_{w_1}$ and $\Gamma_{w_2}$ by $\mathcal{D}$, i.e.,
\[\mathcal{D}\defs\{\mathbf{x}\in\Real^{2}:\ x_1>\mcw_1(x_2){\ }\mbox{and}{\ }x_2>\mcw_2(x_1)\}.\]

The fluid is confined in $\mathcal{D}$, and the velocity potential function $\Phi$ satisfies the following slip boundary conditions on the boundaries $\Gamma_{w_1}$ and $\Gamma_{w_2}$:
\begin{align}
 &\partial_{x_1}\Phi-\mcw^\prime_1(x_2)\partial_{x_2}\Phi=0, &\mbox{\ on\ } &\Gamma_{w_1},\label{vertical slip boundary condition}\\
 &\partial_{x_2}\Phi-\mcw^\prime_2(x_1)\partial_{x_2}\Phi=0, &\mbox{\ on\ } &\Gamma_{w_2}\label{horizontal slip boundary condition}.
\end{align}

When $t=0$, the state of the fluid is given such that the velocity potential function $\Phi$ is equipped with the initial conditions:
\begin{equation}\label{initial conditions}
    \Phi(0,\mbx)=\Phi_0(\mbx)\quad \mbox{and}\quad  \partial_t\Phi(0,\mbx)=\Phi_1(\mbx), \qquad  \mbx\in\mathcal{D}.
\end{equation}

Thus, to show the local existence of the potential flow in the cornered-space domain  $\mathcal{D}$, one needs to prove the local well-posedness of the initial-boundary value problem \eqref{1.6}, \eqref{vertical slip boundary condition}, \eqref{horizontal slip boundary condition}, and \eqref{initial conditions}.
Since the equation \eqref{1.6} is of hyperbolic type which enjoys the property of finite speed of propagation, one may further assume that $\Gamma_{w_1}$ and $\Gamma_{w_2}$ are small perturbation of straight lines, and the initial conditions \eqref{initial conditions} describe small perturbation of the state of the flow near the corner $ (0,0) $.

Since  $\Phi$ satisfies both \eqref{vertical slip boundary condition} and \eqref{horizontal slip boundary condition} at the corner point $ (0,0) $, one immediately obtains that $ \nabla\Phi(t,0,0) = (0,0) $ and the flow is static.
Let the density $ \rho_0 $ at the corner point $ (0,0) $ be
\[
	\rho_0\defs \varrho(0,0,0;\gamma,B_0)=(1+(\gamma-1)B_0)^{\frac{1}{\gamma-1}}.
\]
Moreover, let
\[
\begin{split}
\bar\Gamma_{w_1} &\defs\{(x_1,x_2)\in\mathbb{R}^2;{\ }x_1=\overline{\mcw}_1(x_2)\equiv 0,{\ }x_2>0\},\\
\bar\Gamma_{w_2} &\defs\{(x_1,x_2)\in\mathbb{R}^2;{\ }x_1>0,{\ }x_2=\overline{\mcw}_2(x_1)\equiv 0,\},
\end{split}
\]
%
and the domain be bounded by them (see Figure \ref{fig:background state})
\[\overline{\mathcal{D}}\defs\{\mathbf{x}\in\Real^{2}:\ x_1>0{\ }\mbox{and}{\ }x_2>0\}.\]
It is obvious that 
$$(\overline{\Phi}(t,\mbx),\bar{\rho}(t,\mbx))\defs(0,\rho_0)$$ 
is a steady solution to the unsteady potential flow equations \eqref{1.1}. It can be further verified that $\overline{\Phi}(t,\mbx)$ is a steady solution to equation \eqref{eq:potential-equations} in $ \overline{\mathcal{D}} $, satisfying the slip boundary conditions on $ \bar{\Gamma}_{w_1} $ and $ \bar{\Gamma}_{w_2} $. 

\begin{figure}[h]\label{background state}
\centering
{\includegraphics[width=0.3\columnwidth]{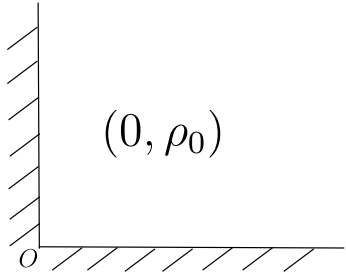}}\caption{A steady solution to  \eqref{1.1}}\label{fig:background state}
\end{figure}

Therefore, the initial-boundary value problem \eqref{1.6}, \eqref{vertical slip boundary condition}, \eqref{horizontal slip boundary condition}, and \eqref{initial conditions} for generic smooth initial and boundary data can be reduced to the stability problem for the steady solution $\overline{\Phi}(t,\mbx)$ under small perturbation of the boundary and the initial data.

\vskip 1em
\noindent
$\textbf{Problem 1:}$
Suppose $\Gamma_{w_1}$ and $\Gamma_{w_2}$ and the initial data $(\Phi_0(\mbx),\Phi_1(\mbx))$ satisfy the following conditions:
\begin{itemize}
	\item The boundaries $\Gamma_{w_1}$ and $\Gamma_{w_2}$ satisfy assumptions ($\textbf{A1}$) and ($\textbf{A2}$). Moreover, $\Gamma_{w_1}$ and $\Gamma_{w_2}$ are small perturbations of the straight boundaries $\overline{\Gamma}_{w_1}$ and $\overline{\Gamma}_{w_2}$, respectively such that $\mathcal{W}_1(x_2)$ and $\mathcal{W}_2(x_1)$, as well as their derivatives, are close to zero.
	\item The initial data are small perturbations of $\overline{\Phi}(t,\mbx)$, i.e., $(\Phi_0(\mathbf{x}),\Phi_1(\mathbf{x}))$ is close to $(0,0)$.
\end{itemize} 
Does there exist a unique local classical solution ${\Phi}(t,\mbx)$ to the equation $\eqref{eq:potential-equations}$ in the cornered-space domain $ \mathcal{D} $ with initial-boundary conditions \eqref{vertical slip boundary condition}-
\eqref{initial conditions}, which is still close to $\overline{\Phi}(t,\mbx)$
? 

\medskip
This paper is devoted to investigating $\textbf{Problem 1}$ and will give a positive answer by proving Theorem \ref{main theorem}, which is the main theorem of this paper. The key difficulty is the corner singularity on the boundary and one needs to analyse the behaviour of the solution near the corner point.
To the best of our knowledge, up to now, a general theory on well-posedness of initial-boundary value problems of hyperbolic systems on non-smooth domains is not available.
Nevertheless, there are progresses toward this issue.
In \cite{Osher,Osher1}, Osher gives ill-posed examples of initial-boundary value problems of hyperbolic equations on a non-smooth domain, which shows the complexity of such problems.
There are also well-posed results on domains with corners.
In particular, as the corner angle is sufficiently small, in \cite{GodinMAMS,Godin} Godin derives local well-posedness of smooth solutions for two dimensional Euler system in bounded domains with finite corner points.
As the corner is a right angle, under certain symmetry assumptions, Gazzola-Secchi obtains the well-posedness of Euler equations in rectangular cylinders in \cite{GS} and Yuan establishes the stability of normal shocks in 2-D flat nozzles for two dimensional unsteady Euler system in \cite{Yuan2012}. 
Both in \cite{GS} and \cite{Yuan2012}, the symmetry assumptions play an essential role such that extension techniques can be employed to reduced the problem near the corner into an initial-boundary value problem on a domain with smooth boundaries. 
By developing new techniques based on the extension method, it is established in \cite{FHXX} by Fang-Huang-Xiang-Xiao  and in \cite{FXX} by Fang-Xiang-Xiao the local dynamic stability of steady normal shock solutions for unsteady potential flows under small perturbation of the physical boundary without the symmetry assumptions in \cite{GS} and \cite{Yuan2012}. 

It turns out that the ideas and techniques developed in \cite{FHXX,FXX} help to deal with the difficulties brought by the corner singularity in  $\textbf{Problem 1}$.
While new difficulties arise because the boundary conditions on both edges of the corner angle do not satisfy the linear stability conditions, which hold on the shock front, one of the edges of the corner angle, in \cite{FHXX,FXX}.
This fact makes it difficult to establish \emph{a priori} estimates on the boundary trace of the highest order derivatives of the solution and the techniques in \cite{FHXX,FXX} and \cite{Majda1987} do not work.
New techniques will be developed in this paper to control the boundary terms, based on a key observation that the boundary operators are co-normal (see $(\romannumeral2)$ in lemma \ref{lem: 2.1}).

It is well-known that in general unsteady flows governed by quasilinear hyperbolic systems of conservation laws will formulate singularities in finite time and nonlinear waves such as shocks, rarefaction waves, contact discontinuities, etc. may occur.
Thanks to continuous efforts of many mathematicians, there have been systematic theory for one space dimensional cases, see \cite{Bressan,Dafermos,Smoller} and the references cited therein.
For multidimensional problems, important progresses have also been made in the past decades.
Dynamical stability of the elementary nonlinear waves have been established in, for instance, \cite{Alinhac1989,AlinhacCPDE,CouSecchi,CouSecchi2004,CouSecchi2009,Majda1983E,Majda1983S,Majda1987} by employing the well-established mathematical theory for initial-boundary value problems of hyperbolic systems on smooth domains.
There are also progresses on self-similar solutions for important physical phenomena such as shock reflections, supersonic flows onto a wedge, interaction between the elementary nonlinear waves, etc.
See, for instance, \cite{BCFMemoris,BCF,Chen2003Multidimensional,CFHX,CFeldman2010,CFeldman2018,CFeldmanX,Chen5,EL,LiZheng2009,LiZheng2010,Zheng_Book}.
See also \cite{FHXX,FXX,GS,Godin,GodinMAMS,Li-Witt-Yin2018,XuYin2020,Yuan2012} for the studies on multidimensional problems on non-smooth domains.

The remainder of this paper is organized as follows. In section \ref{sec:2}, the initial boundary value problem in the cornered-space domain $\mathcal{D}$ is reformulated to a new one with straight boundaries by introducing coordinate transformations. Under the transformations, certain coefficients vanish on the boundary of the space domain such that extension technique can be employed. Moreover, the boundary operators remain co-normal.
 Section \ref{sec:3} is devoted to the well-posedness of the linearized problem.  The extension techniques will be employed to establish the existence of the solution in $H^2$ spaces, similar as in \cite{FHXX,FXX}. 
In order to carry out the nonlinear iteration, $H^2$ regularity is not sufficient and the solution of the linearized problem should enjoy higher order regularity.
However, the regularity of the coefficients of the equation in the extended domain is not sufficient to establish higher order \emph{a priori} estimates, so they have to be established directly in the cornered-space domain. 
In section \ref{sec:4}, based on the observation that the boundary operators are co-normal, new techniques will be developed to establish the higher order estimates, in particular, the estimates of the highest order derivatives of the solution on the boundary. 
 In section \ref{sec:5}, a classical iteration scheme will be carried out which converges to the local solution of the reformulated nonlinear initial-boundary value problem in the cornered-space domain.

\section{Coordinate transformation and main result}\label{sec:2}
In this section, we will introduce two coordinate transformations $\mathbb{T}_1$ and $\mathbb{T}_2$, 
under which the boundaries $\Gamma_{w_1}$ and $\Gamma_{w_2}$ are straightened and the extension technique can be used to solve the initial boundary value problem in new coordinate system. 
First we introduce the following transformation to straighten $\Gamma_{w_2}$
\begin{equation}\label{nonlinear initial boundary value problem}
    \mathbb{T}_1:\left\{\!\!
\begin{array}{ll}
    y_0=t,\\
    y_1=-x_1-\int_0^{x_1}\mcw_2^{\prime 2}(\tau)\mathrm{d}\tau-(x_2-\mcw_2(x_1))\mcw'_2(x_1),\\
    y_2=x_2-\mcw_2(x_1).
\end{array}
\right.
\end{equation}
Define $\tilde{\Phi}(y_0,y_1,y_2):=\Phi(t,x_1,x_2)$.
Then one has
\begin{align}
    \partial_t\Phi(t,x_1,x_2)&=\partial_{y_0}\tilde{\Phi}(y_0,y_1,y_2),\\
     \partial_{x_1}\Phi(t,x_1,x_2)&=(-1-(x_2-\mcw_2(x_1))\mcw_2''(x_1))\partial_{y_1}\tilde{\Phi}(y_0,y_1,y_2)\nonumber\\
     &\qquad-\mcw_2'(x_1)\partial_{y_2}\tilde{\Phi}(y_0,y_1,y_2),\\
    \partial_{x_2}\Phi(t,x_1,x_2)&=-\mcw_2'(x_1)\partial_{y_1}\tilde{\Phi}(y_0,y_1,y_2)+\partial_{y_2}\tilde{\Phi}(y_0,y_1,y_2).
\end{align}
Let $(y_0,\mby)\defs(y_0,y_1,y_2)$ be the time-spatial variables in new coordinate system, then one has
$$
\mathbf{J}\defs\frac{\partial (y_0,\mby)}{\partial(t,\mbx)}=
\begin{pmatrix}
   1&0&0\\
   0&-1-(x_2-\mcw_2(x_1))\mcw_2''(x_1)&-\mcw_2'(x_1)\\
   0&-\mcw_2'(x_1)&1
\end{pmatrix}.
$$
When $\mcw_2(x_1)$ and its derivatives are sufficiently small, $\mathbb{T}_1$ is invertible. We assume  
\begin{equation}\label{inverse of T1}
\mathbb{T}_1^{-1}:\left\{
\begin{array}{ll}
t=y_0,\\
x_1=u(y_1,y_2),\\
x_2=y_2+\mcw_2(u(y_1,y_2)),
\end{array}
\right.
\end{equation}
where $u(y_1,y_2)$ is the expression of $x_1$ in $(y_0,\mby)$-coordinate determined by $\mathbb{T}_1$. From now on, for the shortness, we omit the dependence of varibles of $\mcw_1(x_2)$ and $\mcw_2(x_1)$ and always keep in mind that $\mcw_1$ is a function of $x_2$ and $\mcw_2$ is a function of $x_1$.
By direct calculation, one has
\begin{equation*}
\begin{aligned}
    \mathrm{D}^2\Phi&=
    \mathbf{J}^\top\mathrm{D}^2_{\mby}\tilde{\Phi}\mathbf{J}-\{((x_2-\mcw_2)\mcw_2'''-\mcw_2'\mcw_2'')\tilde{\Phi}_{y_1}+\mcw_2''\tilde{\Phi}_{y_2}\}\mbe_1^\top\mbe_1\\
    &\qquad\qquad\qquad-\mcw_2''\tilde{\Phi}_{y_1}\mbe_1^\top\mbe_2-\mcw_2''\tilde{\Phi}_{y_1}\mbe_2^\top\mbe_1.
\end{aligned}
\end{equation*}
where $\mbe_1=(0,1,0)\in \mathbb{R}^3$ and $\mbe_2=(0,0,1)\in \mathbb{R}^3$.
Let $\mathbf{A}\defs\left(a_{ij}\right)_{3\times 3}$, then one has 
\begin{equation*}
    \begin{aligned}
    \sum_{i,j=0}^2a_{ij}\partial_{x_ix_j}\Phi&=\tr(\mathbf{A}^\top\mathrm{D}^2\Phi)\\
    &=\tr(\mathbf{J}^\top \mathbf{A}\mathbf{J}\mathrm{D}^2_{\mby}\tilde{\Phi})-a_{11}\{((x_2-\mcw_2)\mcw_2'''-\mcw_2'\mcw_2'')\partial_{y_1}\tilde{\Phi}+\mcw_2''\partial_{y_1}\tilde{\Phi}\}\\
    &\qquad\qquad\qquad\qquad\quad-2a_{12}\mcw_2''\partial_{y_1}\tilde{\Phi}\\
    &=\sum_{i,j=0}^2\tilde{a}_{ij}\partial_{y_iy_j}\tilde{\Phi}-(a_{11}((x_2-\mcw_2)\mcw_2'''-\mcw_2'\mcw_2'')+2a_{12}\mcw_2'')\partial_{y_1}\tilde{\Phi}\nonumber\\
    &\qquad\qquad\qquad\qquad\quad-a_{11}\mcw_2''\partial_{y_2}\tilde{\Phi},
\end{aligned}
\end{equation*}
where $\tilde{a}_{ij}$ is the $(i,j)$-th entry of $\mathbf{J}^\top \mathbf{A}\mathbf{J}$ such that
\[
\tilde{a}_{ij}=\sum_{k,\ell=0}^2\mathbf{J}_{ki}a_{k\ell}\mathbf{J}_{\ell j}.
\]
By direct calculation, we have
    \begin{align}
    &\tilde{a}_{00}=1,\\
    &\tilde{a}_{02}=\tilde{a}_{20}
    =-\mcw_2'\Phi_{x_1}+\Phi_{x_2},\label{2.4}\\
    &\tilde{a}_{22}=a_{11}\mcw_2'^2-2a_{21}\mcw_2'+a_{22},\\
    &\tilde{a}_{01}=\tilde{a}_{10}=-a_{01}(1+(x_2-\mcw_2)\mcw_2'')-a_{02}\mcw_2',\\
    &\tilde{a}_{12}=\tilde{a}_{21} =a_{11}(1+(x_2-\mcw_2)\mcw_2'')\mcw_2'-a_{12}(1+(x_2-\mcw_2)\mcw_2'')\nonumber\\
    &\qquad\qquad\quad{\ }-a_{22}\mcw_2'+a_{21}\mcw_2'^2,\label{2.5}\\
    &\tilde{a}_{11}=a_{11}(1+(x_2-\mcw_2)\mcw_2'')^2+a_{21}(-\mcw_2')(-1-(x_2-\mcw_2)\mcw_2'')\nonumber\\
    &\qquad{\ }\qquad\quad-\mcw_2'(-a_{12}(1+(x_2-\mcw_2)\mcw''_2)-\mcw_2'a_{22}).
\end{align}
 
It is clear that boundary $\Gamma_{w_2}$ becomes $\{y_2=0\}$ in the $(y_0,\mby)$-coordinate. Let 
\[F(y_1,y_2)\defs x_1(y_1,y_2)-\mcw_1(x_2(y_1,y_2)).
\] 
Then one has $F(0,0)=0$, and
in $(y_0,\mby)$-coordinate, the boundary $\Gamma_{w_1}$ can be expressed as $F(y_1,y_2)=0$. By direct calculation, one has 
\[\frac{\partial F}{\partial y_1}\neq 0,\]
 provided that the perturbations $\mcw_1$ and $\mcw_2$ are sufficiently small. Then by the implicit function theorem, $F(y_1,y_2)=0$ determines a function $y_1=\sigma(y_2)$ with $\sigma(0)=0$. 
Hence the boundary $\Gamma_{w_1}$ in the $\mby$-coordinate can be expressed as $y_1=\sigma(y_2)$.
\vspace{-0.5em}
 \begin{figure}[h]
	\centering
	\setlength{\abovecaptionskip}{-1.1cm}
	\includegraphics[width=1\textwidth,height=0.4\textwidth]{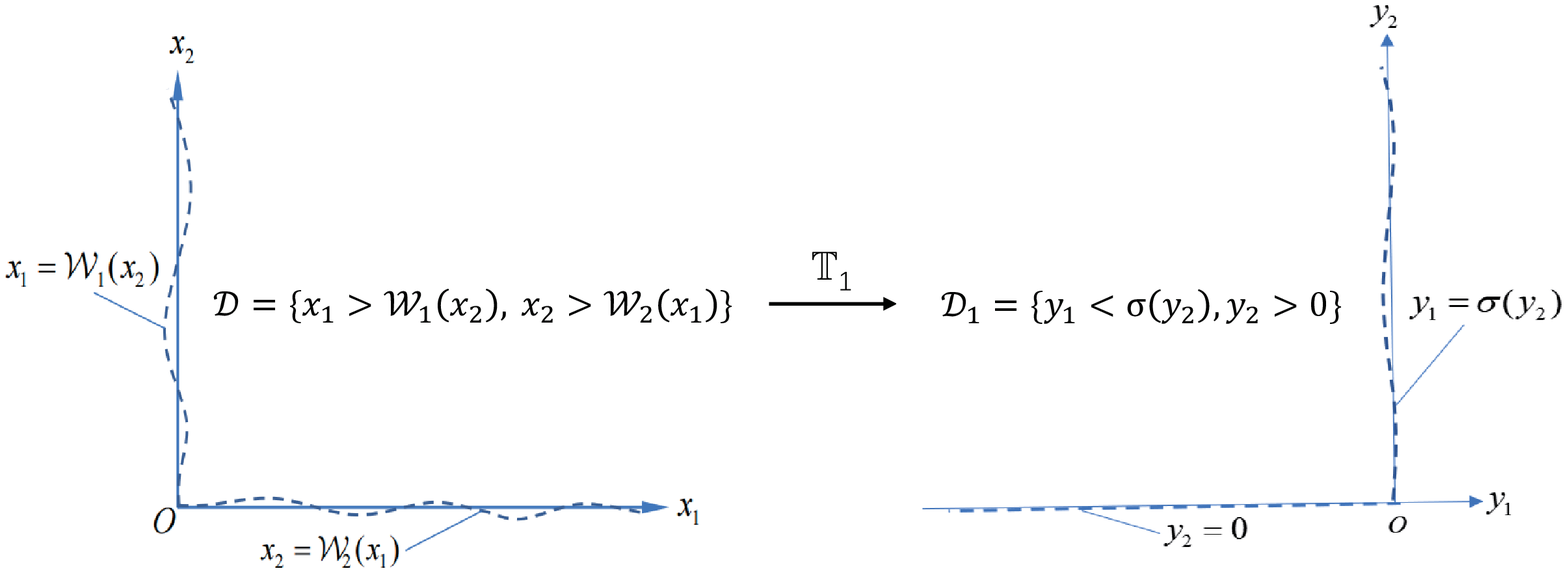}
	\caption{Coordinate transformation $\mathbb{T}_1$}
\end{figure} 
Therefore, under the transformation $\mathbb{T}_1$, the space domain $\mathcal{D}$ is converted to 
\[
\mathcal{D}_1=\{(y_1,y_2)\in\mathbb{R}^2;y_1<\sigma(y_2),y_2>0\},
\]
with boundary $\partial\mathcal{D}_1=\tilde{\Gamma}_1\cup\tilde{\Gamma}_2$, where
\begin{align}
	\tilde{\Gamma}_1&\defs\{(y_1,y_2)\in\mathbb{R}^2;y_1=\sigma(y_2),y_2>0\},\nonumber\\
	\tilde{\Gamma}_2&\defs\{(y_1,y_2)\in\mathbb{R}^2;y_1<0,y_2=0\}.\nonumber
\end{align}
By straightforward calculation, we find that the slip boundary conditions become
\begin{align}
	(-1-y_2\mcw_2''+\mcw_1'\mcw'_2)\partial_{y_1}\tilde{\Phi}-(\mcw'_1+\mcw_2')\partial_{y_2}\tilde{\Phi}&=0,\mbox{\ on\ }\tilde{\Gamma}_1,\label{vertical bdr condition in y coordinate}\\
	\partial_{y_2}\tilde{\Phi}&=0,\mbox{\ on\ }\tilde{\Gamma}_2.\label{horizontal bdry cond in y coordinate}
\end{align}
The initial conditions in the $(y_0,\mby)$-coordinates are
\begin{align}
    \tilde{\Phi}(0,y_1,y_2)&=\Phi_0(u,y_2+\mcw_2(u)),\\
    \partial_{y_0}\tilde{\Phi}(0,y_1,y_2)&=\Phi_1(u,y_2+\mcw_2(u)).
\end{align}
Next, we straighten $\Gamma_1$ by introducing the following coordinate transformation:
\begin{equation}
\mathbb{T}_2:\left\{\!\!
\begin{array}{ll}
    z_0=y_0,\\
    z_1=-y_1+\sigma(y_2),\\
    z_2=y_2.
\end{array}
\right.
\end{equation}
Then one has 
\begin{align}
x_1&=u(-z_1+\sigma(z_2),z_2),\nonumber\\
x_2&=\mcw_2(u(-z_1+\sigma(z_2),z_2))+z_2,\nonumber
\end{align}
where $u=u(y_1,y_2)$ is the expression of $x_1$ in $(y_0,\mby)$-coordinate determined by $\mathbb{T}_1$, as in \eqref{inverse of T1}.
Under transformation $\mathbb{T}_2$, the space domain $\mathcal{D}_1$ is mapped to
\[
\Omega\defs\{(z_1,z_2)\in\mathbb{R}^2;z_1>0,z_2>0\},
\]
\vspace{-1em}
\begin{figure}[h]
	\centering
	\includegraphics[width=0.9\textwidth,height=0.3\textwidth]{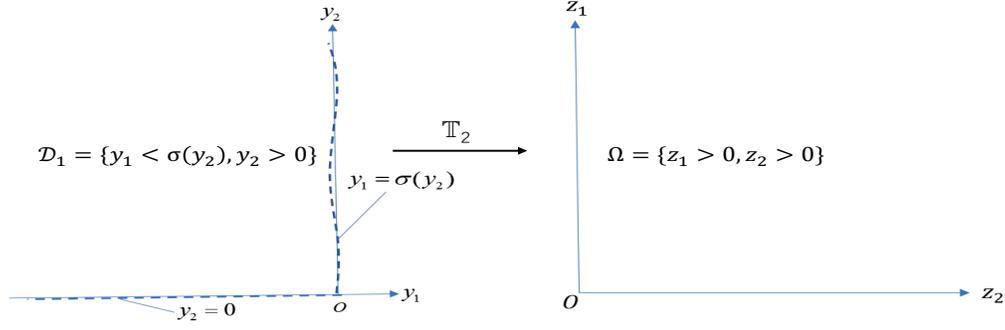}
	\caption{Coordinate transformation $\mathbb{T}_2$.}
\end{figure}
\vskip 1em
with boundary $\partial\Omega=\Gamma_1\cup\Gamma_2$,
where
\begin{align}
	\Gamma_1&\defs\{(z_1,z_2)\in \mathbb{R}^2;{\ }z_1=0,z_2>0\}, \nonumber\\
	\Gamma_2&\defs\{(z_1,z_2)\in \mathbb{R}^2;{\ }z_1>0,z_2=0\}.\nonumber
\end{align}

Define $\hat{\Phi}(z_0,z_1,z_2)\defs\tilde{\Phi}(y_0,y_1,y_2)$.
Then $\hat{\Phi}(z_0,z_1,z_2)$ satisfies 
$$\sum_{i,j=0}^2\alpha_{ij}\partial_{z_iz_j}\hat{\Phi}+\sum_{i=0}^2\alpha_i\partial_{z_i}\hat{\Phi}=0,$$
where
\begin{align}
    \alpha_{00}&=\tilde{a}_{00}=1,\label{alpha00}\\
    \alpha_{11}&=\tilde{a}_{11}-\sigma^\prime(y_2)\tilde{a}_{12}-\sigma^\prime(y_2)\tilde{a}_{21}+(\sigma^{\prime })^2(y_2)\tilde{a}_{22},\label{alpha11}\\
    \alpha_{22}&=\tilde{a}_{22},\\
    \alpha_{02}&=\alpha_{20}=\tilde{a}_{02},\label{2.21}\\
    \alpha_{21}&=\alpha_{12}=\tilde{a}_{22}\sigma'(y_2)-\tilde{a}_{21},\label{alpha12}\\
  \alpha_{01}&=\alpha_{10}=\tilde{a}_{02}\sigma'(y_2)-\tilde{a}_{10},\label{alpha10}
\end{align}
and
\begin{align}
    \alpha_0&=0,\\
    \alpha_1&=a_{11}((x_2-\mcw_2)\mcw_2'''-\mcw_2'\mcw_2'')+2a_{12}\mcw_2'',\\
    \alpha_2&=-a_{11}\mcw''_2.
\end{align}
Boundary conditions \eqref{vertical bdr condition in y coordinate} and \eqref{horizontal bdry cond in y coordinate} become
\begin{align}
	\bar{b}_1(z_1,z_2)\partial_{z_1}\hat{\Phi}+\bar{b}_2(z_1,z_2)\partial_{z_2}\hat{\Phi}&=0,\mbox{\ on\ }\Gamma_1,\label{vertical bdr cond in first quadrant}	\\
	\partial_{z_2}\hat{\Phi}&=0,\mbox{\ on\ }\Gamma_2,\label{bdr cond on flat boundary}
\end{align}
where
\begin{align}
\bar{b}_1(z_1,z_2)&=p(z_1,z_2)+(\mcw_2'(x_1)+\mcw_1'(x_2))\sigma'(z_2)+\mcw'_1(x_2)\mcw_2'(x_1),\label{2.32}\\
\bar{b}_2(z_1,z_2)&=\mcw_1'(x_2(z_1,z_2))+\mcw_2'(x_1(z_1,z_2)),\label{2.33}
\end{align}
with
$$p(z_1,z_2)=-1-z_2\mcw_2''(x_1(z_1,z_2)).$$
Initial conditions \eqref{initial conditions} become
\begin{align}
    \hat{\Phi}(0,z_1,z_2)&=\hat{\Phi}_0(z_1,z_2),\label{initial condition_1}\\
    \quad \partial_{z_0}\hat{\Phi}(0,z_1,z_2)&=\hat{\Phi}_1(z_1,z_2),\label{initial condition_2}
\end{align}
where
\begin{align}
&\hat{\Phi}_0(z_1,z_2)\defs\tilde{\Phi}_0(u(z_1+\sigma(z_2)),\mcw_2(u(z_1+\sigma(z_2))),\nonumber\\ &\hat{\Phi}_1(z_1,z_2)\defs\tilde{\Phi}_1(u(z_1+\sigma(z_2)),\mcw_2(u(z_1+\sigma(z_2))).\nonumber	
\end{align}

Let $\Gamma_0\defs\{0\}\times\Omega$ and $\Omega_T:=(0,T)\times \Omega $, where $T$ is any positive real number.
We summarize the mathematical problem as follows:
\begin{equation}\label{eq:nolinear in z-coordinate}
\begin{cases}
  \alpha_{ij}\partial_{z_iz_j}\hat{\Phi}+\alpha_i\partial_{z_i}\hat{\Phi}=0,&\mbox{\ in\ \ }\Omega_T,\\  
    G(\partial_{z_1}\hat{\Phi},\partial_{z_2}\hat{\Phi};\mcw'_1,\mcw'_2,\mcw''_2,\sigma')=0,&\mbox{\ on\ }\Gamma_1,\\
    \partial_{z_2}\hat{\Phi}=0,&\mbox{\ on\ }\Gamma_2,\\
    \hat{\Phi}(0,z_1,z_2)=\hat{\Phi}_0(z_1,z_2),\mbox{\ \ }\partial_{z_0}\hat{\Phi}(0,z_1,z_2)=\hat{\Phi}_1(z_1,z_2),&\mbox{\ on \ }\Gamma_0,
\end{cases}
\end{equation}
where
\[
G(\partial_{z_1}\hat{\Phi},\partial_{z_2}\hat{\Phi};\mcw'_1,\mcw'_2,\mcw''_2,\sigma')\defs \bar{b}_1(z_1,z_2)\partial_{z_1}\hat{\Phi}+\bar{b}_2(z_1,z_2)\partial_{z_2}\hat{\Phi}.
\]
The initial data $\hat{\Phi}_0$ and $\hat{\Phi}_1$ are defined in \eqref{initial condition_1} and \eqref{initial condition_2}, respectively and 
 at the background state, one has
 \begin{equation}\label{2.30}
 \left(\alpha_{ij}\right)_{3\times 3}=\begin{pmatrix}
 1&0&0\\
 0&-\rho_0^{\gamma-1}&0\\
 0&0&-\rho_0^{\gamma-1}
 \end{pmatrix}.
\end{equation}

\begin{lem}\label{lem: 2.1}
	The coefficients of the initial boundary value problem \eqref{eq:nolinear in z-coordinate} have the following properties:
	\begin{enumerate}[label=(\roman*)]
		\item $\alpha_{02}=\alpha_{20}$ and $\alpha_{12}=\alpha_{21}$ vanish on $\Gamma_2$, namely, for $i=0,1$,
		\[\alpha_{i2}(z_0,z_1,0)=\alpha_{2i}(z_0,z_1,0)=0.\]
		 
		\item The boundary operator of \eqref{eq:nolinear in z-coordinate} is ``co-normal'', i.e., the vector $(\alpha_{10},\alpha_{11},\alpha_{12})\in\mathbb{R}^3$ is parallel to $(0,\bar{b}_1,\bar{b}_2)\in \mathbb{R}^3$ on $\Gamma_1$, where $\bar{b}_1$ and $\bar{b}_2$ are the coefficients of the boundary conditions on $\Gamma_1$ in \eqref{eq:nolinear in z-coordinate}.
		\item It holds that 
		\begin{align}
		\bar{b}_{2}(0,0)=\partial_{z_2}\bar{b}_2(0,0)=\partial_{z_2z_2}\bar{b}_2(0,0).\label{lem2.1-2.33}	
		\end{align}
	\end{enumerate}
\end{lem}
\begin{proof}

 $(\romannumeral1)$:
By the formula of $\tilde{a}_{02}$ in \eqref{2.4} and the slip boundary conditon \eqref{horizontal slip boundary condition}, one has $\tilde{a}_{02}=\tilde{a}_{20}=0$ on $\Gamma_2$. Then \eqref{2.21} implies $\alpha_{02}=\alpha_{20}=0$ on $\Gamma_2$.
By \eqref{2.5} and the slip boundary condition
\eqref{horizontal slip boundary condition}, one has
\begin{align}
	 \tilde{a}_{12}|_{z_2=0}=\partial_{x_1}\Phi(\mcw_2'\partial_{x_1}\Phi-\partial_{x_2}\Phi)+\mcw_2'\partial_{x_2}\Phi(\mcw_2'\partial_{x_1}\Phi-\partial_{x_2}\Phi)=0.\label{2.35}
\end{align}
Differentiating with respect to $y_2$ on both sides of  $x_1(\sigma(y_2),y_2)=\mcw_1(x_2(\sigma(y_2),y_2))$, 
one deduces
$$
\sigma^\prime(y_2)=\frac{p(x_1,x_2)\mcw_1'(x_2(\sigma(y_2),y_2))-\mcw_2'(x_1(\sigma(y_2),y_2))}{1-\mcw_1'\mcw_2'},
$$
where $$p(x_1,x_2):=-1-(x_2-\mcw_2)\mcw_2''.
$$
By assumption $(\mathbf{A2})$ and the facts that $\sigma(0)=0$, $\mathbb{T}_1$ mapps the origin in $\mbx$-coordinate to the origin in $\mby$-coordinate, and $\mathbb{T}_1$ is invertible, one has
$$
\sigma^\prime(0)=\frac{-\mcw'_1(x_2(\sigma(0),0))-\mcw_2'(x_1(\sigma(0),0))}{1-\mcw'_1(x_2(\sigma(0),0))\mcw_2'(x_1(\sigma(0),0)}=0.
$$
Combining \eqref{alpha12}, \eqref{2.35}, and the fact that $\sigma^\prime(0)=0$, one deduces that $\alpha_{12}=\alpha_{21}=0$ on $\Gamma_2$ (equivalently on $\tilde{\Gamma}_2$).
By assumption $(\mathbf{A2})$, it is obvious that $\bar{b}_2=0$ on $\Gamma_2$.

$(\romannumeral2):$
By the formulas of $\tilde{a}_{10}$, $\tilde{a}_{20}$, and $\sigma^\prime(y_2)$, we obtain
\begin{align}
-\tilde{a}_{10}+\sigma^\prime\tilde{a}_{20}&=-p\partial_{x_1}\Phi+\partial_{x_2}\Phi\mcw_2'-\sigma^\prime\mcw_2'\partial_{x_1}\Phi+\sigma^\prime\partial_{x_2}\Phi\nonumber\\
&=-(p+\sigma^\prime\mcw_2')\partial_{x_1}\Phi+(\mcw_2'+\sigma^\prime)\partial_{x_2}\Phi.
\end{align}
But it is easy to verify that $$\frac{\mcw_2^\prime+\sigma^\prime}{-p-\sigma^\prime\mcw_2'}=-\frac{\mcw_1'}{1}.$$
By the slip boundary condition on $\Gamma_{w_1}$, one has
$
-\tilde{a}_{10}+\sigma^\prime\tilde{a}_{20}=0
$ on $\Gamma_1$, which implies $\alpha_{01}=\alpha_{10}=0$ on $\Gamma_1$ by \eqref{alpha10}.
By the formulas of $\alpha_{11}$ and $\alpha_{12}=\alpha_{21}$ in \eqref{alpha11} and \eqref{alpha12}, one can deduce that
\begin{equation}\label{bdr observation}
\frac{\alpha_{12}}{\alpha_{11}}\Big|_{z_1=0}=\frac{\mcw_1'+\mcw_2'}{p+\mcw_1'\mcw_2'+\sigma'(\mcw_1'+\mcw_2')}=\frac{\bar{b}_2}{\bar{b}_1}.
\end{equation}
In fact, we have
\begin{align}
&(\tilde{a}_{11}-\sigma^\prime\tilde{a}_{12})|_{y_1\sigma(y_2)}\nonumber\\
&\qquad=c^2(-p^2-\mcw_2'^2-p\sigma^\prime\mcw_2'-\sigma^\prime\mcw_2')+(\partial_{x_1}\Phi)^2(p^2+p\sigma^\prime\mcw_2')\nonumber\\
&\qquad\qquad+(\partial_{x_2}\Phi)^2(\mcw_2'^2+\sigma^\prime\mcw_2')+\partial_{x_1}\Phi\partial_{x_2}\Phi(-2p\mcw_2'-p\sigma^\prime-\sigma^\prime\mcw_2'^2)\nonumber\\
&\qquad=c^2(-p^2-\mcw_2'^2-p\sigma^\prime\mcw_2'-\sigma^\prime\mcw_2')\nonumber\\
&\qquad\qquad+\partial_{x_1}\Phi(p^2+p\sigma^\prime\mcw_2')(\partial_{x_1}\Phi-\mcw_1'\partial_{x_2}\Phi)\nonumber\\
&\qquad\qquad-\frac{(\mcw_2'^2+\sigma^\prime\mcw_2')}{\mcw_1'}\partial_{x_2}\Phi(\partial_{x_1}\Phi-\mcw_1'\partial_{x_2}\Phi)\nonumber\\
&\qquad=c^2(-p^2-\mcw_2'^2-p\sigma^\prime\mcw_2'-\sigma^\prime\mcw_2'),
\end{align}
where in the last equality, we have used the slip boundary condition $\eqref{vertical slip boundary condition}$.
By direct calculation, one also has $$(-\tilde{a}_{12}+\sigma^\prime\tilde{a}_{22})|_{y_1=\sigma(y_2)}=-c^2(p\mcw_2'+\mcw_2'+\sigma^\prime\mcw_2'^2+\sigma^\prime).$$
Thus one has
\begin{align}
\frac{\tilde{a}_{11}-\sigma\tilde{a}_{12}}{\tilde{a}_{22}\sigma'-\tilde{a}_{12}}\bigg|_{z_1=0}&=\frac{p^2+\mcw_2'^2+p\sigma'\mcw_2'+\sigma'\mcw_2'}{p\mcw_2'+\mcw_2'+\sigma'\mcw_2'^2+\sigma'}\nonumber\\
&=\frac{p+\mcw_1'\mcw_2'}{\mcw_1'+\mcw_2'}.\nonumber
\end{align}

Then by the formula of $\alpha_{11}$ in \eqref{alpha11} and the formula of $\alpha_{12}$ in \eqref{alpha12}, we obtain
\begin{align}
\frac{\alpha_{11}}{\alpha_{12}}\bigg|_{z_1=0}&=\frac{\tilde{a}_{11}-\sigma\tilde{a}_{12}}{\tilde{a}_{22}\sigma'-\tilde{a}_{12}}\bigg|_{z_1=0}+\sigma'\nonumber\\
&=\frac{p+\mcw_1'\mcw_2'}{\mcw_1'+\mcw_2'}+\sigma'=\frac{\bar{b}_1}{\bar{b}_2}.\nonumber
\end{align}
Hence one has
\[
\frac{\alpha_{12}}{\alpha_{11}}\bigg|_{z_1=0}=\frac{\bar{b}_2}{\bar{b}_1}.
\]
$(\romannumeral3)$: It is clear that both $\mathbb{T}_1$ and $\mathbb{T}_2$ are invertible and $\mathbb{T}_1\circ \mathbb{T}_2$ maps the orgin in the $\mbx$-coordinate to the origin in the $\mbz$-coordinate, where $\mbz\defs(z_1,z_2)$ is the spatial variable in new coordinate under transformation $\mathbb{T}_2$. So by \eqref{2.33} and the assumption $(\textbf{A2})$ on the perturbations in section \ref{sec:1}, it is easy to see that \eqref{lem2.1-2.33} holds.
This completes the proof of this lemma.
\end{proof}
In $\mbz$-coordinate, \textbf{Problem 1} can be reformulated as the following problem:

\smallskip
\textbf{Problem 2:} Does there exsit a unique local classical solution to problem \eqref{eq:nolinear in z-coordinate}, when $(\mathcal{W}_1,\mathcal{W}_2)$ are small perturbations of $(\overline{\mcw_1},\overline{\mcw_2})=(0,0)$ and the initial data $(\hat{\Phi}_0(z_1,z_2),\hat{\Phi}_1(z_1,z_2))$ are small perturbations of $(0,0)$?

\smallskip
 Since coordinate transformations $\mathbb{T}_1$ and $\mathbb{T}_2$ are invertible when the perturbations are small, \textbf{Problem 1} and \textbf{Problem 2} are equivalent. The main theorem of the paper is as follows, which gives a positive answer to \textbf{Problem 2}.
\begin{thm}\label{main theorem}
Suppose the perturbed solid walls satisfy the three assumptions $(A1)$ and $(A2)$ in section \ref{sec:1}. Then there exists a constant $\epsilon>0$ such that if
\begin{align}
	\|\hat{\Phi}_0\|_{H^4(\Omega)}+\|\hat{\Phi}_1\|_{H^3(\Omega)}+\|(\mcw_1,\mcw_2)\|_{W^{6,\infty}(\mathbb{R}^+)}\le \epsilon,
\end{align}
and $\hat{\Phi}_0$ and $\hat{\Phi}_1$ satisfy the compatibility conditions up to order 2 and are compactly supported in some neighbours of the origin, then there exist two constants $\eta_0\ge 1$ and $T_0>0$ such that the nonlinear problem \eqref{eq:nolinear in z-coordinate} admits a unique smooth solution $\hat{\Phi}\in H^{4}(\Omega_{T_0})$, satisfying
\begin{align}
	\|e^{-\eta z_0}\hat{\Phi}\|_{H^{4}(\Omega_{T})}\le C\epsilon
\end{align}
for $\eta\ge \eta_0$ and $T\le T_0$,
where $C=C(\rho_0,\gamma,\eta_0,T_0)$ is a positive constant.
\end{thm}
\begin{rmk}\label{nonlinear compatibility conditons}
    The compatibility conditions mentioned in Theorem \ref{main theorem} come from the requirement that the initial-boundary data of problem \eqref{eq:nolinear in z-coordinate} should be consistent. More precisely, by initial conditions in \eqref{eq:nolinear in z-coordinate} and the first equation of \eqref{eq:nolinear in z-coordinate}, we know that at $z_0=0$,
\[
D^{\beta}\hat{\Phi}=D^{\beta}\hat{\Phi}_0,\quad \partial_{z_0}D^{\beta}\hat{\Phi}=D^{\beta}\hat{\Phi}_1
\]
and
\[
\quad \partial_{z_0}^2D^{\beta}\hat{\Phi}=-D^\beta(\frac{1}{\alpha_{00}}(\sum_{i=0}^2\alpha_i\partial_{z_i}\hat{\Phi}+\sum_{(i,j)\neq(0,0)}^2\alpha_{ij}\partial_{z_iz_j}\hat{\Phi})),
\]
where $D^{\beta}=\partial_{z_1}^{\beta_1}\partial_{z_2}^{\beta_2}$ is the spatial derivatives and $\beta=(\beta_1,\beta_2)$ is the multi-index corresponds to spatial derivative.
Then by the induction on $k$ (\emph{i.e.}, assume we have already known the expression of $\partial_{z_0}^{m+1}D^{\beta}\hat{\Phi}$ at $z_0=0$ for all $m\leq k$.) and by taking derivative $D^{\beta}\partial_{z_0}^k$ on equation $\eqref{eq:nolinear in z-coordinate}_{1}$, we will have the expression of $\partial_{z_0}^{k+2}D^{\beta}\hat{\Phi}$ at $z_0=0$. We omit the details for the shortness. Then we have the expression of $D^{\lambda}\hat{\Phi}$ at $z_0=0$ for all multi-index $\lambda=(\lambda_0,\lambda_1,\lambda_2)$. Let
\begin{equation}\label{3.29x}
\Phi_{\lambda}:=D^{\lambda}\hat{\Phi}\big|_{z_0=0}.
\end{equation}
On the other hand, since we have two boundary conditions in $\eqref{eq:nolinear in z-coordinate}$, for any $(k_0,k_1,k_2)\in\mathbb{N}^3$, we have
\begin{align}
	D^{(k_0,0,k_2)}\left(\sum_{i=1}^2\bar{b}_i\partial_{z_i}\hat{\Phi}\right)&=0\quad\mbox{on }\Gamma_1,\nonumber\\
	D^{(k_0,k_1,0)}\partial_{z_2}\hat{\Phi}&=0\quad\mbox{on }\Gamma_2.\nonumber
\end{align}
Let $z_0=0$ and plug \eqref{3.29x} into the two identities above for all integers $k_0+k_1\leq \mathfrak{s}$ and $k_0+k_2\leq \mathfrak{s}$. Then we can obtain the identities that the initial and boundary data must satisfy for all integers $k_0+k_1\leq \mathfrak{s}$ and $k_0+k_2\leq \mathfrak{s}$. These identities are called the compatibility conditions up to order $\mathfrak{s}$. In this paper, the compatibility conditions up to order two are required, i.e., $\mathfrak{s}=2$.
\end{rmk}

\section{The linearized problem (\Rmnum{1}): existence of the solution in $H^2(\Omega_T)$}\label{sec:3}

In order to establish the local well-posedness of the nonlinear initial-boundary value problem, the classical iteration scheme will be carried out. Then the unique existence of the solutions to the linearized problem \eqref{linear problem} as well as the \textit{a priori} estimates of the solution are needed to establish the convergence of the iteration. However, there is no general theory which could be employed, because of the presence of a corner singularity on the boundary of the space domain. Therefore, we have to establish a well-posedness theorem on the unique existence and \emph{a priori} estimates of the solution for the linearized problem, which will be done in this and the next section. 
In this section, the extension techniques will be employed to establish the existence of the solution in $H^2(\Omega_T)$.
Similar as in \cite{FHXX,FXX}, the regularity of the coefficients in the extended domain is not sufficient to establish higher order estimates of the solution.
Therefore, one have to establish the estimates for the higher order derivatives of the solution directly in the cornered space domain, which will be done in the next section.



\subsection{The main theorem on the linearized problem.}

Let us consider following initial-boundary value problem:
\begin{equation}\label{linear problem}
\begin{cases}
\mathcal{L}\varphi=f, &\mbox{in\ } \Omega_T,\\
\mathcal{B}\varphi=0,&\mbox{on \ }\Gamma_1,\\
\partial_{z_2}\varphi=0,&\mbox{on \ }\Gamma_2,\\
\varphi(0,z_1,z_2)=\varphi_0,&\mbox{on \ }\Gamma_0,\\
\partial_{z_0}\varphi(0,z_1,z_2)=\varphi_1(z_1,z_2),&\mbox{on \ }\Gamma_0,
\end{cases}
\end{equation}
where 
$$\mathcal{L}\defs \sum_{i,j=0}^2r_{ij} (z_0,z_1,z_2)\partial_{ij}
\qquad\mbox{and}\qquad
\mathcal{B}\defs b_1(z_1,z_2)\partial_1+b_2(z_1,z_2)\partial_2,
$$ 
$\partial_{i}\defs\partial_{z_i}$, and $\partial_{ij}=\partial_{z_iz_j}$.
The coefficients of $\mathcal{L}$ and $\mathcal{B}$ satisfy
\begin{enumerate}[label=(\roman*)]
\item  $\mathcal{L}$ is a hyperbolic differential operator of second order. $r_{ij}=r_{ij}(z_0,z_1,z_2)$ are smooth functions and $r_{00}\equiv 1$. $b_i$ are smooth functions with respect to $z_1$ and $z_2$ but do not depend on $z_0$.
\item There exist an integer $s_0\ge 3$, and constants $\delta>0$ and $\bar{r}_{ij}$ $(0\le i,j\le 2)$, such that $$\sup_{0\le z_0\le T}\|D^\alpha(r_{ij}-\bar{r}_{ij})\|_{L^2(\Omega)}<\delta \mbox{\ for\ all\ } |\alpha|\le s_0,$$
where $\bar{r}_{00}=1$, $\bar{r}_{11}=\bar{r}_{22}<0$, and $\bar{r}_{ij}=0$ for $i\neq j$.
We also require that
$$ |\partial^{\ell+1}_{z_2}b_1|+|\partial^{\ell}_{z_2}b_2|\le C\delta\qquad 
\mbox{for } \ell=0,1,2.
$$
\item $r_{12}=r_{21}=r_{20}=r_{02}=0$ on the boundary $\Gamma_2$ and $\partial_{z_2}^k b_2(0,0)=0$ for $k=0,1,2$.
\item The following identities hold on $\Gamma_1$:
$$
\frac{r_{12}}{r_{11}}=\frac{b_2}{b_1}\quad\mbox{ and}\quad r_{01}=r_{10}=0.
$$
\end{enumerate}
\begin{rmk}
The compatibility conditions up to order $\mathfrak{s}$ of problem \eqref{linear problem} can be defined in the same way as done in remark \ref{nonlinear compatibility conditons}. 
\end{rmk}
Then we have the following proposition.
\begin{prop}\label{well-posedness}
There exists $\delta_*>0$ such that if assumptions $(\romannumeral1)-(\romannumeral4)$ holds for $\delta\le \delta_*$, and if $\vp_0$ and $\vp_1$ satisfy the compatibility conditions up to order 2, problem $\eqref{linear problem}$ admits a smooth solution $\vp\in H^4(\Omega_T)$ and there exists a constant $\eta_0$, such that  for any $T>0$ and $\eta\geq\eta_0$, it holds that
    \begin{align}
    &\sum_{|\alpha|\le 4}\eta  \|e^{-\eta z_0}D^
    \alpha \vp\|^2_{L^2(\Omega_T)}+e^{-2\eta T}\|D^
    \alpha \vp(T,\cdot)\|^2_{L^2(\Omega)}\nonumber\\
    &\quad\lesssim \frac{1}{\eta}\sum_{|\alpha|\le 3}\|e^{-\eta z_0}\mathcal{L}(D^\alpha\vp)\|^2_{L^2(\Omega_T)}
    +\|e^{-\eta z_0}f\|^2_{H^{3}(\Omega_T)}\nonumber\\
    &\quad\quad+\|f|_{t=0}\|^2_{H^2(\Omega)}+\|\vp_0\|_{H^4(\Omega)}^2+\|\vp_1\|^2_{H^3(\Omega)}.\label{energy estimate}
\end{align}
\end{prop}
Without loss of generality, we assume $(\vp_0,\vp_1)=(0,0)$ in the following two sections below. Otherwise, one can reduce problem \eqref{linear problem} into a problem with homogeneous initial data by introducing auxiliary functions. 

The proof of Proposition \ref{well-posedness} will be separated into two parts. One is the unique existence of the solution in $H^2(\Omega_T)$, which will be established in this section. The other is the \emph{a priori} estimates of higher order derivatives, which will be done in the next section.

To establish the unique existence of the solution, we are motivated to apply the extension techniques by observing the boundary condition on $ \Gamma_2 $ as well as the properties of the coefficients enjoyed, such that the well-established theory for initial-boundary value problems of hyperbolic equations can be employed. We are going to show the following lemma in this section, which establishes the well-posedness of problem \eqref{linear problem} in $H^2(\Omega_T)$.

\begin{lem}\label{H2 solvability}
There exists $\delta_1>0$ such that if assumptions $(\romannumeral1)-(\romannumeral4)$ hold for $\delta\le \delta_1$, problem $\eqref{linear problem}$ admits a solution $\vp\in H^2(\Omega_T)$ and there exists a constant $\eta_1$, such that for any $T>0$ and $\eta\geq\eta_1$,
    \begin{align}
    &\sum_{|\alpha|\le 2}\eta  \|e^{-\eta z_0}D^
    \alpha \vp\|^2_{L^2(\Omega_T)}+e^{-2\eta T}\|D^
    \alpha \vp(T,\cdot)\|^2_{L^2(\Omega)}\nonumber\\
    &\quad\lesssim \frac{1}{\eta}\sum_{|\alpha|\le 1}\|e^{-\eta z_0}\mathcal{L}(D^\alpha\vp)\|^2_{L^2(\Omega_T)}
    +\|e^{-\eta z_0}f\|^2_{H^{1}(\Omega_T)}
+\|f|_{t=0}\|^2_{L^2(\Omega)}.\label{energy estimate}
\end{align}
\end{lem}
As mentioned previously, the $H^2$ solvability of the linearized problem is not sufficient to yield smooth solutions of the nonlinear problem by carrying out nonlinear iteration. Therefore, we have to deduce higher order \emph{a priori} estimate of the solution derived in lemma \ref{H2 solvability}. However, one can easily check that the regularity of the coefficients of the equation in the extended domain is not sufficient to establish the needed higher order estimates. Hence, we shall go back to the cornered-space domain to establish the higher order estimates, and the following lemma will be proved in the next section.


\begin{lem}\label{higher regularity}
There exists $\delta_2>0$ such that if assumptions $(\romannumeral1)-(\romannumeral4)$ hold for $\delta\le \delta_2$, then there exists a constant $\eta_2>1$ such that for any $T>0$ and $\eta\ge \eta_2$, the $H^2(\Omega_T)$ solution of problem $\eqref{linear problem}$ satisfies
    \begin{align}
    &\sum_{|\alpha|\le 4}\eta  \|e^{-\eta z_0}D^
    \alpha \vp\|^2_{L^2(\Omega_T)}+e^{-2\eta T}\|D^
    \alpha \vp(T,\cdot)\|^2_{L^2(\Omega)}\nonumber\\
    &\lesssim \frac{1}{\eta}\sum_{|\alpha|\le 3}\|e^{-\eta z_0}\mathcal{L}(D^\alpha\vp)\|^2_{L^2(\Omega_T)}
    +\|e^{-\eta z_0}f\|^2_{H^{3}(\Omega_T)}
    +\|f|_{t=0}\|^2_{H^2(\Omega)}.\label{energy estimate in lemma 4.1}
\end{align}
\end{lem}
\textbf{Proof of Proposition \ref{well-posedness}.}

 Combining lemma \ref{H2 solvability} and lemma \ref{higher regularity}, one can easily prove Proposition \ref{well-posedness}. In fact, by lemma \ref{H2 solvability} and lemma \ref{higher regularity}, it is easy to see that when $\delta<\delta_*\defs\min(\delta_1,\delta_2)$, problem $\eqref{linear problem}$ admits a smooth solution in $H^4(\Omega_T)$ and it satisfies the estimate given in proposition \ref{well-posedness} for $\eta\ge \eta_0\defs\max(\eta_1,\eta_2)$, where $(\delta_1,\eta_1)$ and $(\delta_2,\eta_2)$ are the constants obtained in lemma \ref{H2 solvability} and lemma \ref{higher regularity}, respectively.
 
  Hence, in order to prove Proposition \ref{well-posedness}, it suffices to show that lemma \ref{H2 solvability} and lemma \ref{higher regularity} hold. In this section, we will give a proof to lemma \ref{H2 solvability} and the proof of lemma \ref{higher regularity} is postponed to section \ref{sec:4}.
\subsection{The proof of Lemma \ref{H2 solvability}}
 
It is difficult to solve problem \eqref{linear problem} in the cornered-space domain $\Omega$ directly, so we introduce an extended problem first. Precisely, one extends $r_{20}$, $r_{02}$, $r_{12}$, $r_{21}$, and $b_2$ oddly with respect to $\{z_2=0\}$. Taking $r_{02}$ for example, we define
\begin{equation}
Er_{02}:=\left\{\!\!
\begin{array}{ll}
r_{02}(z_0,z_1,z_2), \mbox{ \!\ \ \ \ \ when } z_2>0,\\
-r_{02}(z_0,z_1,-z_2),\mbox{ when } z_2<0.
\end{array}
\right.
\end{equation}
Other coefficients and the right hand side term $f$ is extended evenly with respect to $\{z_2=0\}$. Thanks to assumptions $(\romannumeral2)$ and $(\romannumeral3)$, all the extended coefficients are still in $W^{1,\infty}(\tilde{\Omega}_T)$, where $\tilde{\Omega}_T:=[0,T]\times\mathbb{R}_+\times\mathbb{R}$.
For the notational simplicity, we omit the $``E"$ in all extended functions, the extended vertical boundary is still denoted by $\Gamma_1$ and the initial space domain in still denoted by $\Gamma_0$.
We try to obtain a solution to \eqref{linear problem} by solving the following initial boundary value problem:
\begin{equation}\label{extended linear problem}
\begin{cases}
\mathcal{L}\varphi=f, &\mbox{in\ } \tilde{\Omega}_T,\\
\mathcal{B}\varphi=0,&\mbox{on \ }\Gamma_1,\\
\varphi(0,z_1,z_2)=0,&\mbox{on \ }\Gamma_0,\\
\partial_{z_0}\varphi(0,z_1,z_2)=0,&\mbox{on \ }\Gamma_0.
\end{cases}
\end{equation}
\begin{proof} 
The proof of lemma \ref{H2 solvability} is divided into three steps. In the first two steps, we establish the energy estimate of the solutions to \eqref{extended linear problem} up to the second order. Then in the third step, we investigate a regularized problem associated to \eqref{extended linear problem} by mollifying its coefficients via the convolution with respect to $z_2$ (since the extended coefficients are non-smooth only in the $z_2$ direction) with the one dimensional classical Friedrichs mollifier $\rho_{\epsilon}$ and derives its uniform-in-$\epsilon$ estimate up to the second order. Then by applying the result in \cite{IKAWA1968} (or \cite{IKAWA1970}) to the regularized problem, one derives a unique solution $\vp^\epsilon$ to the regularized problem for each $\epsilon>0$. Owing to the uniform-in-$\epsilon$ second order estimate of $\vp^\epsilon$, one deduces an $H^2(\Omega_T)$-solution to the extended problem \eqref{extended linear problem} by taking the limit $\epsilon\rightarrow 0^+$ in the regularized problem and it still satisfies the second order estimate. Then by the properties of the extended coefficients and the uniqueness of the solution (the uniqueness is guaranteed by the second order energy inequality), one can show that the unique solution to the extended linear problem \eqref{extended linear problem} is actually a solution to the linear problem \eqref{linear problem}.

\vskip 0.2cm
\textbf{Step 1}. 
In this step, we will deduce the first order energy estimate of the solution to problem \eqref{extended linear problem}.
Multiplying $2e^{-2\eta z_0}\partial_{z_0}\varphi$ on both sides of equation in
$\eqref{extended linear problem}$, we have
\begin{align}
    2e^{-2\eta z_0}\mathcal{L}\varphi\mathcal{Q}\varphi=&\partial_i(e^{-2\eta z_0}r_{ij}\partial_j
    \varphi \partial_{0}\varphi)+\partial_{j}(e^{-2\eta z_0}r_{ij}\partial_i\varphi \partial_0\varphi)+e^{-2\eta z_0}P(D\varphi)\nonumber\\
    &-\partial_{0}(e^{-2\eta z_0}r_{ij}\partial_i\varphi\partial_j\varphi)
    +2\eta e^{-2\eta z_0}(2r_{0j}\partial_j\varphi \partial_0\varphi-r_{ij}\partial_i\varphi\partial_j\varphi),\nonumber
\end{align}
where $P(D\varphi)$ is a quadratic polynomial with respect to $D\varphi$. Clearly we have $P(D\varphi)\le C|D\varphi|^2$.
Integrating the identity over $\tilde{\Omega}_T$ with $z:=(z_0,z_1,z_2)$, we obtain
\begin{align}
    \int_{\tilde{\Omega}_T}\!\!2e^{-2\eta z_0}\mathcal{L}\varphi\mathcal{Q}\varphi dz\!=&\!\left[\int_{\tilde{\Omega}}e^{-2\eta z_0}H_0dz_1dz_2\right]_{z_0=0}^{z_0=T}-\int_0^T\!\!\!\int_{\mathbb{R}}e^{-2\eta z_0}H_1|_{z_1=0}dz_2dz_0\nonumber\\
    &+\int_{\tilde{\Omega}_T}\!\!e^{-2\eta z_0}P(D\varphi)dz+2\eta\int_{\tilde{\Omega}_T} e^{-2\eta z_0}H_0dz,
\end{align}
where for $\bm{\xi}=(\xi_0,\xi_1,
\xi_2)\in\mathbb{R}^3$, we define
\begin{align}
    H_0(\bm{\xi})&\defs 2\sum_{i,k=0}^2r_{i0}\xi_i Q_k\xi_k-Q_0\sum_{i,j=0}^2r_{ij}\xi_i\xi_j,\label{defn H0}\\
     H_1(\bm{\xi})&\defs 2\sum_{i,k=0}^2r_{i1}\xi_i Q_k\xi_k-Q_1\sum_{i,j=0}^2r_{ij}\xi_i\xi_j.\label{defn H1}
\end{align}
At the background state, one has
\begin{align}
    H_0(D\vp)&=2\partial_0\varphi \sum_{k=0}^2Q_k\partial_k\varphi-Q_0(r_{11}|\partial_1\varphi|^2+r_{22}|\partial_2\varphi|^2+|\partial_0\varphi|^2)\nonumber\\
    &=(D\varphi) \mathbf{M} (D\varphi)^\top,\label{quadratic of H0}
\end{align}
where $$\mathbf{M}=\begin{pmatrix}
    Q_0& Q_1& Q_2\\
    Q_1&-r_{11}Q_0&0\\
    Q_2&0&-r_{22}Q_0
\end{pmatrix}.$$
Select $(Q_0,Q_1,Q_2)$ properly such that
\begin{equation}\label{determinant cond}
\left\{\!\!\!\!\!\!
   \begin{array}{ll}
     &Q_0>0,\\
     &-Q_0^2r_{11}-Q_1^2>0,\\
     &r_{22}Q_1^2Q_0+Q_0^3r_{11}r_{22}+r_{11}Q_0Q_2^2>0,
\end{array}
\right.
\end{equation}
\emph{i.e.}, such that $\mathbf{M}$ is positive definite.
In view of assumption $(\romannumeral1)$, we just need to let
\begin{equation}\label{Q0 selection}
    \left\{\!\!
    \begin{array}{ll}
    Q_0>0,\vspace{0.5mm}\\
    Q_0^2>\dfrac{Q_1^2}{-r_{11}},\vspace{2mm}\\
    Q_0^2>\dfrac{-Q_1^2r_{22}-Q_2^2r_{11}}{r_{11}r_{22}}.
\end{array}
    \right.
\end{equation}
It is easy to see from \eqref{Q0 selection} that $Q_1$ and $Q_2$ can be arbitrary.
Then $H_0\ge C_1 |D\varphi|^2$ for some positive constant $C_1$. 
 On the other hand, it is easy to see $H_0\le C_2|D\varphi|^2$ due to assumption $(\romannumeral1)$. By assumption $(\romannumeral4)$, on the vertical boundary $\Gamma_1$, we have
 \begin{align}
    H_1&=2\sum_{k=0}^2Q_k\partial_k\varphi(r_{11}\partial_1\varphi+r_{12}\partial_2\varphi)-Q_1\sum_{i,j=0}^2r_{ij}\partial_i\varphi\partial_j\varphi\nonumber\\
    &=2\sum_{k=0}^2Q_k\partial_k\varphi (\frac{r_{11}}{b_1}\mathcal{B}\varphi)-Q_1\sum_{i,j=0}^2r_{ij}\partial_i\varphi\partial_j\varphi\nonumber\\
    &=0.
\end{align}
In the last equality above, we used the condition $\mathcal{B}\varphi=0\mbox{\ on\ }\Gamma_1$ and set $Q_1=0$. 
Hence we have
\begin{align}
    &\eta\int_{\tilde{\Omega}_T}e^{-2\eta z_0}|D\varphi|^2dz+\int_{\tilde{\Omega}}e^{-2\eta T}|D\varphi|^2dz_1dz_2\nonumber\\
    &\le C\int_{\tilde{\Omega}_T}e^{-2\eta z_0}((q\eta +1)|D\varphi|^2+\frac{1}{q \eta}|\mathcal{L}\varphi|^2)dz.
\end{align}
Let $q=1/(2C)$ and $\eta\ge 4$, then we obtain
\begin{align}\label{first order estimate: 2nd step}
    &\eta\int_{\tilde{\Omega}_T}e^{-2\eta z_0}|D\varphi|^2dz+\int_{\tilde{\Omega}}e^{-2\eta T}|D\varphi|^2dz_1dz_2
    \le C\frac{1}{ \eta}\int_{\tilde{\Omega}_T}e^{-2\eta z_0}|\mathcal{L}\varphi|^2dz.
\end{align}
\vskip 0.3cm
\textbf{Step 2.} 
In this step, we will establish the second order estimate of the solutions based on the first order estimate derived in step 1. 
Clearly, $\partial_{z_0}\varphi$ satisfies
\begin{equation}\label{extended partialz0}
\begin{cases}
\mathcal{L}(\partial_{z_0}\varphi)=-[\partial_{z_0},\mathcal{L}]\varphi+\partial_{z_0}f, &\mbox{in\ } \tilde{\Omega}_T,\\
\mathcal{B}(\partial_{z_0}\varphi)=0,&\mbox{on \ }\Gamma_1,\\
\partial_{z_0}\varphi(z_0,z_1,z_2)=0,&\mbox{on \ }\Gamma_0,\\
\partial^2_{z_0}\varphi(z_0,z_1,z_2)=F|_{z_0=0},&\mbox{on \ }\Gamma_0.
\end{cases}
\end{equation}
where $$
F=f-\sum_{(i,j)\neq (0,0)}r_{ij}\partial_{ij}\varphi
-\sum_{i=0}^2r_i\partial_i\varphi=f.$$
It is easy to see that $\|F|_{z_0=0}\|^2_{L^2(\tilde{\Omega})}=
\|f|_{z_0=0}\|_{L^2(\tilde{\Omega})}$.
By the same argument as done in the first step, we deduce that $\partial_{z_0}\vp$ satisfies 
\begin{align}\label{first order estimate of partialz0}
    &\eta\|e^{-\eta z_0}D\partial_{z_0}\varphi\|^2_{L^2(\tilde{\Omega}_T)}+e^{-2\eta T}\|D\partial_{z_0}\vp(T,\cdot)\|^2_{L^2(\tilde{\Omega})}\nonumber\\
    &\qquad\le C(\frac{1}{\eta}\|\mathcal{L}(\partial_{z_0}\varphi)\|^2_{L^2(\tilde{\Omega}_T)}+\|f|_{z_0=0}\|^2_{L^2(\tilde{\Omega})}).
\end{align}

Then we proceed to estimate $\partial_{z_2}\vp$.
It is clear that $\partial_{z_2}\vp$ satisfies
\begin{equation}\label{extended partialz2}
\begin{cases}
\mathcal{L}(\partial_{z_2}\varphi)=-[\partial_{z_2},\mathcal{L}]\varphi+\partial_{z_2}f, &\mbox{in\ } \tilde{\Omega}_T,\\
\mathcal{B}(\partial_{z_2}\varphi)=-[\partial_{z_2},\mathcal{B}]\vp,&\mbox{on \ }\Gamma_1,\\
\partial_{z_2}\varphi=0,
&\mbox{on \ }\Gamma_0,\\
\partial_{z_0}(\partial_{z_2}\varphi)=0
&,\mbox{on \ }\Gamma_0.
\end{cases}
\end{equation}

Multiplying $2e^{-2\eta z_0}\partial_{z_0z_2}\vp$ on both sides of $\eqref{extended partialz2}_1$ and integrating by parts over $\Omega_T$, we have
\begin{align}
    \int_{\tilde{\Omega}_T}^{-2\eta z_0}\mathcal{L}(\partial_{z_2}\vp)\partial_{z_0}\partial_{z_2}\vp dz=&-2\int_0^T\!\!\!\int_{\mathbb{R}}e^{-2\eta z_0}\partial_{z_0z_2}\vp \frac{r_{11}}{b_1}\mathcal{B}(\partial_{z_2}\vp) dz_2dz_0|_{z_1=0}\nonumber\\
    &+2\eta\int_{\tilde{\Omega}_T} e^{-2\eta z_0}H_0dz+\left[\int_{\tilde{\Omega}}e^{-2\eta t}H_0dz_1dz_2\right]_{z_0=0}^{z_0=T}\nonumber\\
    &+\int_{\tilde{\Omega}_T}e^{-2\eta z_0}P_1(D\partial_{z_2}\vp)dz\label{eq:integration by parts 2nd order},
\end{align}
where $H_0=|\partial_{z_0}\partial_{z_2}\vp|^2-r_{11}|\partial_{z_1}\partial_{z_2}\vp|^2-r_{22}|\partial_{z_2}\partial_{z_2}\vp|^2$
and $P_1$ is a new quadratic polynomial with respect to $D\partial_{z_2}\vp$. Since the coefficients are in $W^{1,\infty}(\tilde{\Omega}_T)$, it is easy to see that $|P_1(D\partial_{z_2}\vp)|\le C|D\partial_{z_2}\vp|$.

In order to complete the estimate, we need to deal with the boundary term carefully. Firstly, with the help of the boundary condition $\mathcal{B}\vp=0$ on $\Gamma_1$, one has $\mathcal{B}\partial_{z_2}\vp=-(\partial_{z_2}b_1\partial_{z_1}+\partial_{z_2}b_2\partial_{z_2})\vp$. 
Then for $i=1,2$, by assumptions $(\romannumeral1)$, $(\romannumeral3)$, and $(\romannumeral4)$, and the Gauss theorem
\begin{align}
  &2\int_0^T\!\!\int_{\mathbb{R}}e^{-2\eta z_0}(\frac{r_{11}}{b_1}\partial_{z_2}b_i\partial_{z_i}\vp\cdot\partial_{z_0z_2}\vp)|_{z_1=0}dz_2dz_0\nonumber\\
  &\quad= -2\int_{\tilde{\Omega}_T} \partial_{z_1}(e^{-2\eta z_0}\frac{r_{11}}{b_1}\partial_{z_2}b_i\partial_{z_i}\vp\partial_{z_0z_2}\vp )dz\nonumber\\
  &\quad=-2\int_{\tilde{\Omega}_T}\!\!e^{-2\eta z_0}(\partial_{z_1}(\frac{r_{11}}{b_1}\partial_{z_2}b_i)\partial_{z_i}\vp\partial_{z_0z_2}\vp+\frac{r_{11}}{b_1}\partial_{z_2}b_i(\partial_{z_1z_i}\vp\partial_{z_0z_2}\vp+\partial_{z_i}\vp\partial_{z_0z_1z_2}\vp))\nonumber\\
  &\quad\lesssim \int_0^T e^{-2\eta z_0}(\|\partial_{z_0}\vp\|^2_{H^1(\tilde{\Omega})}+\|\vp\|^2_{H^2(\tilde{\Omega})})dz_0+\int_{\tilde{\Omega}_T}e^{-2\eta z_0}\frac{r_{11}}{b_1}\partial_{z_2}b_i\partial_{z_i}\vp\partial_{z_0z_1z_2}\vp dz\nonumber\\
  &\quad\le   C\int_0^T \!\!e^{-2\eta z_0}(\|\partial_{z_0}\vp\|^2_{H^1(\tilde{\Omega})}+\|\vp\|^2_{H^2(\tilde{\Omega})})dz_0+\!\int_{\tilde{\Omega}_T}\!\!\partial_{z_0}(e^{-2\eta z_0}\frac{r_{22}}{b_1}\partial_{z_2}b_i\partial_{z_i}\vp\partial_{z_1z_2}\vp)dz\nonumber\\
  &\quad\quad+2\eta \int_{\tilde{\Omega}_T}e^{-2\eta z_0}\frac{r_{22}}{b_1}\partial_{z_2}b_i\partial_{z_i}\vp\partial_{z_1z_2}\vp dz-\int_{\tilde{\Omega}_T}e^{-2\eta z_0}\frac{r_{22}}{b_1}\partial_{z_2}b_i\partial_{z_0z_i}\vp\partial_{z_1z_2}\vp dz\nonumber\\
  &\quad\quad-\int_{\tilde{\Omega}_T}e^{-2\eta z_0}\partial_{z_0}(\frac{r_{22}}{b_1}\partial_{z_2}b_i)\partial_{z_i}\vp\partial_{z_1z_2}\vp dz\nonumber\\
  &\quad\lesssim \int_0^T e^{-2\eta z_0}(\|\partial_{z_0}\vp\|^2_{H^1(\tilde{\Omega})}+\|\vp\|^2_{H^2(\tilde{\Omega})})dz_0\nonumber\\
  &\quad\quad+\delta (\eta \int_0^T e^{-2\eta z_0}\|\vp\|^2_{H^2(\tilde{\Omega})} dz_0+ e^{-2\eta T}\|\vp|_{z_0=T}\|^2_{H^2(\tilde{\Omega})}).\label{eq:estimate of boundary term 2nd order}
\end{align}
It follows from assumptions $(\romannumeral1)$ and $(\romannumeral4)$ that 
$H_0\ge C|D\partial_{z_2}\vp|^2$.
Then in view of $\eqref{eq:integration by parts 2nd order}$ and $\eqref{eq:estimate of boundary term 2nd order}$, by the Cauchy inequality, we deduce that
\begin{align}
    \eta\int_{\tilde{\Omega}_T}&e^{-2\eta z_0}|D\partial_{z_2}\vp|^2dz+e^{-2\eta T}\int_{\tilde{\Omega}}|D\partial_{z_2}\vp|_{z_0=T}|^2dz_1dz_2\nonumber\\
    &\le C\left(\int_{\tilde{\Omega}_T}e^{-2\eta z_0}((\epsilon\eta+1)|D\partial_{z_2}\vp|^2+\frac{1}{\epsilon\eta}|\mathcal{L}\partial_{z_2}\vp|^2)dz+\|\vp_0\|^2_{H^2(\tilde{\Omega})}+\|\vp_1\|^2_{H^1(\tilde{\Omega})}\right.\nonumber\\
    &\quad+\int_0^T e^{-2\eta z_0}(\|\partial_{z_0}\vp\|^2_{H^1(\tilde{\Omega})}+\|\vp\|^2_{H^2(\tilde{\Omega})})dz_0\nonumber\\
  &\quad+\delta (\eta \int_0^T e^{-2\eta z_0}\|\vp\|^2_{H^2(\tilde{\Omega})} dz_0+ e^{-2\eta T}\|\vp|_{z_0=T}\|^2_{H^2(\tilde{\Omega})}).\label{eq:estimate of Dpartial_z2_phi}
\end{align}
Finally, it follows from the second order equation $\eqref{extended linear problem}_1$ that
\[
\partial_{z_1z_1}\vp=\frac{1}{r_{11}}\left(\mathcal{L}\vp-\sum_{(i,j)\neq(1,1)}r_{ij}\partial_{ij}\vp\right).
\]
Therefore, by \eqref{first order estimate: 2nd step}, \eqref{first order estimate of partialz0} and \eqref{eq:estimate of Dpartial_z2_phi},
\begin{align}
  &\eta\int_{\tilde{\Omega}_T}e^{-2\eta z_0}|\partial_{z_1z_1}\vp|^2dz+e^{-2\eta T}\int_{\tilde{\Omega}}|\partial_{z_1z_1}\vp|_{z_0=T}|^2dz_1dz_2 \nonumber\\
  &\lesssim \sum_{i=0,2}\left(\eta\int_{\tilde{\Omega}_T}e^{-2\eta z_0}|D\partial_{z_i}\vp|^2dz+e^{-2\eta T}\int_{\tilde{\Omega}}|D\partial_{z_i}\vp|_{z_0=T}|^2dz_1dz_2\right)\nonumber\\
  &\quad+ \eta\int_{\tilde{\Omega}_T}e^{-2\eta z_0}|D\vp|^2dz+e^{-2\eta T}\int_{\tilde{\Omega}}|D\vp|_{z_0=T}|^2dz_1dz_2\nonumber\\
  &\quad+\eta\int_{\tilde{\Omega}_T}e^{-2\eta z_0}|\mathcal{L}\vp|^2dz+e^{-2\eta T}\int_{\tilde{\Omega}}|\mathcal{L}\vp|_{z_0=T}|^2dz_1dz_2\nonumber\\
  &\lesssim (\frac{1}{\eta}\|e^{-\eta z_0}\mathcal{L}\varphi\|^2_{L^2(\tilde{\Omega}_T)}+\frac{1}{\eta}\|e^{-\eta z_0}\mathcal{L}(\partial_{z_0}\varphi)\|^2_{L^2(\tilde{\Omega}_T)}+\|f|_{z_0=0}\|^2_{L^2(\tilde{\Omega})})\nonumber\\
  &\quad+\int_{\tilde{\Omega}_T}e^{-2\eta z_0}((q\eta+1)|D\partial_{z_2}\vp|^2+\frac{1}{q\eta}|\mathcal{L}\partial_{z_2}\vp|^2)dz\nonumber\\
    &\quad+\int_0^T e^{-2\eta z_0}(\|\partial_{z_0}\vp\|^2_{H^1(\tilde{\Omega})}+\|\vp\|^2_{H^2(\tilde{\Omega})})dz_0+\delta \eta \int_0^T e^{-2\eta z_0}\|\vp\|^2_{H^2(\tilde{\Omega})} dz_0\nonumber\\
     &\quad+\delta e^{-2\eta T}\|\vp|_{z_0=T}\|^2_{H^2(\tilde{\Omega})}+\eta\int_{\tilde{\Omega}_T}\!\!e^{-2\eta z_0}|\mathcal{L}\vp|^2dz+e^{-2\eta T\!\!}\int_{\tilde{\Omega}}|\mathcal{L}\vp|_{z_0=T}|^2dz_1dz_2.\label{eq:estimate of partial_z1z1}
\end{align}

In order to control the last two terms in \eqref{eq:estimate of partial_z1z1}, by conducting integration by parts with respect to $z_0$ to the integral $\int_{\tilde{\Omega}_T}e^{-2\eta z_0}|v|^2dz$, we introduce the following inequality
\begin{align}
    &\eta\int_{\tilde{\Omega}_T}e^{-2\eta z_0}|v|^2dz+e^{-2\eta T}\int_{\tilde{\Omega}}|v|_{z_0=T}|^2dz_1dz_2\nonumber\\
    &\quad\le \frac{1}{\eta}\int_{\tilde{\Omega}_T}e^{-2\eta z_0}|\partial_{z_0}v|^2dz+\|v|_{z_0=0}\|^2_{L^2(\tilde{\Omega})}.\label{eq:low norm bounded by higher norm}
\end{align}
Therefore,
\begin{align}
    &\eta\int_{\tilde{\Omega}_T}e^{-2\eta z_0}|\mathcal{L}\vp|^2dz+e^{-2\eta T}\int_{\tilde{\Omega}}|\mathcal{L}\vp|_{z_0=T}|^2dz_1dz_2\nonumber\\
    &\le\frac{1}{\eta}\int_{\tilde{\Omega}_T}e^{-2\eta z_0}|\partial_{z_0}(\mathcal{L}\vp)|^2dz+\|f|_{z_0=0}\|^2_{L^2(\tilde{\Omega})}\nonumber\\
    &\lesssim \frac{1}{\eta}\int_{\tilde{\Omega}_T}e^{-2\eta z_0}|\mathcal{L}(\partial_{z_0}\vp)|^2dz+\frac{1}{\eta}\sum_{|\alpha|\le 2}\|e^{-\eta z_0}D^\alpha \vp\|^2_{L^2(\tilde{\Omega}_T)}+\|f|_{z_0=0}\|^2_{L^2(\tilde{\Omega})}.\label{eq:estimate of L(vp)}
\end{align}
Then \eqref{eq:estimate of partial_z1z1} and \eqref{eq:estimate of L(vp)} imply
\begin{align}
   &\eta\int_{\tilde{\Omega}_T}e^{-2\eta z_0}|\partial_{z_1z_1}\vp|^2dz+e^{-2\eta T}\int_{\tilde{\Omega}}|\partial_{z_1z_1}\vp|_{z_0=T}|^2dz_1dz_2 \nonumber\\
   &\lesssim \frac{1}{\eta}\int_{\tilde{\Omega}_T}e^{-2\eta z_0}\sum_{|\alpha|\le 1}|\mathcal{L}(D^\alpha\vp)|^2 dz+\int_{\tilde{\Omega}_T}e^{-2\eta z_0}((q\eta+1)|D\partial_{z_2}\vp|^2+\frac{1}{q\eta}|\mathcal{L}\partial_{z_2}\vp|^2)dz\nonumber\\
    &\quad+\int_0^T e^{-2\eta z_0}(C\delta\|\vp\|^2_{H^2(\tilde{\Omega})}+\|\partial_{z_0}\vp\|^2_{H^1(\tilde{\Omega})})dz_0
    +\frac{1}{\eta}\sum_{|\alpha|\le 2}\|e^{-\eta z_0}D^\alpha \vp\|^2_{L^2(\tilde{\Omega}_T)}\nonumber\\
    &\quad+\|f|_{z_0=0}\|^2_{L^2(\tilde{\Omega})}.\label{eq:estimate of partial_z1z1_vp final}
\end{align}
Since $D\partial_{z_0}\vp$, $D\partial_{z_2}\vp$ and $\partial_{z_1z_1}\vp$ cover all second order derivatives, by adding \eqref{first order estimate: 2nd step}, \eqref{eq:estimate of Dpartial_z2_phi} and \eqref{eq:estimate of partial_z1z1_vp final} together, letting $q>0$, $\delta>0$ and $\frac{1}{\eta}$ be properly small, we deduce
\begin{align}
	 &\sum_{|\alpha|\le 2}\eta\|e^{-\eta z_0}D^\alpha\vp\|^2_{L^2(\tilde{\Omega}_T)}+e^{-2\eta T}\|D^\alpha\vp|_{z_0=T}\|^2_{L^2(\tilde{\Omega})}\nonumber\\
	&\quad\lesssim \frac{1}{\eta}\sum_{|\alpha|\le 1}\|e^{-\eta z_0}\mathcal{L}(D^\alpha \vp)\|^2_{L^2(\tilde{\Omega}_T)}+\|\mathcal{L}\vp|_{z_0=0}\|^2_{L^2(\tilde{\Omega})}.\label{eq:second order estimate_first}
\end{align}
Since the coefficients of $\mathcal{L}$ are bounded, by \eqref{eq:second order estimate_first}, it holds for $\eta$ large enough that
\begin{align}
    &\sum_{|\alpha|\le 2}\eta\|e^{-\eta z_0}D^\alpha\vp\|^2_{L^2(\tilde{\Omega}_T)}+e^{-2\eta T}\|D^\alpha \vp|_{z_0=T}\|^2_{L^2(\tilde{\Omega})}\nonumber\\
    &\lesssim \frac{1}{\eta}\sum_{|\alpha|\le 1}\|e^{-\eta z_0}D^\alpha f\|^2_{L^2(\tilde{\Omega}_T)}+\|f|_{z_0=0}\|^2_{L^2(\tilde{\Omega})}
.\label{eq:second order estimate}
\end{align}
\vskip 1em
\textbf{Step 3.}
In this step, we will apply Theorem 1 in \cite{IKAWA1968} (or Theorem 1 in \cite{IKAWA1970}) to derive the existence of an $H^2(\tilde{\Omega}_T)$-solution to the extended problem, then by the property of the extended coefficients, one shows that the solution is indeed a solution to problem \eqref{linear problem}.
In order to apply  \cite[Theorem 1]{IKAWA1968} (or \cite[Theorem 1]{IKAWA1970}), we consider a regularized problem associated to \eqref{extended linear problem}.
For $\epsilon>0$, let $\rho_{\epsilon}$ be the one dimensional Friedrichs mollifier, i.e.,
$\rho_\epsilon(s)=\frac{1}{\epsilon}\eta\left(\frac{s}{\epsilon}\right)$, where
\begin{equation}
\eta(s)=
\left\{
\begin{array}{ll}
\!\!K\cdot\exp(-\frac{1}{1-|s|^2}),{\ }|s|<1,\\
\!\!0,{\ } |s|\ge 1,
\end{array}
\right.
\end{equation}
such that $\int_{\mathbb{R}}\eta(s)ds=1$.
Define $\mathcal{L}^\epsilon$ and $\mathcal{B}^\epsilon$ as
\begin{align}
	\mathcal{L}^\epsilon\defs \sum_{i,j=0}^2\tilde{r}_{ij}^\epsilon\partial_{ij},\quad
	\mathcal{B}^\epsilon\defs\sum_{i=1}^2\tilde{b}_i^\epsilon\partial_i,
\end{align}
where
\[
\tilde{r}_{ij}^\epsilon(z_0,z_1,z_2)=\left(\frac{r_{ij}}{r_{11}}\right)^\epsilon(z_0,z_1,z_2)\defs\int_\mathbb{R}\left(\frac{r_{ij}}{r_{11}}\right)(z_0,z_1,s)\rho_{\epsilon}(z_2-s)ds)
\]
and
\[
\tilde{b}_i^\epsilon(z_1,z_2)=\left(\frac{b_{i}}{b_{1}}\right)^\epsilon(z_1,z_2)\defs\int_\mathbb{R}\left(\frac{b_{i}}{b_1}\right)(z_1,s)\rho_{\epsilon}(z_2-s)ds.
\]
Before going on, we present following lemma, which gives the properties of the coefficients of the regularized problem.
\begin{lem}\label{lem:3.1}
Under assumptions $(\romannumeral1)$-$(\romannumeral4)$, we have:
\begin{enumerate}[label=(\arabic*).]
\item $\tilde{r}_{12}^\epsilon$, $\tilde{r}_{21}^\epsilon$, $\tilde{r}_{02}^\epsilon$,  $\tilde{r}_{20}^\epsilon$ and $\tilde{b}_2^\epsilon(0,z_2)$ are odd functions with respect to $z_2$, and the other coefficients are even functions with respect to $z_2$.
\item There exists a positive contant $C$, independent on $\epsilon$, such that $$\left\|\tilde{r}_{ij}^\epsilon\right\|_{L^\infty(\tilde{\Omega}_T)}\le C\delta+\left|\frac{\bar{r}_{ij}}{\bar{r}_{11}}\right|,\quad\left\|D\tilde{r}_{ij}^\epsilon\right\|_{L^\infty(\tilde{\Omega}_T)}\le C\delta,$$
and
$$
\quad\left\|\tilde{b}_i^\epsilon\right\|_{W^{1,\infty}(\tilde{\Omega})}\le C\delta,\quad \left\|\partial_{z_2}^2\tilde{b}_2^\epsilon(0,\cdot)\right\|_{L^\infty}\le C\delta.$$

\item $\tilde{b}_2^\epsilon(0,0)=0$ and $\partial_{z_2}\tilde{b}_2^\epsilon(0,0)=0$.
\item $\mathcal{B}^\epsilon$ is co-normal to $\mathcal{L}^\epsilon$, i.e.,
 $$\left(\frac{r_{10}}{r_{11}}\right)^\epsilon|_{z_1=0}=0\quad\mbox{and }\quad\left(\frac{r_{12}}{r_{11}}\right)^\epsilon|_{z_1=0}=\left(\frac{b_{2}}{b_{1}}\right)^\epsilon|_{z_1=0}.$$
\end{enumerate}
\end{lem}
With this lemma, it is not difficult to show lemma \ref{H2 solvability} and the proof of this lemma is delayed to the end of this subsection. 

Now we consider
\begin{equation}\label{regularized extended linear problem}
\begin{cases}
\mathcal{L}^\epsilon\varphi^\epsilon=f^\epsilon,&\mbox{in\ } \tilde{\Omega}_T,\\
\mathcal{B}^\epsilon\varphi^\epsilon=0,&\mbox{on \ }\Gamma_1,\\
\varphi^\epsilon(0,z_1,z_2)=0,&\mbox{on \ }\Gamma_0,\\
\partial_{z_0}\varphi^\epsilon(0,z_1,z_2)=0,&\mbox{on \ }\Gamma_0,
\end{cases}
\end{equation}
where
$$
f^\epsilon(z_0,z_1,z_2)=\left(\frac{f}{r_{11}}\right)^\epsilon(z_0,z_1,z_2)\defs\int_{\mathbb{R}}\left(\frac{f}{r_{11}}\right)(z_0,z_1,s)\rho_\epsilon(z_2-s) ds.$$

Armed with lemma \ref{lem:3.1},
one can immediately obtain the uniform-in-$\epsilon$ second order estimate of $\vp^\epsilon$ by repeating the process in the first two steps. In fact, one has
\begin{align}
&\sum_{|\alpha|\le 2}\eta\|e^{-\eta z_0}D^\alpha\vp^\epsilon\|^2_{L^2(\tilde{\Omega}_T)}+e^{-2\eta T}\|D^\alpha\vp^\epsilon|_{z_0=T}\|^2_{L^2(\tilde{\Omega})}\nonumber\\
&\le C \frac{1}{\eta}\sum_{|\alpha|\le 1}\|e^{-\eta z_0}D^
\alpha f^\epsilon\|^2_{L^2(\tilde{\Omega}_T)}+\|f^\epsilon|_{z_0=0}\|^2_{L^2(\tilde{\Omega})}.
\end{align}
Clearly one has
\begin{align}
	&\left\|\left(\frac{\mathcal{F}}{r_{11}}\right)^\epsilon(z_0,\cdot)\right\|^2_{L^2(\tilde{\Omega})}\nonumber\\
	&\quad=\int_{\tilde{\Omega}}\left|\left(\frac{\mathcal{F}}{r_{11}}\right)^\epsilon(z_0,z_1,z_2)\right|^2 dz_1dz_2\nonumber\\
	&\quad=\int_{\tilde{\Omega}}\left|\int_{\mathbb{R}}\frac{\mathcal{F}}{r_{11}}(z_0,z_1,s)\rho_{\epsilon}(z_2-s)ds\right|^2 dz_1dz_2\nonumber\\
	&\quad\le \int_{\tilde{\Omega}}\left(\int_{\mathbb{R}}\left|\frac{\mathcal{F}}{r_{11}}\right|^2(z_0,z_1,s)\rho_{\epsilon}(z_2-s)ds\right)\left(\int_{\mathbb{R}}\rho_{\epsilon}(z_2-s)ds\right) dz_1dz_2\nonumber\\
	&\quad\le \left\|\frac{1}{r_{11}}\right\|_{L^\infty(\tilde{\Omega}_T)}^2\cdot\int_{\tilde{\Omega}}\left[\left(\int_{\mathbb{R}}\rho_{\epsilon}(z_2-s)dz_2\right)|\mathcal{F}|^2(z_0,z_1,s)dz_1\right]ds\nonumber\\
	&\quad\le C\int_{\tilde{\Omega}}|\mathcal{F}|^2(z_0,z_1,s)dz_1ds=C\|\mathcal{F}(z_0,\cdot)\|^2_{L^2(\tilde{\Omega})},
\end{align}
where the constant $C$ depends on $\rho_0$ and $\gamma$, but not on $\epsilon$. Similarly one has 
\begin{align}
	\left\|D\left(\frac{\mathcal{F}}{r_{11}}\right)^\epsilon(z_0,\cdot)\right\|_{L^2(\tilde{\Omega})}\le C\left(\left\|\mathcal{F}(z_0,\cdot)\right\|^2_{L^2(\tilde{\Omega})}+\|D\mathcal{F}(z_0,\cdot)\|^2_{L^2(\tilde{\Omega})}\right).
\end{align}
Combining the above estimates and lemma \ref{lem:3.1}, we deduce that
\begin{align}
&\sum_{|\alpha|\le 2}\eta\|e^{-\eta z_0}D^\alpha\vp^\epsilon\|^2_{L^2(\tilde{\Omega}_T)}+e^{-2\eta T}\|D^\alpha \vp^\epsilon|_{z_0=T}\|^2_{L^2(\tilde{\Omega})}\nonumber\\
&\quad\le C \frac{1}{\eta}\sum_{|\alpha|\le 1}\|e^{-\eta z_0}D^
\alpha f\|^2_{L^2(\tilde{\Omega}_T)}+\|f|_{z_0=0}\|^2_{L^2(\tilde{\Omega})}.\label{eq3.37}
\end{align}
By \cite[Theorem 1]{IKAWA1968} (or \cite[Theorem 1]{IKAWA1970}), lemma \ref{lem:3.1}, and inequality \eqref{eq3.37}, one concludes that for each $\epsilon>0$, there exists a solution $\vp^\epsilon$ to problem \eqref{regularized extended linear problem} satisfying the uniform estimate \eqref{eq3.37}. Hence there exists a subsequence $\{\vp^{\epsilon_j}\}_{\epsilon_j> 0}$ converging to a function $\vp$ weakly in $H^2(\tilde{\Omega}_T)$. By lemma \ref{lem:3.1} and the uniform estimate \eqref{eq3.37}, one can pass the limit $\epsilon\rightarrow 0^+$ in problem \eqref{regularized extended linear problem}, which implies $\vp$ solves problem \eqref{extended linear problem} in the weak sense and it also satisfies estimate \eqref{eq3.37}. By our extension, it is not difficult to see that $\vp(z_0,z_1,-z_2)$ is also a solution to problem \eqref{extended linear problem}.
It follows from the uniqueness of the extended problem \eqref{extended linear problem} that 
\[
\vp(z_0,z_1,z_2)=\vp(z_0,z_1,-z_2)
\] 
for all $(z_0,z_1,z_2)\in \tilde{\Omega}_T$.
Differentiating on both sides of the above identity then letting $z_2=0$, one deduces that $\partial_{z_2}\vp(z_0.z_1,0)=0$. This reveals that $\vp$ is indeed a solution to problem \eqref{linear problem}.
\end{proof}
\textbf{Proof of lemma \ref{lem:3.1}.}
\begin{proof}
$(1)$ They are true due to the constructions of $\tilde{r}_{ij}^\epsilon$'s and $\tilde{b}_i^\epsilon$ and the property of the mollifier. For example, for $\tilde{r}_{12}^\epsilon$, one has
\begin{align}
	\tilde{r}_{12}^\epsilon(z_0,z_1,-z_2)&=\int_{\mathbb{R}}\frac{r_{12}}{r_{11}}(z_0,z_1,-z_2-\tau)\rho_{\epsilon}(\tau)d\tau\nonumber\\
	&=-\int_{\mathbb{R}}\frac{r_{12}}{r_{11}}(z_0,z_1,z_2+\tau)\rho_{\epsilon}(-\tau)d\tau\nonumber\\
	&=-\int_{\mathbb{R}}\frac{r_{12}}{r_{11}}(z_0,z_1,s)\rho_{\epsilon}(z_2-s)ds\nonumber\\
	&=-\tilde{r}_{12}^\epsilon(z_0,z_1,z_2),
\end{align}
where in the second equality, we have used the oddness of $r_{12}$, and the evenness of  $r_{11}$ and $\rho_{\epsilon}$, and in the last equality, the changing of variable is used. The properties of the other coefficients can be derived by similar arguments. Since the argument is similar and standard, we omit the details here.
\vskip 0.2cm
$(2)$ By assumption $(\romannumeral2)$, properties $(1)$ from the modified coefficients, and the properties of the mollifier, one has
\begin{align*}
\|\tilde{r}_{ij}^\epsilon\|_{L^\infty(\tilde{\Omega}_T)}=\left\|\left(\frac{r_{ij}}{r_{11}}\right)^\epsilon\right\|_{L^\infty(\tilde{\Omega}_T)}\!\!&\le \left\|\left(\frac{r_{ij}}{r_{11}}\right)^\epsilon-\frac{\bar{r}_{ij}}{\bar{r}_{11}}\right\|_{L^\infty(\tilde{\Omega}_T)}+\left|\frac{\bar{r}_{ij}}{\bar{r}_{11}}\right|\nonumber\\
&= \left\|\left(\frac{r_{ij}}{r_{11}}-\frac{\bar{r}_{ij}}{\bar{r}_{11}}\right)^\epsilon\right\|_{L^\infty(\tilde{\Omega}_T)}+\left|\frac{\bar{r}_{ij}}{\bar{r}_{11}}\right|.
\end{align*}
It is clear that
\begin{align}
\left\|\left(\frac{r_{ij}}{r_{11}}-\frac{\bar{r}_{ij}}{\bar{r}_{11}}\right)^\epsilon\right\|_{L^\infty(\tilde{\Omega}_T)}&\le \left\|\frac{r_{ij}}{r_{11}}-\frac{\bar{r}_{ij}}{\bar{r}_{11}}\right\|_{L^\infty(\Omega_T)}\nonumber\\
&=\left\|\frac{(r_{ij}-\bar{r}_{ij})\bar{r}_{11}-\bar{r}_{ij}(r_{11}-\bar{r}_{11})}{r_{11}\bar{r}_{11}}\right\|_{L^\infty(\Omega_T)}\nonumber\\
&\le \left\|\frac{(r_{ij}-\bar{r}_{ij})}{r_{11}}\right\|_{L^\infty(\Omega_T)}+\left|\frac{\bar{r}_{ij}}{\bar{r}_{11}}\right|\cdot\left\|\frac{(r_{11}-\bar{r}_{11})}{r_{11}}\right\|_{L^\infty(\Omega_T)}\nonumber\\
&\le C\delta.\label{3.27}
\end{align}
We also have
\begin{align}
\left\|D\tilde{r}_{ij}^\epsilon\right\|_{L^\infty(\tilde{\Omega}_T)}=\left\|\left(D\frac{r_{ij}}{r_{11}}\right)^\epsilon\right\|_{L^\infty(\tilde{\Omega}_T)}&=\left\|\left(D\frac{r_{ij}}{r_{11}}\right)^\epsilon\right\|_{L^\infty(\Omega_T)}\le\left\|D\left(\frac{r_{ij}}{r_{11}}\right)\right\|_{L^\infty(\Omega_T)}.\nonumber
\end{align}
It is easy to see that
\begin{align}
\left\|D\left(\frac{r_{ij}}{r_{11}}\right)\right\|_{L^\infty(\Omega_T)}
&=\left\|\frac{1}{r_{11}}D r_{ij}-\frac{r_{ij}}{r_{11}^2}D r_{11}\right\|_{L^\infty(\Omega_T)}\nonumber\\
&\le\left\|\frac{1}{r_{11}}\right\|_{L^\infty(\Omega_T)}\!\!\!\!\!\cdot\left\|r_{ij}-\bar{r}_{ij}\right\|_{W^{1,\infty}(\Omega_T)}\nonumber\\
&\quad+\left\|\frac{r_{ij}}{r_{11}^2}\right\|_{L^\infty(\Omega_T)}\!\!\!\!\!\cdot\|r_{11}-\bar{r}_{11}\|_{W^{1,\infty}(\Omega_T)}\nonumber\\
&\le C\delta.\label{3.29}
\end{align}

Similarly, one can deduce that 
$$\left\|\left(\frac{b_i}{b_1}\right)^\epsilon\right\|_{W^{1,\infty}(\tilde{\Omega}_T)}\le C\delta.$$ 
In fact, since $b_2(0,0)=\partial_{z_2}^2b_2(0,0)=0$, the twice differentiable function of $b_2(0,z_2)$ is bounded at $z_2=0$. Hence $|\partial_{z_2}b_2^\epsilon(0,z_2)|\le C\delta$ for all $z_2$.

We remark that the positive constant $C$ in the above inequalities only depends on $\rho_0$ and $\gamma$, but is independent on $\epsilon$.
\vskip 0.2cm	
$(3)$ By the definition of $\tilde{b}_2^\epsilon$, the oddness of $\dfrac{b_2}{b_1}$ and the eveness of $\rho_{\epsilon}$, one has
\begin{align}
	\tilde{b}_2^\epsilon(0,0)&=\int_{\mathbb{R}}\left(\frac{b_2}{b_1}\right)(0,\tau)\rho_{\epsilon}(-\tau)d\tau\nonumber\\
	&=-\int_{\mathbb{R}}\left(\frac{b_2}{b_1}\right)(0,-\tau)\rho_{\epsilon}(\tau)d\tau\nonumber\\
    &=-\int_{\mathbb{R}}\left(\frac{b_2}{b_1}\right)(0,z_2-\tau)\rho_{\epsilon}(\tau)d\tau\Big|_{z_2=0}\nonumber\\
    &=-\tilde{b}_2^\epsilon(0,0),
\end{align}
which implies $\tilde{b}_2^\epsilon(0,0)=0$.
Then 
\begin{align}
	\partial_{z_2}\tilde{b}_2^\epsilon(0,0)&=\frac{1}{\epsilon}\int_{\mathbb{R}}\left(\frac{b_2}{b_1}\right)(0,\tau)\frac{\partial\eta}{\partial z_2}\left(\frac{z_2-\tau}{\epsilon}\right)d\tau\Big|_{z_2=0}\nonumber\\
	&=-\frac{1}{\epsilon}\int_{\mathbb{R}}\left(\frac{b_2}{b_1}\right)(0,\tau)\frac{\partial \eta}{\partial z_2}\left(\frac{\tau-z_2}{\epsilon}\right)d\tau\Big|_{z_2=0}\nonumber\\
	&=\frac{1}{\epsilon}\int_{\mathbb{R}}\left(\frac{b_2}{b_1}\right)(0,\tau)\frac{\partial \eta}{\partial \tau}\left(\frac{\tau-z_2}{\epsilon}\right)d\tau\Big|_{z_2=0}\nonumber\\
	&=-\frac{1}{\epsilon}\int_{\mathbb{R}}\frac{\partial }{\partial \tau}\left(\frac{b_2}{b_1}\right)(0,\tau)\eta\left(\frac{\tau-z_2}{\epsilon}\right)d\tau\Big|_{z_2=0}\nonumber\\
	&=-\frac{1}{\epsilon}\int_{\mathbb{R}}\frac{\partial}{\partial \tau}\left(\frac{b_2}{b_1}\right)(0,\tau)\eta\left(\frac{z_2-\tau}{\epsilon}\right)d\tau\Big|_{z_2=0}\nonumber\\
	&=\frac{1}{\epsilon}\int_{\mathbb{R}}\left(\frac{b_2}{b_1}\right)(0,\tau)\frac{\partial \eta}{\partial \tau}\left(\frac{z_2-\tau}{\epsilon}\right)d\tau\Big|_{z_2=0}\nonumber\\
	&=-\frac{1}{\epsilon}\int_{\mathbb{R}}\left(\frac{b_2}{b_1}\right)(0,\tau)\frac{\partial \eta}{\partial z_2}\left(\frac{z_2-\tau}{\epsilon}\right)d\tau\Big|_{z_2=0}\nonumber\\
	&=-\partial_{z_2}\tilde{b}_2^\epsilon(0,0).
\end{align}
Hence we deduce that $\partial_{z_2}\tilde{b}_2^\epsilon(0,0)=0$.
\vskip 0.2cm
$(4)$ This is an easy consequence of the following fact by lemma \ref{lem: 2.1}:
$$\left(\frac{r_{10}}{r_{11}}\right)|_{z_1=0}=0\quad\mbox{and }\quad\left(\frac{r_{12}}{r_{11}}\right)|_{z_1=0}=\left(\frac{b_{2}}{b_{1}}\right)|_{z_1=0}.$$
\end{proof}
\section{The linearized problem (\Rmnum{2}): Higher order estimates}\label{sec:4}
It is clear that the $H^2(\Omega_T)$-solution of the linearized problem obtained in section \ref{sec:3} is not sufficient to yield a smooth solution to the nonlinear problem by nonlinear iteration method. Hence, in order to deduce the existence of smooth solutions of the nonlinear problem, one has to establish higher order the \emph{a priori} estimate of the solution of the linearized problem derived in section \ref{sec:3}. However, since the regularity of the coefficients of the equation in the extended domain is not sufficient to establish higher order \textit{a priori} estimates. One has to establish higher order estimate in the cornered-space domain directly. But due to the violation of the linear stability conditions of the boundary operator and the presence of the corner singularity, it is difficult to derive higher order \emph{a priori} estimate of the solutions, in particular, the estimate of the boundary terms. Based on the observation that the boundary operators are co-normal and the vanishing properties of the coefficients on the boundaries (see lemma \ref{lem: 2.1} for details), one expresses the boundary terms in terms of some commutators, which reduces the order of the derivatives contained in the boundary terms. Then by the Gauss theorem and the trace theorem, the estimates of boundary terms of the highest order derivatives of the solution can be established. 

In this section, we give a proof to lemma \ref{higher regularity} by establishing  the estimates of the third and fourth order. Due to the corner singularity, the third and the fourth order estimates cannot be derived by same manner, which is similar to the situation in \cite{FXX,FHXX}. Hence they will be deduced separately in the next two subsections. 
\subsection{Third order estimate of the solution}
In this subsection, we establish the third order estimate of the solution obtained in lemma \ref{H2 solvability}. Since the boundary of the space domain is not smooth, it is difficult for us to find multipliers such that the boundary terms on both sides of the corner point have good sign, which is different from the initial boundary value problems on smooth domains. By the properties of the boundary operator $\mathcal{B}$ and the coefficients of the equation (see lemma \ref{lem: 2.1}), we can find suitable multipliers such that the boundary terms can be expressed as some commutators, so that the order of the derivative of the solution is reduced. Then the boundary terms can be estimated by control the commutators, which can be done by using the Gauss theorem and integrating by parts with respect to the time derivative. The third order estimate is summarized as the following lemma: 
\begin{lem}\label{third order regularity}
There exists $\delta_3>0$ such that if assumptions $(\romannumeral1)-(\romannumeral4)$ hold for $\delta\le \delta_3$, then there exists a constant $\eta_3>1$ such that for any $T>0$ and $\eta\ge \eta_3$, the $H^2(\Omega_T)$ solution of problem $\eqref{linear problem}$ satisfies
    \begin{align}
    &\sum_{|\alpha|= 3}\eta  \|e^{-\eta z_0}D^
    \alpha \vp\|^2_{L^2(\Omega_T)}+e^{-2\eta T}\|D^
    \alpha \vp(T,\cdot)\|^2_{L^2(\Omega)}\nonumber\\
    &\lesssim \frac{1}{\eta}\sum_{|\alpha|\le 2}\|e^{-\eta z_0}\mathcal{L}(D^\alpha\vp)\|^2_{L^2(\Omega_T)}
    +\|e^{-\eta z_0}f\|^2_{H^{2}(\Omega_T)}
    +\|f|_{t=0}\|^2_{H^1(\Omega)}.\label{energy estimate in lemma 4.2}
\end{align}
\end{lem}

\begin{proof}
 
Since $\partial_{z_0}$ is tangential to both boundaries $\{z_1=0\}$ and $\{z_2=0\}$, one can apply \eqref{eq:second order estimate_first} to $\partial_{z_0}\vp$ to obtain that
\begin{align}
&\sum_{|\alpha|\le 2}\eta\|e^{-\eta z_0}D^\alpha\partial_{z_0}\vp\|^2_{L^2(\tilde{\Omega}_T)}+e^{-2\eta T}\|D^\alpha\partial_{z_0} \vp|_{z_0=T}\|^2_{L^2(\tilde{\Omega})}\nonumber\\
&\quad\lesssim \frac{1}{\eta}\sum_{|\alpha|\le 2}\|e^{-\eta z_0}\mathcal{L}(D^\alpha \partial_{z_0}\vp)\|^2_{L^2(\tilde{\Omega}_T)}+\|\mathcal{L} (\partial_{z_0}\vp)|_{z_0=0}\|^2_{L^2(\tilde{\Omega})}.\label{eq:3.38}
\end{align}
Next, we will consider the first order estimate of $\partial_{z_2}^2\vp$. It is clear that $\partial_{z_2}^2\vp$ satisfies 
\begin{equation}\label{eq:IBVP of partial_z2z2}
\begin{cases}
\mathcal{L}(\partial_{z_2z_2}\vp)=-[\partial_{z_2z_2},\mathcal{L}]\vp+\partial_{z_2z_2}f, &\mbox{in\ } \tilde{\Omega}_T,\\
\mathcal{B}(\partial_{z_2z_2}\varphi)=-[\partial_{z_2z_2},\mathcal{B}]\vp,&\mbox{on \ }\Gamma_1,\\
\partial_{z_2}^2\varphi=0,&\mbox{on \ }\Gamma_0\\
\partial_{z_0}(\partial_{z_2}^2\varphi)=0,&\mbox{on \ }\Gamma_0.
\end{cases}
\end{equation}
Via the equation and $\partial_{z_2}\vp|_{z_2=0}=0$, we can further derive the boundary condition for $\partial_{z_2z_2}\vp$ on $\{z_2=0\}$ as 
\[
\partial_{z_2}(\partial_{z_2z_2}\vp)=\frac{1}{r_{22}}(\mathcal{L}-r_{00}\partial_{z_0}^2-2r_{02}\partial_{z_0}\partial_{z_2}-2r_{12}\partial_{z_1}\partial_{z_2}-r_{11}\partial_{z_1}^2)\partial_{z_2}\vp.
\]
As required in assumption $(\romannumeral3)$, we have $r_{12}=r_{02}$ on $\{z_2=0\}$. Moreover, $(\partial_{z_0}^2\partial_{z_2}\vp,\partial_{z_0z_1z_2}\vp,\partial_{z_1}^2\partial_{z_2}\vp,\partial_{z_0z_2}\vp,\partial_{z_1z_2}\vp)$ vanishes on $\{z_2=0\}$, since $\partial_{z_2}\vp|_{z_2=0}=0$. As a result, we obtain
\begin{equation}\label{eq:bdry condition for p_z2z2_vp}
    \partial^3_{z_2}\vp=\frac{1}{r_{22}}\mathcal{L}(\partial_{z_2}\vp)\mbox{\ on\ }\{z_2=0\}.
\end{equation}
Multiplying $2e^{-2\eta z_0}\partial_{z_0}\partial^2_{z_2}\vp$ on both sides of $\eqref{eq:IBVP of partial_z2z2}$, and integrating by parts over $\Omega_T$, one has
\begin{align}
    2\int_{\Omega_T}\!\!e^{-2\eta z_0}\mathcal{L}(\partial_{z_2}^2\varphi)\partial_{z_0}\partial_{z_2}^2\varphi dz\!=&\!\left[\int_{\Omega}e^{-2\eta z_0}H_0dz_1dz_2\right]_{z_0=0}^{z_0=T}-\int_0^T\!\!\!\int_{\mathbb{R}^+}e^{-2\eta z_0}H_1|_{z_1=0}dz_2dz_0\nonumber\\
    &+\int_{\Omega_T}\!\!e^{-2\eta z_0}P_2(D\varphi)dz+2\eta\int_{\Omega_T} e^{-2\eta z_0}H_0dz\nonumber\\
    &-2\int_0^T r_{22}e^{-2\eta z_0}\partial_{z_0}\partial_{z_2}^2\vp\partial_{z_2}^3\vp|_{z_2=0} dz_1dz_0\label{eq: 1st step of 3rd order estimate}
\end{align}
where
\begin{align}
 H_0&=|\partial_{z_0}\partial_{z_2}^2\vp|^2-r_{11}|\partial_{z_1}\partial_{z_2}^2\vp|^2-r_{22}|\partial_{z_2}^3\vp|^2,\nonumber\\
H_1&=2((r_{11}\partial_{z_1}+r_{12}\partial_{z_2})\partial_{z_2}^2\vp)\partial_{z_0}\partial_{z_2}^2\vp,
\end{align}
and $P_2(D\partial^2_{z_2}\vp)$ is a quadratic polynomial in $D\partial^2_{z_2}\vp$ with bounded coefficients.
By assumptions $(\romannumeral1)$ and $(\romannumeral2)$, it is easy to see that $H_0\ge C |D\partial_{z_2}^2\vp|^2$ for some positive constant $C$.
Hence we deduce that
\begin{align}
    \eta &\int_{\Omega_T}\!\!e^{-2\eta z_0}|D\partial_{z_2}^2\vp|^2dz+e^{-2\eta T}\int_\Omega |D\partial_{z_2}^2\vp|_{z_0=T}|^2dz_1z_2\nonumber\\
    &\lesssim \int_{\Omega_T}e^{-2\eta z_0}(q\eta +1)|D\partial_{z_2}^2\vp|^2+\frac{1}{q\eta}|\mathcal{L}(\partial_{z_2}^2\vp|^2)dz\nonumber\\
    &\quad+\int_0^T\!\!\!\int_{\mathbb{R}^+}\!\!e^{-2\eta z_0}H_1|_{z_1=0}dz_2dz_0+\int_0^T r_{22}e^{-2\eta z_0}\partial_{z_0}\partial_{z_2}^2\vp\partial_{z_2}^3\vp|_{z_2=0} dz_1dz_0.\label{eq:1st order estimate of partial_z2z2_vp, firts step}
\end{align}
Next, we are forced to control the boundary terms on the right hand-side of the above inequality.
Employing assumption $(\romannumeral4)$ and the boundary condition that $\mathcal{B}\vp=0$ on $\{z_1=0\}$, one arrives at
\begin{align}
        -H_1|_{z_1=0}&=-2r_{11}((\partial_{z_1}+\frac{r_{12}}{r_{11}}\partial_{z_2})\partial_{z_2}^2\vp)\partial_{z_0}\partial^2_{z_2}\vp\nonumber\\
        &=-\frac{2r_{11}}{b_1}(b_1\partial_{z_1}\partial_{z_2}^2\vp+b_2\partial_{z_2}\partial_{z_2}^2\vp)\partial_{z_0}\partial_{z_2}^2\vp\nonumber\\
        &=\frac{2r_{11}}{b_1}([\partial^2_{z_2},\mathcal{B}]\vp)\partial_{z_0}\partial_{z_2}^2\vp\nonumber\\
        &=\frac{2r_{11}}{b_1}\partial_{z_0}\partial_{z_2}^2\vp(2(\partial_{z_2}b_1)\partial_{z_1z_2}\vp+2(\partial_{z_2}b_2)\partial^2_{z_2}\vp&
        \nonumber\\
        &\quad+(\partial_{z_2}^2b_1)\partial_{z_1}\vp+(\partial_{z_2}^2b_2)\partial_{z_2}\vp).\label{eq:commutator -H_1}
\end{align}
Therefore, by the Gauss theorem we have
\begin{align}
&\int_0^T\!\!\!\int_{\mathbb{R}^+}e^{-2\eta z_0}H_1|_{z_1=0}dz_2dz_0\nonumber\\
&=-\int_{\Omega_T}\partial_{z_1}( e^{-2\eta z_0}\frac{2r_{22}}{b_1}([\partial_{z_2}^2,\mathcal{B}]\vp)\partial_{z_0}\partial_{z_2}^2\vp)dz\nonumber\\
&=-\int_{\Omega_T}e^{-2\eta z_0}(\partial_{z_1}\left(\frac{2r_{22}}{b_1}\right)([\partial_{z_2}^2,\mathcal{B}]\vp)\partial_{z_0}\partial_{z_2}^2\vp)+\frac{r_{22}}{b_1}\partial_{z_1}([\partial_{z_2}^2,\mathcal{B}]\vp)\partial_{z_0}\partial_{z_2}^2\vp dz\nonumber\\
&\quad -\int_{\Omega_T}e^{-2\eta z_0}\frac{r_{22}}{b_1}([\partial_{z_2}^2,\mathcal{B}]\vp)\partial_{z_0z_1}\partial_{z_2}^2\vp dz\nonumber\\
&\le C\delta\int_0^T\!\! e^{-2\eta z_0}\|\partial_{z_0}\vp(z_0,\cdot)\|^2_{H^2(\Omega)}+\|\vp(z_0,\cdot)\|^2_{H^3(\Omega)}dz_0\nonumber\\
&\quad-\!\int_{\Omega_T}\!\!e^{-2\eta z_0}\frac{r_{22}}{b_1}([\partial_{z_2}^2,\mathcal{B}]\vp)\partial_{z_0z_1}\partial_{z_2}^2\vp dz\nonumber\\
&\le  C\delta\int_0^T\!\! e^{-2\eta z_0}( \|\vp(z_0,\cdot)\|^2_{H^3(\Omega)}+\|\partial_{z_0}\vp(z_0,\cdot)\|^2_{H^2(\Omega)}dz_0+\mathcal{K}\nonumber,
\end{align}
where
\begin{align}
&\mathcal{K}=-\int_{\Omega_T}\partial_{z_0}(e^{-2\eta z_0}\frac{r_{22}}{b_1}([\partial_{z_2}^2,\mathcal{B}]\vp)\partial_{z_1}\partial_{z_2}^2\vp)dz\nonumber\\
&\qquad-2\eta \int_{\Omega_T}e^{-2\eta z_0}\frac{r_{22}}{b_1}([\partial_{z_2}^2,\mathcal{B}]\vp)\partial_{z_1}\partial_{z_2}^2\vp)+\int_{\Omega_T}\!\!e^{-2\eta z_0}\partial_{z_0}(\frac{r_{22}}{b_1})([\partial_{z_2}^2,\mathcal{B}]\vp)\partial_{z_1}\partial_{z_2}^2\vp dz\nonumber\\
&\qquad+\int_{\Omega_T}e^{-2\eta z_0}\frac{r_{22}}{b_1}\partial_{z_0}([\partial_{z_2}^2,\mathcal{B}]\vp)\partial_{z_1}\partial_{z_2}^2\vp dz.
\end{align}
By assumptions $(\romannumeral1)$, $(\romannumeral2)$, and $(\romannumeral4)$ and $\eqref{eq:commutator -H_1}$, we have
\begin{align}
    |\mathcal{K}|&\lesssim \delta((\eta+1) \int_0^T e^{-2\eta z_0}\|\vp(z_0,\cdot)\|^2_{H^3(\Omega)}dz_0+e^{-2\eta T} \|\vp(T,\cdot)\|^2_{H^3(\Omega)} )\nonumber\\
    &\quad+\delta\int_0^T e^{-2\eta z_0}(\|\partial_{z_0}\vp(z_0,\cdot)\|^2_{H^2(\Omega)}+\|\vp(z_0,\cdot)\|^2_{H^3(\Omega)})dz_0.\label{eq:estimate of K}
\end{align}
Hence we obtain
\begin{align}
    &\left|\int_0^T\int_{\mathbb{R}^+}e^{-2\eta z_0}H_1|_{z_1=0}dz_2dz_0\right|\nonumber\\
    &\qquad\lesssim \delta((\eta+1) \int_0^T e^{-2\eta z_0}\|\vp(z_0,\cdot)\|^2_{H^3(\Omega)}dz_0+e^{-2\eta T}\|\vp(T,\cdot)\|^2_{H^3(\Omega)} )\nonumber\\
    &\qquad\quad+\delta\int_0^T e^{-2\eta z_0}\|\partial_{z_0}\vp(z_0,\cdot)\|^2_{H^2(\Omega)}dz_0.\label{eq:3rd order vertical bdry term}
\end{align}
Now we turn to estimate the last term in \eqref{eq: 1st step of 3rd order estimate}, which is the boundary term on $\{z_2=0\}$. With the aid of \eqref{eq:bdry condition for p_z2z2_vp} and the Gauss theorem, one has
\begin{align}
    -2\int^T_0 \!\!\int_{\mathbb{R}^+} &e^{-2\eta z_0}r_{22}(\partial_{z_0}\partial_{z_2}^2\vp\partial_{z_2}^3\vp)|_{z_2=0}dz_0dz_1\nonumber\\
    &=-2\int_0^Te^{-2\eta z_0}\partial_{z_0}\partial_{z_2}^2\vp\mathcal{L}(\partial_{z_2}\vp)dz_0dz_1\nonumber\\
    &=2\int_{\Omega_T}\partial_{z_2}(e^{-2\eta z_0}\partial_{z_0}\partial_{z_2}^2\vp\mathcal{L}(\partial_{z_2}\vp))dzdz_0\nonumber\\
    &=2\int_{\Omega_T}e^{-2\eta z_0}(\partial_{z_0}\partial_{z_2}^3\vp\mathcal{L}(\partial_{z_2}\vp)+\partial_{z_0}\partial_{z_2}^2\vp\partial_{z_2}\mathcal{L}(\partial_{z_2}\vp))dzdz_0\nonumber\\
    &\defs R_1+R_2.\label{eq:defn of I_1 and I_2}
\end{align}
Via integration by parts with respect to $z_0$, we have
\begin{align}
    |R_1|&=\left|2\int_{\Omega_T}\partial_{z_0}(e^{-2\eta z_0}\partial_{z_2}^3\vp\mathcal{L}(\partial_{z_2}\vp))+ e^{-2\eta z_0}(2\eta\partial_{z_2}^3\vp\mathcal{L}(\partial_{z_2}\vp)-\partial_{z_2}^3\vp\partial_{z_0}\mathcal{L}(\partial_{z_2}\vp)\ dz\right|\nonumber\\
    &\le 2e^{-2\eta T}\int_{\Omega}(|\partial_{z_2}^3\vp\mathcal{L}(\partial_{z_2}\vp))|_{z_0=T}|dz_1dz_2+2\int_{\Omega}(|\partial_{z_2}^3\vp\mathcal{L}(\partial_{z_2}\vp))|_{z_0=0}|dz_1dz_2\nonumber\\
    &\quad+4\eta \int_{\Omega_T}e^{-2\eta z_0}|\partial_{z_2}^3\vp\mathcal{L}(\partial_{z_2}\vp)|dz+2\int_{\Omega_T}e^{-2\eta z_0}|\partial_{z_2}^3\vp\partial_{z_0}\mathcal{L}(\partial_{z_2}\vp)|dz\nonumber\\
    &\lesssim e^{-2\eta T}\int_\Omega (q|\partial_{z_2}^3\vp|_{z_0=T}|^2+\frac{1}{q}|\mathcal{L}(\partial_{z_2}\vp)|_{z_0=T}|^2)dz_1z_2+\|\vp_0\|^2_{H^3(\Omega)}+\|\vp_1\|^2_{H^2(\Omega)}\nonumber\\
    &\quad+\|\partial_{z_2}f|_{z_0=0}\|^2_{L^2(\Omega)}+2\eta\int_{\Omega_T}e^{-2\eta z_0}(q|\partial_{z_2}^3\vp|^2+\frac{1}{q}|\mathcal{L}(\partial_{z_2}\vp)|^2)dz\nonumber\\
    &\quad+\int_{\Omega_T}e^{-2\eta z_0}(q\eta |\partial_{z_2}^3\vp|^2+\frac{1}{q\eta }|\partial_{z_0}\mathcal{L}(\partial_{z_2}\vp)|^2)dz
\end{align}
By \eqref{eq:low norm bounded by higher norm} and the fact that the coefficients of $\mathcal{L}$ are $
W^{1,\infty}(\Omega_T)$ functions, we obtain
\begin{align}
    \eta\int_{\Omega_T}&e^{-2\eta z_0}|\mathcal{L}(\partial_{z_2}\vp)|^2dz+e^{-2\eta T}\int_{\Omega}|\mathcal{L}(\partial_{z_2}\vp)|_{z_0=T}|^2dz_1dz_2\nonumber\\
    &\le \frac{1}{\eta}\int_{\Omega_T}e^{-2\eta z_0}|\partial_{z_0}\mathcal{L}(\partial_{z_2}\vp)|^2dz+\|\mathcal{L}(\partial_{z_2}\vp)|_{z_0=0}\|^2_{L^2(\Omega)}\nonumber\\
    &\lesssim \frac{1}{\eta}\int_{\Omega_T}e^{-2\eta z_0}(\sum_{|\alpha|\le 3}|D^\alpha \vp|^2+|\mathcal{L}(\partial_{z_0z_2}\vp)|^2)dz+\|\partial_{z_2}f|_{z_0=0}\|^2_{L^2(\Omega)}.
\end{align}
Combining above two inequalities, we are led to
\begin{align}
    &|R_1|\lesssim q(\eta \int_{\Omega_T}e^{-2\eta z_0}|\partial_{z_2}^3\vp|^2dz+e^{-2\eta T}\int_\Omega |\partial_{z_2}^3\vp|_{z_0=T}|^2dz_1dz_2)\nonumber\\
    &\qquad+\frac{1}{q}(\frac{1}{\eta}\int_{\Omega_T}e^{-2\eta z_0}(\sum_{|\alpha|\le 3}|D^\alpha \vp|^2+|\mathcal{L}(\partial_{z_0z_2}\vp)|^2)dz+\|\partial_{z_2}f|_{z_0=0}\|^2_{L^2(\Omega)})\nonumber\\
    &\qquad+\|\partial_{z_2}f|_{z_0=0}\|^2_{L^2(\Omega)}.\label{eq:I_1 estimation}
\end{align}
Now we proceed to the estimate of $R_2$. In fact, one has
\begin{align}
 |R_2|&\le2\int_{\Omega_T}e^{-2\eta z_0}|\partial_{z_0}\partial_{z_2}^2
 \vp\partial_{z_2}\mathcal{L}(\partial_{z_2}\vp)|dz\nonumber\\
 &\le  \int_{\Omega_T}e^{-2\eta z_0}(q\eta|\partial_{z_0}\partial_{z_2}^2\vp|^2+\frac{1}{q\eta}|\partial_{z_2}\mathcal{L}(\partial_{z_2}\vp)|^2)dz\nonumber\\
 &\le \int_{\Omega_T}e^{-2\eta z_0}(q\eta|\partial_{z_0}\partial_{z_2}^2\vp|^2+\frac{1}{q\eta}(|\mathcal{L}(\partial_{z_2}^2\vp)|^2+\sum_{|\alpha|\le 3}|D^\alpha\vp|^2))dz.\label{eq:I_2 estimation}
\end{align}
Then inequality \eqref{eq:I_2 estimation} together with \eqref{eq:1st order estimate of partial_z2z2_vp, firts step}, \eqref{eq:3rd order vertical bdry term}, \eqref{eq:defn of I_1 and I_2}, and \eqref{eq:I_1 estimation} yields the following estimate
\begin{align}
    \eta &\int_{\Omega_T}\!\!e^{-2\eta z_0}|D\partial_{z_2}^2\vp|^2dz+e^{-2\eta T}\int_\Omega |D\partial_{z_2}^2\vp|_{z_0=T}|^2dz_1z_2\nonumber\\
    &\lesssim \int_{\Omega_T}e^{-2\eta z_0}(q\eta +1)|D\partial_{z_2}^2\vp|^2+\frac{1}{q\eta}|\mathcal{L}(\partial_{z_2}^2\vp|^2)dz\nonumber\\
    &\quad+ \delta((\eta+1) \int_{\Omega_T}e^{-2\eta z_0}\|\vp(z_0,\cdot)\|^2_{H^3(\Omega)}dz_0+e^{-2\eta T}\|\vp(T,\cdot)\|^2_{H^3(\Omega)} )\nonumber\\
    &\quad+\delta\int_{\Omega_T}e^{-2\eta z_0}\|\partial_{z_0}\vp(z_0,\cdot)\|^2_{H^2(\Omega)}\nonumber\\
    &\quad+q(\eta \int_{\Omega_T}e^{-2\eta z_0}(|\partial_{z_2}^3\vp|^2+|\partial_{z_0}\partial_{z_2}^2\vp|^2)dz+e^{-2\eta T}\int_\Omega |\partial_{z_2}^3\vp|_{z_0=T}|^2dz_1dz_2)\nonumber\\
    &\quad+\frac{1}{q}(\frac{1}{\eta}\int_{\Omega_T}e^{-2\eta z_0}(\sum_{|\alpha|\le 3}|D^\alpha \vp|^2+|\mathcal{L}(\partial_{z_0z_2}\vp)|^2+|\mathcal{L}(\partial_{z_2}^2\vp)|^2)dz+\|\partial_{z_2}f|_{z_0=0}\|^2_{L^2(\Omega)})\nonumber\\
    &\quad+\|\partial_{z_2}f|_{z_0=0}\|^2_{L^2(\Omega)}. \label{eq:1st order estimate of partial_z2z2_vp}
\end{align}
Note that
\begin{align}
    &\partial_{z_1}^2\partial_{z_2}\vp=\frac{1}{r_{11}}(\mathcal{L}(\partial_{z_2}\vp)-\sum_{(i,j)\neq (1,1)}\!\!\!r_{ij}\partial_{ij}\partial_{z_2}\vp),\\
    &\partial_{z_1}^3\vp=\frac{1}{r_{11}}(\mathcal{L}(\partial_{z_1}\vp)-\sum_{(i,j)\neq (1,1)}\!\!\!r_{ij}\partial_{ij}\partial_{z_1}\vp).
\end{align}
So if we select $q$ and $\delta$ properly small and then $\eta$ appropriately large, the third order derivatives on the right hand-side of the inequality above can be absorbed by the left hand-side terms. 
Therefore, 
\begin{align}
    \eta&\int_{\Omega_T}e^{-2\eta z_0}(|\partial_{z_1}^2\partial_{z_2}\vp|^2+|\partial_{z_1}^3\vp|^2)dz+e^{-2\eta T}\!\int_{\Omega}(|\partial_{z_1}^2\partial_{z_2}\vp|^2+|\partial_{z_1}^3\vp|^2)dz_1dz_2\nonumber\\
    &\lesssim \eta\int_{\Omega_T}e^{-2\eta z_0}|D\partial^2_{z_2}\vp|^2dz+e^{-2\eta T}\!\int_{\Omega}|D\partial_{z_2}^2\vp|^2dz_1dz_2\nonumber\\
    &\quad+ \sum_{|\alpha|\le 2}(\eta\int_{\Omega_T}e^{-2\eta z_0}|D^\alpha\partial_{z_0}\vp|^2dz+e^{-2\eta T}\!\int_{\Omega}|D^\alpha\partial_{z_0}\vp|^2dz_1dz_2)\nonumber\\
    &\quad+ \sum_{|\alpha|\le 2}(\eta\int_{\Omega_T}e^{-2\eta z_0}|D^\alpha\vp|^2dz+e^{-2\eta T}\!\int_{\Omega}|D^\alpha\vp|^2dz_1dz_2)\nonumber\\
    &\quad+ \sum_{i=1}^2(\eta\int_{\Omega_T}e^{-2\eta z_0}|\mathcal{L}(\partial_{z_i}\vp)|^2dz+e^{-2\eta T}\!\int_{\Omega}|\mathcal{L}(\partial_{z_i}\vp)|^2dz_1dz_2).\label{eq:estimate of left 3rd order derivatives by equation}
\end{align}
Then it follows from \eqref{eq:low norm bounded by higher norm} that 
\begin{align}
    \sum_{i=1}^2&(\eta\int_{\Omega_T}e^{-2\eta z_0}|\mathcal{L}(\partial_{z_i}\vp)|^2dz+e^{-2\eta T}\!\int_{\Omega}|\mathcal{L}(\partial_{z_i}\vp)|^2dz_1dz_2)\nonumber\\
    &\le \frac{1}{\eta}\int_{\Omega_T}e^{-2\eta z_0}(\sum_{|\alpha|\le 3}|D^\alpha\vp|^2+\sum_{i=1}^2|\mathcal{L}(\partial_{z_iz_0}\vp)|^2)dz+\sum_{i=1}^2\|\partial_{z_i}f|_{z_0=0}\|^2_{L^2(\Omega)}.\label{eq:3-42}
\end{align}
Substituting \eqref{eq:3.38}, \eqref{eq:1st order estimate of partial_z2z2_vp} and \eqref{eq:3-42} into \eqref{eq:estimate of left 3rd order derivatives by equation}, one has
\begin{align}
    \eta&\int_{\Omega_T}e^{-2\eta z_0}(|\partial_{z_1}^2\partial_{z_2}\vp|^2+|\partial_{z_1}^3\vp|^2)dz+e^{-2\eta T}\!\int_{\Omega}(|\partial_{z_1}^2\partial_{z_2}\vp|^2+|\partial_{z_1}^3\vp|^2)dz_1dz_2\nonumber\\
    &\lesssim \int_{\Omega_T}e^{-2\eta z_0}(q\eta +1)|D\partial_{z_2}^2\vp|^2+\frac{1}{q\eta}|\mathcal{L}(\partial_{z_2}^2\vp|^2)dz\nonumber\\
    &\quad+\delta((\eta+1) \int_0^T e^{-2\eta z_0}\|\vp(z_0,\cdot)\|^2_{H^3(\Omega)}dz_0+e^{-2\eta T}\|\vp(T,\cdot)\|^2_{H^3(\Omega)} )\nonumber\\
    &\quad+\delta\int_0^T e^{-2\eta z_0}\|\partial_{z_0}\vp(z_0,\cdot)\|^2_{H^2(\Omega)}dz_0\nonumber\\
    &\quad+q(\eta \int_{\Omega_T}e^{-2\eta z_0}(|\partial_{z_2}^3\vp|^2+|\partial_{z_0}\partial_{z_2}^2\vp|^2)dz+e^{-2\eta T}\int_\Omega |\partial_{z_2}^3\vp|_{z_0=T}|^2dz_1dz_2)\nonumber\\
    &\quad+\frac{1}{q}(\frac{1}{\eta}\int_{\Omega_T}e^{-2\eta z_0}(\sum_{|\alpha|\le 3}|D^\alpha \vp|^2+|\mathcal{L}(\partial_{z_0z_2}\vp)|^2+|\mathcal{L}(\partial_{z_2}^2\vp)|^2)dz+\|\partial_{z_2}f|_{z_0=0}\|^2_{L^2(\Omega)})\nonumber\\
    &\quad+\frac{1}{\eta}\sum_{|\alpha|\le 2}\|e^{-\eta z_0}\mathcal{L}(D^\alpha \vp)\|^2_{L^2(\tilde{\Omega}_T)}+\sum_{|\alpha|\le 1}\|D^\alpha f|_{z_0=0}\|^2_{L^2(\tilde{\Omega})}\nonumber\\
   &\quad+ \frac{1}{\eta}\int_{\Omega_T}e^{-2\eta z_0}(\sum_{|\alpha|\le 3}|D^\alpha\vp|^2)dz.\label{eq:3-44}
\end{align}
To this end, we are ready to conclude the third order estimate. Adding \eqref{eq:second order estimate}, \eqref{eq:3.38}, \eqref{eq:1st order estimate of partial_z2z2_vp} and \eqref{eq:3-44} together, letting $q$ and $\delta$ be properly small and $\eta$ be appropriately large, we deduce that
\begin{align}
    \sum_{|\alpha|\le 3}&\left(\eta \int_{\Omega_T}e^{-2\eta z_0}|D^\alpha\vp|^2dz+e^{-2\eta T}\int_{\Omega}|D^\alpha\vp|_{z_0=T}|^2dz_1z_2\right)\nonumber\\
    &\lesssim \frac{1}{\eta}\sum_{|\alpha|\le 2}\int_{\Omega_T}\!\!e^{-2\eta z_0}|\mathcal{L}(D^\alpha\vp)|^2dz+\sum_{|\alpha|\le 1}\|D^\alpha f|_{z_0=0}\|^2_{L^2(\Omega)}.\label{eq:3rd order estimate}
\end{align}
\end{proof}

\subsection{Fourth order estimate of the solution} 
In this subsection, we establish the fourth order estimate of the solution obtained in lemma \ref{H2 solvability}. For the fourth order estimate, since the normal derivative of the solution contained in the boundary term is one order higher than the one contained in the boundary term of third order estimate, the representations of the boundary terms as the commutators are insufficient to derive the desired estimate. Hence more careful analysis and computation is needed to establish the fourth order estimate. 
We summarize the fourth order estimate as the following lemma: 
\begin{lem}\label{fourth order regularity}
There exists $\delta_4>0$ such that if assumptions $(\romannumeral1)-(\romannumeral4)$ hold for $\delta\le \delta_4$, then there exists a constant $\eta_4>1$ such that for any $T>0$ and $\eta\ge \eta_4$, the $H^2(\Omega_T)$ solution of problem $\eqref{linear problem}$ satisfies
    \begin{align}
    &\sum_{|\alpha|= 4}\eta  \|e^{-\eta z_0}D^
    \alpha \vp\|^2_{L^2(\Omega_T)}+e^{-2\eta T}\|D^
    \alpha \vp(T,\cdot)\|^2_{L^2(\Omega)}\nonumber\\
    &\lesssim \frac{1}{\eta}\sum_{|\alpha|\le 3}\|e^{-\eta z_0}\mathcal{L}(D^\alpha\vp)\|^2_{L^2(\Omega_T)}
    +\|e^{-\eta z_0}f\|^2_{H^{3}(\Omega_T)}
    +\|f|_{t=0}\|^2_{H^2(\Omega)}.\label{energy estimate in lemma 4.2}
\end{align}
\end{lem}

\begin{proof}
 Since $\partial_{z_0}$ is tangential to both boundaries, applying the third order estimation \eqref{eq:3rd order estimate} to $\partial_{z_0}\vp$, one has
\begin{align}
    \sum_{|\alpha|\le 3}&\left(\eta \int_{\Omega_T}e^{-2\eta z_0}|D^\alpha\partial_{z_0}\vp|^2dz+e^{-2\eta T}\int_{\Omega}|D^\alpha\partial_{z_0}\vp|_{z_0=T}|^2dz_1z_2\right)\nonumber\\
    &\lesssim \frac{1}{\eta}\sum_{|\alpha|\le 2}\int_{\Omega_T}\!\!e^{-2\eta z_0}|\mathcal{L}(D^\alpha\partial_{z_0}\vp)|^2dz+\sum_{|\alpha|\le 2}\|D^\alpha f|_{z_0=0}\|^2_{L^2(\Omega)}.\label{eq:3rd order estimate of partial_z0 vp}
\end{align}
Next, we will firstly derive the first order estimate of $\partial_{z_2}^3\vp$, \emph{i.e.}, the estimate of $D \partial_{z_2}^3\vp$. Then by the equation, we are able to control the other fourth order derivatives.
It is easy to verify that $\partial_{z_2}^3\vp$ satisfies
\begin{equation}\label{eq:IBVP of partial_z2z2z_2}
\begin{cases}
\mathcal{L}(\partial_{z_2}^3\varphi)=-[\partial_{z_2}^3,\mathcal{L}]\varphi+\partial_{z_2}^3f, &\mbox{in\ } \Omega_T,\\
\mathcal{B}(\partial_{z_2}^3\varphi)=-[\partial_{z_2}^3,\mathcal{B}]\vp,&\mbox{on \ }\Gamma_1,\\
\partial_{z_2}^3\varphi=0,&\mbox{on \ }\Gamma_0,\\
\partial_{z_0}(\partial_{z_2}^3\varphi)=0,&\mbox{on \ }\Gamma_0.
\end{cases}
\end{equation}
Multiplying $2e^{-2\eta z_0}\partial_{z_0}\partial_{z_2}^3\vp$ on both sides of $\eqref{eq:IBVP of partial_z2z2z_2}_1$, and integrating by parts over $\Omega_T$, one has
\begin{align}
    2\int_{\Omega_T}&\!\!\!e^{-2\eta z_0}\mathcal{L}(\partial_{z_2}^3\vp)\partial_{z_0}\partial_{z_2}^3\vp dz\nonumber\\
    &=\left[\int_{\Omega}e^{-2\eta z_0}H_0dz_1dz_2\right]^{z_0=T}_{z_0=0}+2\eta \int_{\Omega_T}e^{-2\eta z_0}H_0dz-\int_0^T\!\!\!\int_{\mathbb{R}^+}\!\!e^{-2\eta z_0}H_1|_{z_1=0}dz_2dz_0\nonumber\\
    &\quad-2\int_0^T\!\!\int_{\mathbb{R}^+}e^{-2\eta z_0}(r_{22} \partial_{z_0}\partial_{z_2}^3\vp\partial_{z_2}^4\vp)|_{z_2=0}dz_1dz_0+\int_{\Omega_T}\!\!e^{-2\eta z_0} P_3(D\vp)dz
\end{align}
where
\begin{align}
    H_0&=|\partial_{z_0}\partial_{z_2}^3\vp|^2-r_{11}|\partial_{z_1}\partial_{z_2}^3\vp|^2-r_{22}|\partial_{z_2}^4\vp|^2,\\\
    H_1&=2(r_{11}\partial_{z_1}\partial_{z_2}^3\vp+r_{12}\partial_{z_2}^4\vp)\partial_{z_0}\partial_{z_2}^3\vp,
\end{align}
and $P_3(D\partial_{z_2}^3\vp)$ is a quadratic polynomial in $D\partial_{z_2}^3\vp$ with bounded coefficients. Assumptions $(\romannumeral1)$ and $(\romannumeral2)$ imply that
\begin{equation}
H_0\ge C|D\partial_{z_2}^3\vp|^2 \qquad\mbox{and}\qquad
	|P_3(D\vp)|\le C|D\partial_{z_2}^3\vp|^2.
\end{equation}

Hence by the Cauchy inequality, one has
\begin{align}
    \eta&\int_{\Omega_T}\!\!e^{-2\eta z_0}|D\partial_{z_2}^3\vp|^2 dz+e^{-2\eta T}\!\int_{\Omega}|D\partial_{z_2}^3\vp|_{z_0=T}|^2dz_1dz_2\nonumber\\
    &\lesssim \int_{\Omega_T}e^{-2\eta z_0}((q\eta +1)|D\partial_{z_2}^3\vp|^2+\frac{1}{q\eta}|\mathcal{L}(\partial_{z_2}^3\vp)|^2)dz+\int_0^T\!\!\!\int_{\mathbb{R}^+}e^{-2\eta z_0}H_1|_{z_1=0}dz_2dz_0\nonumber\\
    &\quad+2\int_0^T\!\!\!\int_{\mathbb{R}^+}e^{-2\eta z_0}r_{22}\partial_{z_0}\partial_{z_2}^3\vp\partial_{z_2}^4\vp|_{z_2=0} dz_1dz_0.\label{3-52}
\end{align}
By assumption $(\romannumeral3)$, we obtain
\begin{align}
    H_1|_{z_1=0}&=2r_{11}(\partial_{z_1}\partial_{z_2}^3\vp+\frac{r_{12}}{r_{11}}\partial_{z_2}^4\vp)\partial_{z_0}\partial_{z_2}^3\vp\nonumber\\
    &=2\frac{r_{22}}{b_1}\partial_{z_0}\partial_{z_2}^3\vp\{b_1\partial_{z_1}+b_2\partial_{z_2}\}\partial_{z_2}^3\vp\nonumber\\
    &=2\frac{r_{22}}{b_1}\partial_{z_0}\partial_{z_2}^3\vp\mathcal{B}\partial_{z_2}^3\vp\nonumber\\
    &=-2\frac{r_{22}}{b_1}\partial_{z_0}\partial_{z_2}^3\vp([\partial_{z_2}^3,\mathcal{B}]\vp)\nonumber\\
    &=-2\frac{r_{22}}{b_1}\partial_{z_0}\partial_{z_2}^3\vp\sum_{k=0}^2(\partial_{z_2}^{k+1} b_1\partial_{z_1}\partial_{z_2}^{2-k}\vp+\partial_{z_2}^{k+1}b_2\partial_{z_2}^{3-k}\vp),\label{3-53}
\end{align}
where $[\partial^3_{z_2},\mathcal{B}]$ is the commutator and we have used the boundary condition that $\mathcal{B}\vp|_{z_1=0}=0$ in the forth equality.
 For $k=0,1,2$, by the Gauss theorem, we deduce that
\begin{align}
    \left|\int_0^T\right.& \left.e^{-2\eta z_0}\frac{r_{22}}{b_1}\partial_{z_0}\partial_{z_2}^3\vp(\partial_{z_2}^{k+1}b_1)\partial_{z_1}\partial_{z_2}^{2-k}\vp|_{z_1=0}dz_2dz_0\right|\nonumber\\
   & = \left|\int_{\Omega_T} \partial_{z_1}(e^{-2\eta z_0}\frac{r_{22}}{b_1}\partial_{z_0}\partial_{z_2}^3\vp(\partial_{z_2}^{k+1}b_1)\partial_{z_1}\partial_{z_2}^{2-k}\vp)dz\right|\nonumber\\
   &\le \delta\int_{\Omega_T}e^{-2\eta z_0}(\|\partial_{z_0}\vp(z_0,\cdot)\|^2_{H^3(\Omega)}+\|\vp(z_0,\cdot)\|^2_{H^4(\Omega)})dz_0\nonumber\\
   &\quad+\abs{\int_{\Omega_T}e^{-2\eta z_0}\frac{r_{22}}{b_1}\partial_{z_2}^{k+1}b_1\partial_{z_0z_1}\partial_{z_2}^3\vp\partial_{z_1}\partial_{z_2}^{2-k}\vp dz}\nonumber\\
   &= \delta\int_{\Omega_T}e^{-2\eta z_0}(\|\partial_{z_0}\vp(z_0,\cdot)\|^2_{H^3(\Omega)}+\|\vp(z_0,\cdot)\|^2_{H^4(\Omega)})dz_0+\mathcal{A}.
   \end{align}
By integrating by parts with respect to $z_0$, one has for $k=0,1,2$ that
  \begin{align}
   \mathcal{A}&\le\abs{\int_{\Omega_T}\partial_{z_0}(e^{-2\eta z_0}\frac{r_{22}}{b_1}\partial_{z_2}^{k+1}b_1\partial_{z_1}\partial_{z_2}^3\vp\partial_{z_1}\partial_{z_2}^{2-k}\vp )dz}\nonumber\\
   &\quad+2\eta\left|\int_{\Omega_T}e^{-2\eta z_0}\frac{r_{22}}{b_1}\partial_{z_2}^{k+1}b_1\partial_{z_1}\partial_{z_2}^3\vp\partial_{z_1}\partial_{z_2}^{2-k}\vp dz\right|\nonumber\\
   &\quad+\left|\int_{\Omega_T}e^{-2\eta z_0}\partial_{z_0}(\frac{r_{22}}{b_1}\partial_{z_2}^{k+1}b_1)\partial_{z_1}\partial_{z_2}^3\vp\partial_{z_1}\partial_{z_2}^{2-k}\vp dz\right|\nonumber\\
   &\quad+\left|\int_{\Omega_T}e^{-2\eta z_0}\frac{r_{22}}{b_1}\partial_{z_2}^{k+1}b_1\partial_{z_1}\partial_{z_2}^3\vp\partial_{z_0z_1}\partial_{z_2}^{2-k}\vp dz\right|\nonumber\\
   &\lesssim \delta e^{-2\eta T}\|\vp(T,\cdot)\|^2_{H^4(\Omega)}+\delta(\eta+1)\int_0^T e^{-2\eta z_0}\|\vp(z_0,\cdot)\|^2_{H^4(\Omega)}dz_0\nonumber\\
   &\quad+\delta\int_0^T e^{-2\eta z_0}\|\partial_{z_0}\vp(z_0,\cdot)\|^2_{H^3(\Omega)}dz_0,
\end{align}
where in the last inequality, we have used assumption $(\romannumeral4)$. By the same argument, we can also obtain
\begin{align}
   \left|\int_0^T\right.& \left.e^{-2\eta z_0}\frac{r_{22}}{b_1}\partial_{z_0}\partial_{z_2}^3\vp(\partial_{z_2}^{k+1}b_2)\partial_{z_1}\partial_{z_2}^{3-k}\vp|_{z_1=0}dz_2dz_0\right|\nonumber\\
   &\lesssim \delta e^{-2\eta T}\|\vp(T,\cdot)\|^2_{H^4(\Omega)}+\delta(\eta+1)\int_0^T e^{-2\eta z_0}\|\vp(z_0,\cdot)\|^2_{H^4(\Omega)}dz_0\nonumber\\
   &\quad+\delta\int_0^T e^{-2\eta z_0}\|\partial_{z_0}\vp(z_0,\cdot)\|^2_{H^3(\Omega)}dz_0.\label{3-56}
\end{align}
By estimates $\eqref{3-53}$-$\eqref{3-56}$, we deduce
\begin{align}
    &\left|\int_0^T\right.\left.\!\!\!\!\int_{\mathbb{R}^+}e^{-2\eta z_0}H_1|_{z_1=0}dz_2dz_0\right|\nonumber\\
    &\quad\lesssim\delta\left(\eta\int_0^T e^{-2\eta z_0}\|\vp(z_0,\cdot)\|^2_{H^4(\Omega)}dz_0+e^{-2\eta T}\|\vp(T,\cdot)\|^2_{H^4(\Omega)}\right)\nonumber\\
    &\qquad+\delta \int_0^T \!\!\!e^{-2\eta z_0}(\|\vp(z_0,\cdot)\|^2_{H^3(\Omega)}+\|\partial_{z_0}\vp(z_0,\cdot)\|^2_{H^3(\Omega)})dz_0.\label{3-57}
\end{align}
We still need to control the boundary term on boundary $\{z_2=0\}$ in \eqref{3-52}.
By the equation and the properties of the coefficients, we know
\begin{align}
    r_{22}\partial_{z_2}^4\vp|_{z_2=0}&=\{\mathcal{L}-r_{00}\partial_{z_0}^2-2r_{01}\partial_{z_0z_1}-r_{11}\partial_{z_1}^2\}\partial_{z_2}^2\vp,\label{3-58}\\
    \partial_{z_i}\partial_{z_2}^3\vp|_{z_2=0}&=\frac{1}{r_{22}}\mathcal{L}(\partial_{z_iz_2}\vp)\mbox{\ }(i=0,1).\label{3-59}\\
    \partial_{z_2}^3\vp|_{z_2=0}&=\frac{1}{r_{22}}\mathcal{L}(\partial_{z_2}\vp)\label{3-60}.
\end{align}
Therefore, 
\begin{align}
    \int_0^T&\!\!\!\int_{\mathbb{R}^+}e^{-2\eta z_0}r_{22}\partial_{z_0}\partial_{z_2}^3\vp\partial_{z_2}^4\vp|_{z_2=0} dz_1dz_0\nonumber\\
    &=\int_0^T\!\!\!\int_{\mathbb{R}^+}\!\!e^{-2\eta z_0}\partial_{z_0}\partial_{z_2}^3\vp(\{\mathcal{L}-r_{00}\partial_{z_0}^2-2r_{01}\partial_{z_0z_1}-r_{11}\partial_{z_1}^2-\sum_{i=0}^2r_i\partial_{z_i}\}\partial_{z_2}^2\vp)dz_1dz_0\nonumber\\
    &\defs \mathcal{I}_1+\mathcal{I}_2+\mathcal{I}_3+\mathcal{I}_4+\mathcal{I}_5.
\end{align}
For $\mathcal{I}_1$, by the Gauss theorem, one has
\begin{align}
    |\mathcal{I}_1|&=\left|\int_{\Omega_T}\partial_{z_2}(e^{-2\eta z_0}\partial_{z_0}\partial_{z_2}^3\vp\mathcal{L}\partial_{z_2}^2\vp)dz\right|\nonumber\\
    &\le\left|\int_{\Omega_T}e^{-2\eta z_0}\partial_{z_0}\partial_{z_2}^3\vp\partial_{z_2}(\mathcal{L}\partial_{z_2}^2\vp)dz\right|+\left|\int_{\Omega_T}e^{-2\eta z_0}\partial_{z_0}\partial_{z_2}^4\vp\mathcal{L}(\partial_{z_2}^2\vp)dz\right|\nonumber\\
    &\lesssim \int_0^T e^{-2\eta z_0}(\frac{1}{q\eta }\|\partial_{z_2}\mathcal{L}(\partial_{z_2}^2\vp)(z_0,\cdot)\|^2_{L^2(\Omega)}+q\eta\|\partial_{z_0}\vp(z_0,\cdot)\|^2_{H^3(\Omega)})dz_0\nonumber\\
    &\quad+\left|\int_{\Omega_T}\partial_{z_0}(e^{-2\eta z_0}\partial_{z_2}^4\vp\mathcal{L}(\partial_{z_2}^2\vp))+ e^{-2\eta z_0}(2\eta\partial_{z_2}^4\vp\mathcal{L}(\partial_{z_2}^2\vp)-\partial_{z_2}^4\vp\partial_{z_0}\mathcal{L}(\partial_{z_2}^2\vp))dz\right|\nonumber\\
    &\lesssim  \int_0^T e^{-2\eta z_0}(\frac{1}{q\eta }\|\partial_{z_2}\mathcal{L}(\partial_{z_2}^2\vp)(z_0,\cdot)\|^2_{L^2(\Omega)}+q\eta\|\partial_{z_0}\vp(z_0,\cdot)\|^2_{H^3(\Omega)})dz_0\nonumber\\
    &\quad+e^{-2\eta T}\left(q\|\partial_{z_2}^4\vp(T,\cdot)\|^2_{L^2(\Omega)}+\frac{1}{q}\|\mathcal{L}(\partial_{z_2}^2\vp)(T,\cdot)\|^2_{L^2(\Omega)}\right)\nonumber\\
    &\quad+\eta\int_0^T e^{-2\eta z_0}q\|\partial_{z_2}^4\vp(z_0)\|^2_{L^2(\Omega)}+\frac{1}{q}\|\mathcal{L}\partial_{z_2}^2\vp(z_0)\|^2_{L^2(\Omega)}dz_0+\|D^2f|_{t=0}\|^2_{L^2(\Omega)}\nonumber\\
    &\quad+\int_0^T e^{-2\eta z_0}(\frac{1}{q\eta }\|\partial_{z_0}\mathcal{L}(\partial_{z_2}^2\vp)(z_0,\cdot)\|^2_{L^2(\Omega)}+q\eta\|\vp(z_0,\cdot)\|^2_{H^4(\Omega)})dz_0.
    \label{3-61}
\end{align}
By the boundedness of the coefficients of $\mathcal{L}$, we know that 
\begin{align}
    \int_{0}^T&e^{-2\eta z_0}(\|\partial_{z_2}\mathcal{L}(\partial_{z_2}^2\vp)\|^2_{L^2(\Omega)}+\|\partial_{z_0}\mathcal{L}(\partial_{z_2}^2\vp)\|^2_{L^2(\Omega)})dz_0\nonumber\\
    &\lesssim \int_0^T\!\!e^{-2\eta z_0}\left(\sum_{|\alpha|\le 4}\|D^\alpha\vp(z_0,\cdot)\|^2_{L^2(\Omega)}+\sum_{|\alpha|\le 3}\|\mathcal{L}(D^\alpha\vp)(z_0,\cdot)\|^2_{L^2(\Omega)}\right)dz_0.\label{3-62}
\end{align}
With the help of $\eqref{eq:low norm bounded by higher norm}$, one gets
\begin{align}
    \eta\int_{0}^T&e^{-2\eta z_0}\|\mathcal{L}\partial_{z_2}^2\vp\|^2_{L^2(\Omega)}dz_0+e^{-2\eta T}\|\mathcal{L}(\partial_{z_2}^2\vp)(T,\cdot)\|^2_{L^2(\Omega)}\nonumber\\
    &\le \frac{1}{\eta}\int_{0}^Te^{-2\eta z_0}\|\partial_{z_0}\mathcal{L}(\partial_{z_2}^2\vp)\|^2_{L^2(\Omega)}dz_0+\|\mathcal{L}(\partial_{z_2}^2\vp)|_{z_0=0}\|^2_{L^2(\Omega)}\nonumber\\
    &\lesssim \frac{1}{\eta}\int_{0}^Te^{-2\eta z_0}(\|\mathcal{L}(\partial_{z_0}\partial_{z_2}^2\vp)\|^2_{L^2(\Omega)}+\sum_{|\alpha|\le 4}\|D^\alpha\vp\|^2_{L^2(\Omega)})dz_0\nonumber\\
    &\quad+\|\partial_{z_2}^2f|_{z_0=0}\|^2_{L^2(\Omega)}.\label{3-63}
\end{align}
Combining $\eqref{3-61}$, $\eqref{3-62}$ and $\eqref{3-63}$, we deduce that
\begin{align}
    |\mathcal{I}_1|&\lesssim q\left(\eta \int_{\Omega_T}e^{-2\eta z_0}\sum_{|\alpha|\le 4}\|D^\alpha\vp(z_0,\cdot)\|^2_{L^2(\Omega)}dz_0+e^{-2\eta T}\sum_{|\alpha|\le 4}\|D^\alpha\vp(T,\cdot)\|^2_{L^2(\Omega)}\right)\nonumber\\
    &\quad+\frac{1}{q\eta}\int_0^T\!\!e^{-2\eta z_0}\left(\sum_{|\alpha|\le 4}\|D^\alpha\vp(z_0,\cdot)\|^2_{L^2(\Omega)}+\sum_{|\alpha|\le 3}\|\mathcal{L}(D^\alpha\vp)(z_0,\cdot)\|^2_{L^2(\Omega)}\right)dz_0\nonumber\\
    &\quad+\frac{1}{q\eta}\int_{0}^T\!\!e^{-2\eta z_0}\left(\sum_{|\alpha|\le 3}\|\mathcal{L}(D^\alpha\vp)(z_0,\cdot)\|^2_{L^2(\Omega)}+\sum_{|\alpha|\le 4}\|D^\alpha\vp(z_0,\cdot)\|^2_{L^2(\Omega)}\right)dz_0\nonumber\\
    &\quad+\frac{1}{q}\|\partial_{z_2}^2f|_{z_0=0}\|^2_{L^2(\Omega)}.
\end{align}
Since both $\mathcal{I}_2$ and $\mathcal{I}_3$ contain the time derivative $\partial_{z_0}$, by the argument similar to the one used for $\mathcal{I}_1$, we obtain
\begin{align}
    |\mathcal{I}_2|&+|\mathcal{I}_3|\nonumber\\
    &\lesssim q\left(\eta \int_{\Omega_T}e^{-2\eta z_0}\sum_{|\alpha|\le 4}\|D^\alpha\vp\|^2_{L^2(\Omega)}dz_0+e^{-2\eta T}\sum_{|\alpha|\le 4}\|D^\alpha\vp(T,\cdot,\cdot)\|^2_{L^2(\Omega)}\right)\nonumber\\
    &\quad+\frac{1}{q\eta}\int_0^T\!\!e^{-2\eta z_0}\left(\sum_{|\alpha|\le 4}\|D^\alpha\vp(z_0)\|^2_{L^2(\Omega)}+\sum_{|\alpha|\le 3}\|\mathcal{L}(D^\alpha\vp)(z_0)\|^2_{L^2(\Omega)}\right)dz_0\nonumber\\
    &\quad+\frac{1}{q\eta}\int_{\Omega_T}e^{-2\eta z_0}\left(\sum_{|\alpha|\le 3}\|\mathcal{L}(D^\alpha\vp)\|^2_{L^2(\Omega)}+\sum_{|\alpha|\le 4}\|D^\alpha\vp\|^2_{L^2(\Omega)}\right)dz_0\nonumber\\
    &\quad+\frac{1}{q}\|\partial_{z_2}^2f|_{z_0=0}\|^2_{L^2(\Omega)}.
\end{align}
To estimate $\mathcal{I}_4$, we can not directly use Gauss theorem and then integration by part with respect to $z_0$, since $\partial_{z_1}^2\partial_{z_2}^2\vp$ does not contain the time derivative $\partial_{z_0}$. Instead, we integrate with respect to $z_1$, since $\partial_{z_1}$ is also tangential to $\Gamma_2$. Actually, one has
\begin{align}
    |\mathcal{I}_4|&\le\underbrace{\left|\int_0^T\!\!\int_{\mathbb{R}^+}\partial_{z_1}\left( e^{-2\eta z_0}\frac{r_{11}}{r_{22}}\mathcal{L}(\partial_{z_0z_2}\vp) \partial_{z_1}\partial_{z_2}^2\vp\right)dz_1dz_0\right|}_{\mathcal{I}_{4,1}}\nonumber\\
    &\quad+ \underbrace{\left|\int_0^T\!\!\int_{\mathbb{R}^+}e^{-2\eta z_0}\frac{r_{11}}{r_{22}}\partial_{z_1}(\mathcal{L}\partial_{z_0z_2}\vp)\partial_{z_1}\partial_{z_2}^2\vp dz_1dz_0\right|}_{\mathcal{I}_{4,2}}\nonumber\\
    &\quad+\underbrace{\left|\int_0^T\!\!\int_{\mathbb{R}^+}e^{-2\eta z_0}\partial_{z_1}\left(\frac{r_{11}}{r_{22}}\right)\partial_{z_0}\partial_{z_2}^3\vp\partial_{z_1}\partial_{z_2}^2\vp dz_1dz_0\right|}_{\mathcal{I}_{4,3}}.
\end{align}
By Cauchy inequality and the trace theorem, it is easy to see that
\begin{align}
    \mathcal{I}_{4,1}&=\left| \int_0^T e^{-2\eta z_0}(\frac{r_{11}}{r_{22}})\partial_{z_0}\partial_{z_2}^3\vp \partial_{z_1}\partial_{z_2}^2\vp(z_0,0,0) dz_0\right|\nonumber\\
    &= \left|\int_0^Te^{-2\eta z_0}(\frac{r_{11}}{b_1r_{22}^2})(\partial_{z_0}\partial_{z_2}^3\vp)\mathcal{B}(\partial_{z_2}^2\vp)(z_0,0,0)dz_0\right|.\label{3.68}
\end{align}
where in the second equality, we have used $b_2(0,0)=0$.
It is clear that
\begin{align}
    \mathcal{B}(\partial_{z_2}^2\vp)=2(\partial_{z_2}b_1\partial_{z_1z_2}\vp+\partial_{z_2}b_2\partial_{z_2z_2}\vp)+\partial_{z_2z_2}b_1\partial_{z_1}\vp+\partial_{z_2z_2}b_2\partial_{z_2}\vp.\label{3.69}
\end{align}
From the boundary conditions
\[
\mathcal{B}\varphi|_{z_1=0}=0,\quad \partial_{z_2}\vp|_{z_2=0}=0, \quad\mbox{and}\quad b_2(0,0)=0,
\]
one can see that
\begin{equation}\label{3.70}
    \partial_{z_1z_2}\vp(z_0,0,0)=0\quad \mbox{and} \quad
    \partial_{z_1}\vp(z_0,0,0)=0.
\end{equation}
Combining \eqref{3.69} and \eqref{3.70}, one has
\[
\mathcal{B}(\partial_{z_2}^2\vp)(z_0,0,0)=2\partial_{z_2}b_2(0,0)\partial_{z_2z_2}\vp(z_0,0,0).
\]
By assumption $(\romannumeral3)$, we conclude that $\mathcal{B}(\partial_{z_2}^2\vp)(z_0,0,0)=0$. Hence we have $\mathcal{I}_{4,1}=0$.

Noticing that $\partial_{z_2}\vp$ and its tangential derivatives vanish on $\{z_2=0\}$, we have
\begin{align}
\mathcal{L}(\partial_{z_0z_2}\vp)|_{z_2=0}&=\left\{-[\partial_{z_0},\mathcal{L}]\partial_{z_2}\vp+\partial_{z_0}(\mathcal{L}(\partial_{z_2}\vp))\right\}\Big|_{z_2=0}\nonumber\\
&=-((2(\partial_{z_0}r_{02})\partial_{z_0}\partial_{z_2}+2(\partial_{z_0}r_{12})\partial_{z_1z_2}+(\partial_{z_0}r_{22})\partial_{z_2}^2))\partial_{z_2}\vp)|_{z_2=0}\nonumber\\
&\quad-\partial_{z_0}r_2\partial_{z_2z_2}\vp|_{z_2=0}+\partial_{z_0}(\mathcal{L}(\partial_{z_2}\vp))|_{z_2=0}.
\end{align}
So by integrating by part with respect to $z_0$ and recalling \eqref{3-59} and \eqref{3-60}, we have
\begin{align}
    \mathcal{I}_{4,2}&\le\underbrace{\left|\int_{0}^T\!\!\!\int_{\mathbb{R}^+}2e^{-2\eta z_0}\frac{r_{11}}{r_{22}}\left((\partial_{z_0z_1}r_{02})\partial_{z_0}\partial_{z_2}^2\vp+(\partial_{z_0}r_{02})\partial_{z_0z_1}\partial_{z_2}^2\vp \right)\partial_{z_1}\partial_{z_2}^2\vp dz_1dz_0\right|}_{\mathcal{J}_1}\nonumber\\
    &\quad+\underbrace{ \left|\int_0^T\int_{\mathbb{R}^+}2e^{-2\eta z_0}\frac{r_{11}}{r_{22}}\left((\partial_{z_0z_1}r_{12})\partial_{z_1}\partial_{z_2}^2\vp+(\partial_{z_0}r_{12})\partial_{z_1}^2\partial_{z_2}^2\vp \right)\partial_{z_1}\partial_{z_2}^2\vp dz_1dz_0\right|}_{\mathcal{J}_2}\nonumber\\
    &\quad+\underbrace{\left|\int_0^T\int_{\mathbb{R}^+}2e^{-2\eta z_0}\frac{r_{11}}{r_{22}}\left((\partial_{z_0z_1}r_{22})\partial_{z_2}^3\vp+(\partial_{z_0}r_{22})\partial_{z_1}\partial_{z_2}^3\vp \right)\partial_{z_1}\partial_{z_2}^2\vp dz_1dz_0\right|}_{\mathcal{J}_3}\nonumber\\
    &\quad+\underbrace{\left|\int_0^T\int_{\mathbb{R}^+}2e^{-2\eta z_0}\frac{r_{11}}{r_{22}}\left((\partial_{z_0z_1}r_{2})\partial_{z_2}^2\vp+(\partial_{z_0}r_{2})\partial_{z_1}\partial_{z_2}^2\vp \right)\partial_{z_1}\partial_{z_2}^2\vp dz_1dz_0\right|}_{\mathcal{J}_4}\nonumber\\
    &\quad+\underbrace{\left|\int_0^T\int_{\mathbb{R}^+}2e^{-2\eta z_0}\frac{r_{11}}{r_{22}}\left(\partial_{z_0z_1}\mathcal{L}(\partial_{z_2}\vp \right)\partial_{z_1}\partial_{z_2}^2\vp dz_1dz_0\right|}_{\mathcal{J}_5}.
\end{align}

We estimate $\mathcal{J}_i$ term by term.
By the Gauss theorem, it is clear that
\begin{align}
	\mathcal{J}_1&=\left|\int_{\Omega_T}\partial_{z_2}\left(2e^{-2\eta z_0}\frac{r_{11}}{r_{22}}\left((\partial_{z_0z_1}r_{02})\partial_{z_0}\partial_{z_2}^2\vp+(\partial_{z_0}r_{02})\partial_{z_0z_1}\partial_{z_2}^2\vp \right)\partial_{z_1}\partial_{z_2}^2\vp \right)dzdz_0\right|\nonumber\\
	&\le \underbrace{\left|\int_{\Omega_T}\partial_{z_2}\left(2e^{-2\eta z_0}\frac{r_{11}}{r_{22}}(\partial_{z_0z_1}r_{02})\partial_{z_0}\partial_{z_2}^2\vp\partial_{z_1}\partial_{z_2}^2\vp\right)dzdz_0\right|}_{\mathcal{J}_{1,1}}\nonumber\\
	&\quad+\underbrace{\left|\int_{\Omega_T}\partial_{z_2}\left(2e^{-2\eta z_0}\frac{r_{11}}{r_{22}}(\partial_{z_0}r_{02})\partial_{z_0z_1}\partial_{z_2}^2\vp \partial_{z_1}\partial_{z_2}^2\vp \right)dzdz_0\right|}_{\mathcal{J}_{1,2}}.\label{3.75}
	\end{align}
By simple calculation we have	
	\begin{align}
	\mathcal{J}_{1,1}&\le\left|\int_{\Omega_T}2e^{-2\eta z_0}\partial_{z_2}\left(\frac{r_{11}}{r_{22}}\right)(\partial_{z_0z_1}r_{02})\partial_{z_0}\partial^2_{z_2}\vp\partial_{z_1}\partial_{z_2}^2\vp dzdz_0\right|\nonumber\\
	&\quad+\left|\int_{\Omega_T}2e^{-2\eta z_0}\frac{r_{11}}{r_{22}}(\partial_{z_0z_1z_2}r_{02})\partial_{z_0}\partial^2_{z_2}\vp\partial_{z_1}\partial_{z_2}^2\vp dzdz_0\right|\nonumber\\
	&\quad+\left|\int_{\Omega_T}2e^{-2\eta z_0}\frac{r_{11}}{r_{22}}(\partial_{z_0z_1}r_{02})\partial_{z_0}\partial^3_{z_2}\vp\partial_{z_1}\partial_{z_2}^2\vp dzdz_0\right|\nonumber\\
	&\quad+\left|\int_{\Omega_T}2e^{-2\eta z_0}\frac{r_{11}}{r_{22}}(\partial_{z_0z_1}r_{02})\partial_{z_0}\partial^2_{z_2}\vp\partial_{z_1}\partial_{z_2}^3\vp dzdz_0\right|.\nonumber
\end{align}
By the H\"{o}lder's inequality, one has
	\begin{align}
	\mathcal{J}_{1,1}&\lesssim \int_0^T e^{-2\eta z_0}\|\partial_{z_0z_1}r_{02}(z_0,\cdot)\|_{L^2}\|\partial_{z_0}\partial_{z_2}^2\vp(z_0,\cdot)\|_{L^4}\|\partial_{z_1}\partial_{z_2}^2\vp(z_0,\cdot)\|_{L^4}dz_0\nonumber\\	
	&\quad+\int_0^T e^{-2\eta z_0}\|\partial_{z_0z_1z_2}r_{02}(z_0,\cdot)\|_{L^2}\|\partial_{z_0}\partial_{z_2}^2\vp(z_0,\cdot)\|_{L^4}\|\partial_{z_1}\partial_{z_2}^2\vp(z_0,\cdot)\|_{L^4}dz_0\nonumber\\
	&\quad+\int_0^T e^{-2\eta z_0}\|\partial_{z_0z_1}r_{02}(z_0,\cdot)\|_{L^4}\|\partial_{z_0}\partial_{z_2}^3\vp(z_0,\cdot)\|_{L^2}\|\partial_{z_1}\partial_{z_2}^2\vp(z_0,\cdot)\|_{L^4}dz_0\nonumber\\
	&\quad+\int_0^T e^{-2\eta z_0}\|\partial_{z_0z_1}r_{02}(z_0,\cdot)\|_{L^2}\|\partial_{z_0}\partial_{z_2}^2\vp(z_0,\cdot)\|_{L^4}\|\partial_{z_1}\partial_{z_2}^3\vp(z_0,\cdot)\|_{L^2}dz_0.\label{3.76}
	\nonumber\\
	&\lesssim \delta \int_0^T e^{-2\eta z_0}\left(\|\partial_{z_0}\vp(z_0,\cdot)\|^2_{H^3(\Omega)}+\|\vp(z_0,\cdot)\|^2_{H^4(\Omega)}\right)dz_0,
\end{align}
where in the last inequality, we have used the Sobolev embedding $H^1(\Omega)\hookrightarrow L^4(\Omega)$ and assumption $(\romannumeral2)$.
Similarly, one deduces that
\begin{align}
	\mathcal{J}_{1,2}&\lesssim \delta\int_0^T e^{-2\eta z_0}\left(\|\partial_{z_0}\vp(z_0,\cdot)\|^2_{H^3(\Omega)}+\|\vp(z_0,\cdot)\|^2_{H^4(\Omega)}\right)dz_0\nonumber\\
	&\quad+\underbrace{\left|\int_{\Omega_T}e^{-2\eta z_0}\partial_{z_0}r_{02}\partial_{z_0z_1}\partial_{z_2}^3\vp\partial_{z_1}\partial_{z_2}^2\vp dzdz_0\right|}_{\mathcal{J}^0_{1,2}}.\label{3.77}
	\end{align}
Integrating by part with respect to $z_0$, one has
\begin{align}	
\mathcal{J}^0_{1,2}&\lesssim \left|\int_{\Omega_T}\partial_{z_0}\left(e^{-2\eta z_0}\partial_{z_0}r_{02}\partial_{z_1}\partial_{z_2}^3\vp\partial_{z_1}\partial_{z_2}^2\vp\right) dzdz_0\right|\nonumber\\
	&\quad+\eta\left|\int_{\Omega_T}e^{-2\eta z_0}\partial_{z_0}r_{02}\partial_{z_1}\partial_{z_2}^3\vp\partial_{z_1}\partial_{z_2}^2\vp dzdz_0\right| \nonumber\\
	&\quad+\left|\int_{\Omega_T}e^{-2\eta z_0}\partial_{z_0}r_{02}\partial_{z_1}\partial_{z_2}^3\vp\partial_{z_0z_1}\partial_{z_2}^2\vp dzdz_0\right|\nonumber\\
	&\quad+\left|\int_{\Omega_T}e^{-2\eta z_0}\partial_{z_0}^2r_{02}\partial_{z_1}\partial_{z_2}^3\vp\partial_{z_1}\partial_{z_2}^2\vp dzdz_0\right|.\nonumber
	\end{align}
By the Cauchy inequality, one deduces
\begin{align}
	\mathcal{J}_{1,2}^0&\lesssim e^{-2\eta T}\left(q\|\partial_{z_1}\partial_{z_2}^3\vp(T,\cdot)\|^2_{L^2(\Omega)}+\frac{1}{q}\|\partial_{z_1}\partial_{z_2}^2\vp(T,\cdot)\|^2_{L^2(\Omega)}\right)+\|\vp_0\|^2_{H^4(\Omega)}\nonumber\\
	&\quad+\eta \left(\int_0^T e^{-2\eta z_0}\left(q\|\partial_{z_1}\partial_{z_2}^3\vp(z_0,\cdot)\|^2_{L^2(\Omega)}+\frac{1}{q}\|\partial_{z_1}\partial_{z_2}^2\vp(z_0,\cdot)\|^2_{L^2(\Omega)} \right)dz_0\right)\nonumber\\
	&\quad+\int_0^T e^{-2\eta z_0}\left(\|\partial_{z_1}\partial_{z_2}^3(z_0,\cdot)\|^2_{L^2(\Omega)}+\|\partial_{z_0z_1}\partial_{z_2}^2\vp(z_0,\cdot)\|^2_{L^2(\Omega)}\right)dz_0\nonumber\\
	&\quad+\delta \int_0^T\|\partial_{z_1}\partial_{z_2}^3\vp(z_0,\cdot)\|^2_{L^2(\Omega)}+\|\partial_{z_1}\partial_{z_2}^2\vp(z_0,\cdot)\|^2_{L^2(\Omega)}dz_0.\label{3.78}
\end{align}
Combining \eqref{3.75}-\eqref{3.78}, we obtain
\begin{align}
	\mathcal{J}_1&\lesssim  (\delta+1) \int_0^T e^{-2\eta z_0}\left(\|\partial_{z_0}\vp(z_0,\cdot)\|^2_{H^3(\Omega)}+\|\vp(z_0,\cdot)\|^2_{H^4(\Omega)}\right)dz_0\nonumber\\
	&\quad+(q\eta+1)\int_0^T \|\partial_{z_1}\partial_{z_2}^3\vp(z_0,\cdot)\|^2_{L^2}dz_0+q e^{-2\eta T}\|\partial_{z_1}\partial_{z_2}^3\vp(T,\cdot)\|^2_{L^2(\Omega)}\nonumber\\
	&\quad+\frac{1}{q}\left(\eta\int_0^T \|\partial_{z_1}\partial_{z_2}^2\vp(z_0,\cdot)\|^2_{L^2}dz_0+e^{-2\eta T}\|\partial_{z_1}\partial_{z_2}^2\vp(T,\cdot)\|^2_{L^2(\Omega)}\right).
\end{align}
For $\mathcal{J}_2$, we have
\begin{align}
\mathcal{J}_2&=
 \left|\int_0^T\int_{\mathbb{R}^+}2e^{-2\eta z_0}\frac{r_{11}}{r_{22}}\left((\partial_{z_0z_1}r_{12})\partial_{z_1}\partial_{z_2}^2\vp+(\partial_{z_0}r_{12})\partial_{z_1}^2\partial_{z_2}^2\vp \right)\partial_{z_1}\partial_{z_2}^2\vp dz_1dz_0\right|\nonumber\\
 &\lesssim \underbrace{\left|\int_0^T\int_{\mathbb{R}^+}2e^{-2\eta z_0}\frac{r_{11}}{r_{22}}(\partial_{z_0z_1}r_{12})\partial_{z_1}\partial_{z_2}^2\vp \partial_{z_1}\partial_{z_2}^2\vp dz_1dz_0\right|}_{\mathcal{J}_{2,1}}\nonumber\\
 &\quad+\underbrace{\left|\int_0^T\int_{\mathbb{R}^+}2e^{-2\eta z_0}\frac{r_{11}}{r_{22}}(\partial_{z_0}r_{12})\partial_{z_1}^2\partial_{z_2}^2\vp \partial_{z_1}\partial_{z_2}^2\vp dz_1dz_0\right|}_{\mathcal{J}_{2,2}}.
\end{align}
By the Gauss theorem and assumption $(\romannumeral2)$, it is not difficult to deduce that
\begin{align}
	\mathcal{J}_{2,1}&=\left|\int_{\Omega_T}\partial_{z_2}\left(2e^{-2\eta z_0}\frac{r_{11}}{r_{22}}(\partial_{z_0z_1}r_{12})\partial_{z_1}\partial_{z_2}^2\vp \partial_{z_1}\partial_{z_2}^2\vp\right) dzdz_0\right|\nonumber\\
	&\lesssim \int_0^T e^{-2\eta z_0}\|\partial_{z_0z_1}r_{12}(z_0,\cdot)\|_{L^4(\Omega)}\|\partial_{z_1}\partial_{z_2}^2\vp(z_0,\cdot)\|_{L^4(\Omega)}^2dz_0\nonumber\\
	&\quad+\int_0^T e^{-2\eta z_0}\|\partial_{z_0z_1z_2}r_{12}(z_0,\cdot)\|_{L^2(\Omega)}\|\partial_{z_1}\partial_{z_2}^2\vp(z_0,\cdot)\|_{L^4(\Omega)}^2dz_0\nonumber\\
	&\quad+\int_0^T e^{-2\eta z_0}\|\partial_{z_0z_1}r_{12}(z_0,\cdot)\|_{L^4(\Omega)}\|\partial_{z_1}\partial_{z_2}^3\vp(z_0,\cdot)\|_{L^2(\Omega)}\|\partial_{z_1}\partial_{z_2}^2\vp(z_0,\cdot)\|_{L^4(\Omega)}dz_0\nonumber\\
	&\lesssim \delta\int_0^T e^{-2\eta z_0}\|\partial_{z_1}\partial_{z_2}^3\vp(z_0,\cdot)\|^2_{H^1(\Omega)}dz_0,
\end{align}
where the Sobolev embedding $H^1(\Omega)$ $\hookrightarrow$ $L^4(\Omega)$ is used in the last inequality.
For $\mathcal{J}_{2,2}$, one has
\begin{align}
	\mathcal{J}_{2,2}&=\left|\int_0^T\int_{\mathbb{R}^+}e^{-2\eta z_0}\frac{r_{11}}{r_{22}}(\partial_{z_0}r_{12})\partial_{z_1}\left(\left|\partial_{z_1}\partial_{z_2}^2\vp\right|^2\right) dz_1dz_0\right|\nonumber\\
	&\le\underbrace{\left|\int_0^T\int_{\mathbb{R}^+}\partial_{z_1}\left(e^{-2\eta z_0}\frac{r_{11}}{r_{22}}(\partial_{z_0}r_{12})\left|\partial_{z_1}\partial_{z_2}^2\vp\right|^2\right)dz_1dz_0\right|}_{\mathcal{J}_{2,2}^1}\nonumber\\
	&\quad+ \underbrace{\left|\int_0^T\int_{\mathbb{R}^+}e^{-2\eta z_0}\partial_{z_1}\left(\frac{r_{11}}{r_{22}}\right)(\partial_{z_0}r_{12})\left|\partial_{z_1}\partial_{z_2}^2\vp\right|^2dz_1dz_0\right|}_{\mathcal{J}_{2,2}^2}\nonumber\\
	&\quad+\underbrace{\left|\int_0^T\int_{\mathbb{R}^+}e^{-2\eta z_0}\frac{r_{11}}{r_{22}}(\partial_{z_0z_1}r_{12})\left|\partial_{z_1}\partial_{z_2}^2\vp\right|^2dz_1dz_0\right|}_{\mathcal{J}_{2,2}^3}.
	\end{align}
We claim $\mathcal{J}_{2,2}^1=0$. Indeed, applying $\partial_{z_2}^2$ to the boundary condition $\mathcal{B}\vp(z_0,0,z_2)=0$ and letting $z_2=0$, one has
\begin{align}
	b_1\partial_{z_1}\partial_{z_2}^2\vp(z_0,0,0)=(\mathcal{B}\partial_{z_2}^2\vp)(z_0,0,0)-b_2\partial_{z_2}^3\vp(z_0,0,0).\label{3.83}
\end{align}
By \eqref{3.69}, \eqref{3.70}, and assumption ($\romannumeral3$), it is clear that $(\mathcal{B}\partial_{z_2}^2\vp)(z_0,0,0)=0$, $b_2(0,0)=0$, and $b_1(0,0)=-1\neq 0$. This together with \eqref{3.83} imply that $\partial_{z_1}\partial_{z_2}^2\vp(z_0,0,0)=0$, hence our claim holds.
For $\mathcal{J}_{2,2}^2$, by the trace theorem and assumptions $(\romannumeral1)$ and $(\romannumeral2)$, one obtains
\begin{align}
	\mathcal{J}_{2,2}^2\lesssim \int_0^T e^{-2\eta z_0}\|\partial_{z_1}\partial_{z_2}^2\vp(z_0,\cdot)\|^2_{H^1(\Omega)}dz_0.
\end{align}
By the Gauss theorem, the Sobolev embedding theorem, and assumption $(\romannumeral2)$, we deduce that
\begin{align}
\mathcal{J}_{2,2}^3	&= \int_{\Omega_T} \partial_{z_2}\left( e^{-2\eta z_0}\frac{r_{11}}{r_{22}}\partial_{z_0z_1}r_{12}\left|\partial_{z_1}\partial_{z_2}^2\vp\right|^2\right)dzdz_0\nonumber\\
	&\lesssim \int_0^T e^{-2\eta z_0}\|\partial_{z_0z_1}r_{12}(z_0,\cdot)\|_{L^2(\Omega)}\|\partial_{z_1}\partial_{z_2}^2\vp(z_0,\cdot)\|^2_{L^4(\Omega)} dz_0\nonumber\\
	&\quad+\int_0^T e^{-2\eta z_0}\|\partial_{z_0z_1z_2}r_{12}(z_0,\cdot)\|_{L^2(\Omega)}\|\partial_{z_1}\partial_{z_2}^2\vp(z_0,\cdot)\|^2_{L^4(\Omega)}dz_0\nonumber\\
	&\quad+\int_0^T e^{-2\eta z_0}\|\partial_{z_0z_1}r_{12}(z_0,\cdot)\|_{L^2(\Omega)}\|\partial_{z_1}\partial_{z_2}^2\vp(z_0,\cdot)\|_{L^4(\Omega)}\|\partial_{z_1}\partial_{z_2}^3\vp(z_0,\cdot)\|_{L^2(\Omega)}dz_0\nonumber\\
	&\lesssim \delta \int_0^Te^{-2\eta z_0}\|\partial_{z_1}\partial_{z_2}^2(z_0,\cdot)\|^2_{H^1(\Omega)} dz_0.\label{3.85}
	\end{align}
	Hence we conclude that
	\begin{align}
		\mathcal{J}_{2}\lesssim (\delta+1)\int_0^Te^{-2
		\eta z_0}\|\partial_{z_1}\partial_{z_2}^2\vp(z_0,\cdot)\|^2_{H^1(\Omega)}dz_0.
	\end{align}
For $\mathcal{J}_{3}$, one has
\begin{align}
	\mathcal{J}_3&=\left|\int_0^T\int_{\mathbb{R}^+}2e^{-2\eta z_0}\frac{r_{11}}{r_{22}}\left((\partial_{z_0z_1}r_{22})\partial_{z_2}^3\vp+(\partial_{z_0}r_{22})\partial_{z_1}\partial_{z_2}^3\vp \right)\partial_{z_1}\partial_{z_2}^2\vp dz_1dz_0\right|\nonumber\\
	&\le \underbrace{\left|\int_0^T\int_{\mathbb{R}^+}2e^{-2\eta z_0}\frac{r_{11}}{r_{22}}(\partial_{z_0z_1}r_{22})\partial_{z_2}^3\vp \partial_{z_1}\partial_{z_2}^2\vp dz_1dz_0\right|}_{J_{3,1}}\nonumber\\
	&\quad+\underbrace{\left|\int_0^T\int_{\mathbb{R}_+}2 e^{-2\eta z_0}(\partial_{z_0}r_{22})\partial_{z_1}\partial_{z_2}^3\vp \partial_{z_1}\partial_{z_2}^2\vp dz_1dz_0\right|}_{\mathcal{J}_{3,2}}.\label{3.87}
\end{align}
By the Gauss theorem and the H{\"o}lder inequality, it is not difficult to see that
\begin{align}
	\mathcal{J}_{3,1}&\lesssim \int_0^T e^{-2\eta z_0}\|\partial_{z_0z_1}r_{22}(z_0,\cdot)\|_{L^2(\Omega)}\|\partial_{z_2}^3\vp(z_0,\cdot)\|_{L^4(\Omega)}\|\partial_{z_1}\partial_{z_2}^2\vp(z_0,\cdot)\|_{L^4(\Omega)}dz_0\nonumber\\
	&\quad+ \int_0^T e^{-2\eta z_0}\|\partial_{z_0z_1z_2}r_{22}(z_0,\cdot)\|_{L^2(\Omega)}\|\partial_{z_2}^3\vp(z_0,\cdot)\|_{L^4(\Omega)}\|\partial_{z_1}\partial_{z_2}^2\vp(z_0,\cdot)\|_{L^4(\Omega)}dz_0\nonumber\\
	&\quad+ \int_0^T e^{-2\eta z_0}\|\partial_{z_0z_1}r_{22}(z_0,\cdot)\|_{L^2(\Omega)}\|\partial_{z_2}^4\vp(z_0,\cdot)\|_{L^2(\Omega)}\|\partial_{z_1}\partial_{z_2}^2\vp(z_0,\cdot)\|_{L^4(\Omega)}dz_0\nonumber\\
	 &\quad+\int_0^T e^{-2\eta z_0}\|\partial_{z_0z_1}r_{22}(z_0,\cdot)\|_{L^2(\Omega)}\|\partial_{z_2}^3\vp(z_0,\cdot)\|_{L^4(\Omega)}\|\partial_{z_1}\partial_{z_2}^3\vp(z_0,\cdot)\|_{L^2(\Omega)}dz_0\nonumber\\
	 &\lesssim \delta\int_0^T e^{-2\eta z_0}\left(\|\partial_{z_2}^3\vp(z_0,\cdot)\|^2_{H^1(\Omega)}+\|\partial_{z_1}\partial_{z_2}^2\vp(z_0,\cdot)\|^2_{H^1(\Omega)}\right)dz_0.\label{3.88}
\end{align}
For $\mathcal{J}_{3,2}$, the Gauss theorem cannot be used directly, since it contains fourth order derivative $\partial_{z_1}\partial_{z_2}^3\vp$ on boundary $\{z_2=0\}$.
Therefore, we use \eqref{3-59} to replace $\!\ \partial_{z_1}\partial_{z_2}^3\vp$ in $\mathcal{J}_{3,2}$ and then apply the trace theorem and the Cauchy inequality to derive the estimate.
In fact, one has
\begin{align}
	\mathcal{J}_{3,2}&=\left|\int_0^T\int_{\mathbb{R}_+}2 e^{-2\eta z_0}\frac{\partial_{z_0}r_{22}}{r_{22}}\mathcal{L}(\partial_{z_1z_2}\vp)\partial_{z_1}\partial_{z_2}^2\vp dz_1dz_0\right|\nonumber\\
	&\lesssim \int_0^Te^{-2\eta z_0}\|\partial_{z_0z_2}r_{22}(z_0,\cdot)\|_{L^4(\Omega)}\|\mathcal{L}(\partial_{z_1z_2}\vp)(z_0,\cdot)\|_{L^2(\Omega)}\|\partial_{z_1}\partial_{z_2}^2\vp(z_0,\cdot)\|_{L^2(\Omega)}dz_0\nonumber\\
	&\quad+\int_0^Te^{-2\eta z_0}\|\mathcal{L}(\partial_{z_1z_2}\vp)(z_0,\cdot)\|_{L^2(\Omega)}\|\partial_{z_1}\partial_{z_2}^2\vp(z_0,\cdot)\|_{L^2(\Omega)}dz_0\nonumber\\
	&\quad+ \int_0^T e^{-2\eta z_0}\|\partial_{z_2}\mathcal{L}(\partial_{z_1z_2}\vp)(z_0,\cdot)\|_{L^2(\Omega)}\|\partial_{z_1}\partial_{z_2}^2\vp(z_0,\cdot)\|_{L^2(\Omega)}dz_0\nonumber\\
		&\quad+ \int_0^Te^{-2\eta z_0}\|\mathcal{L}(\partial_{z_1z_2}\vp)(z_0,\cdot)\|_{L^2(\Omega)}\|\partial_{z_1}\partial_{z_2}^3\vp(z_0,\cdot)\|_{L^2(\Omega)}dz_0.\nonumber
		\end{align}
By the Cauchy inequality and assumption $(\romannumeral2)$, we have
		\begin{align}
		\mathcal{J}_{3,2}&\lesssim (\delta+1)\int_0^T e^{-2\eta z_0}\left(\|\partial_{z_0}\vp(z_0,\cdot)\|_{H^3(\Omega)}+\|\partial_{z_0}^2\vp(z_0,\cdot)\|_{H^2(\Omega)}\right)\|\vp(z_0,\cdot)\|_{H^4(\Omega)}dz_0\nonumber\\
		&\quad+\int_0^T e^{-2\eta z_0}\left(q\eta \|\partial_{z_1}\partial_{z_2}^2\vp(z_0,\cdot)\|^2_{L^2(\Omega)}+\frac{1}{ q\eta}\|\partial_{z_2}\mathcal{L}(\partial_{z_1z_2}\vp)(z_0,\cdot)\|^2\right)dz_0\nonumber\\
		&\lesssim \left(\delta+1+q\eta+\frac{1}{q\eta }\right)\sum_{|\alpha|\le 4}\int_0^T e^{-2\eta z_0}\|D^\alpha \vp(z_0,\cdot )\|^2_{L^2(\Omega)} dz_0\nonumber\\
		&\quad+\frac{1}{ q\eta}\sum_{|\alpha|\le 3}\int_0^T e^{-2\eta z_0}\mathcal{L}(D^\alpha\vp)(z_0,\cdot)\|^2dz_0.\label{3.89}
\end{align}
Combining \eqref{3.87}-\eqref{3.89}, one deduces that
\begin{align}
	\mathcal{J}_3	&\lesssim \left(\delta+1+q\eta+\frac{1}{q\eta }\right)\sum_{|\alpha|\le 4}\int_0^T e^{-2\eta z_0}\|D^\alpha \vp(z_0,\cdot )\|^2_{L^2(\Omega)} dz_0\nonumber\\
	&\quad+\frac{1}{ q\eta}\sum_{|\alpha|\le 3}\int_0^T e^{-2\eta z_0}\|\mathcal{L}(D^\alpha\vp)(z_0,\cdot)\|^2_{L^2(\Omega)}dz_0.\label{3.90}
\end{align}
Noticing that $\mathcal{J}_4$ contains no derivatives higer than third order, it can be estimated easily by the Gauss theorem, the trace theorem, and assumption $(\romannumeral2)$. In fact, we have
\begin{align}
	\mathcal{J}_4&=\left|\int_0^T\int_{\mathbb{R}^+}2e^{-2\eta z_0}\frac{r_{11}}{r_{22}}\left((\partial_{z_0z_1}r_{2})\partial_{z_2}^2\vp+(\partial_{z_0}r_{2})\partial_{z_1}\partial_{z_2}^2\vp \right)\partial_{z_1}\partial_{z_2}^2\vp dz_1dz_0\right|\nonumber\\
    &\lesssim (\delta+1)\sum_{|\alpha|\le 4}\int_0^T e^{-2\eta z_0}\|D^\alpha\vp(z_0,\cdot)\|^2_{L^2(\Omega)}dz_0.
\end{align}
For $\mathcal{J}_5$, by the Gauss theorem, one has
\begin{align}
	\mathcal{J}_5&=\left|\int_0^T\int_{\mathbb{R}^+}2e^{-2\eta z_0}\frac{r_{11}}{r_{22}}\left(\partial_{z_0z_1}\mathcal{L}(\partial_{z_2}\vp \right)\partial_{z_1}\partial_{z_2}^2\vp dz_1dz_0\right|\nonumber\\
&=\left|\int_{\Omega_T}\partial_{z_2}\left(2e^{-2\eta z_0}\frac{r_{11}}{r_{22}}\left(\partial_{z_0z_1}\mathcal{L}(\partial_{z_2}\vp \right)\partial_{z_1}\partial_{z_2}^2\vp\right) dzdz_0\right|\nonumber\\
&=\underbrace{\left|\int_{\Omega_T}2e^{-2\eta z_0}\partial_{z_2}\left(\frac{r_{11}}{r_{22}}\right)\left(\partial_{z_0z_1}\mathcal{L}(\partial_{z_2}\vp)\right) \partial_{z_1}\partial_{z_2}^2\vp dzdz_0\right|}_{\mathcal{J}_{5,1}}\nonumber\\
&\quad+\underbrace{\left|\int_{\Omega_T}2 e^{-2\eta z_0}\left(\frac{r_{11}}{r_{22}}\right)\left(\partial_{z_0z_1z_2}\mathcal{L}(\partial_{z_2}\vp)\right) \partial_{z_1}\partial_{z_2}^2\vp dzdz_0\right|}_{\mathcal{J}_{5,2}}\nonumber\\
&\quad+\underbrace{\left|\int_{\Omega_T}2 e^{-2\eta z_0}\left(\frac{r_{11}}{r_{22}}\right)\left(\partial_{z_0z_1}\mathcal{L}(\partial_{z_2}\vp)\right) \partial_{z_1}\partial_{z_2}^3\vp dzdz_0\right|}_{\mathcal{J}_{5,3}}.
\end{align}
By the H{\"o}lder inequality, one has
\begin{align}
	\mathcal{J}_{5,1}&\lesssim \int_0^T e^{-2\eta z_0}\|(\partial_{z_0z_1}\mathcal{L}(\partial_{z_2}\vp))(z_0,\cdot)\|_{L^2(\Omega)}\|\partial_{z_1}\partial_{z_2}^2\vp(z_0,\cdot)\|_{L^2(\Omega)}dz_0\nonumber\\
	&\lesssim  (\delta+1+q\eta )\int_0^T e^{-2\eta z_0}\sum_{|\alpha|\le 4}\|D^\alpha \vp(z_0,\cdot)\|^2_{L^2(\Omega)}dz_0\nonumber\\
	&\quad+\frac{1}{q\eta }\int_0^T e^{-2\eta z_0}\|\mathcal{L}(\partial_{z_0z_1z_2}\vp)(z_0,\cdot)\|^2_{L^2(\Omega)} dz_0.
\end{align}
where we have used assumption $(\romannumeral2)$ and the following fact:
\begin{align}
	\partial_{z_kz_\ell}(\mathcal{L}\partial_{z_2}\vp)&=(\partial_{z_k}\mathcal{L})\partial_{z_\ell z_2}\vp+(\partial_{z_\ell}\mathcal{L})\partial_{z_k z_2}\vp+(\partial_{z_kz_\ell}\mathcal{L})\partial_{z_2}\vp+\mathcal{L}(\partial_{z_k z_\ell z_2}\vp),\label{3.94}
\end{align}
where
\begin{align}
\partial_{z_k}\mathcal{L}&\defs \sum_{i,j=0}^2\partial_{z_k}r_{ij}\partial_{z_iz_j},\label{3.95}\\
\partial_{z_kz_\ell}\mathcal{L}&\defs \sum_{i,j=0}^2\partial_{z_kz_\ell}r_{ij}\partial_{z_iz_j}.\label{3.96}
\end{align}

Integrating by parts with respect to $z_0$, we have
\begin{align}
\mathcal{J}_{5,2}&\le\underbrace{\left|\int_{\Omega_T}\partial_{z_0}\left(2 e^{-2\eta z_0}\left(\frac{r_{11}}{r_{22}}\right)\left(\partial_{z_1z_2}\mathcal{L}(\partial_{z_2}\vp)\right) \partial_{z_1}\partial_{z_2}^2\vp\right) dzdz_0\right|}_{\mathcal{J}_{5,2}^1}\nonumber\\
&\quad+\underbrace{2\eta \left|\int_{\Omega_T}2 e^{-2\eta z_0}\left(\frac{r_{11}}{r_{22}}\right)\left(\partial_{z_1z_2}\mathcal{L}(\partial_{z_2}\vp) \partial_{z_1}\partial_{z_2}^2\vp\right) dzdz_0\right|}_{\mathcal{J}_{5,2}^2}\nonumber\\
&\quad+\underbrace{\left|\int_{\Omega_T}2 e^{-2\eta z_0}\partial_{z_0}\left(\frac{r_{11}}{r_{22}}\right)\left(\partial_{z_1z_2}\mathcal{L}(\partial_{z_2}\vp) \partial_{z_1}\partial_{z_2}^2\vp\right)dzdz_0\right|}_{\mathcal{J}_{5,2}^3}\nonumber\\
&\quad+\underbrace{\left|\int_{\Omega_T}2 e^{-2\eta z_0}\left(\frac{r_{11}}{r_{22}}\right)\left(\partial_{z_1z_2}\mathcal{L}(\partial_{z_2}\vp) \partial_{z_0z_1}\partial_{z_2}^2\vp\right)dzdz_0\right|}_{\mathcal{J}_{5,2}^4}.\label{3.97}
\end{align}
It is clear that
\begin{align}
\mathcal{J}_{5,2}^1&\lesssim e^{-2\eta T}\int_{\Omega}\left| \partial_{z_1z_2}\mathcal{L}(\partial_{z_2}\vp) \partial_{z_1}\partial_{z_2}^2\vp\big|_{z_0=T}\right| dz
+\|f|_{z_0=0}\|^2_{H^3(\Omega)}\nonumber\\
	&\lesssim e^{-2\eta T}\left((q+\delta)\sum_{|\alpha|\le 4}\|D^\alpha \vp(T,\cdot)\|^2_{L^2(\Omega)}+\frac{1}{q}\|\partial_{z_1}\partial_{z_2}^2(T,\cdot)\|^2\right)\nonumber\\ &\quad+e^{-2\eta T}\left(q\eta^2\|\partial_{z_1}\partial_{z_2}^2\vp(T,\cdot)\|^2_{L^2(\Omega)}+\frac{1}{q\eta^2}\|\mathcal{L}(\partial_{z_1}\partial_{z_2}^2\vp)(T,\cdot)\|^2_{L^2(\Omega)}\right)\nonumber\\
	&\quad+\|f|_{z_0=0}\|^2_{H^3(\Omega)}\nonumber\\
	&\lesssim e^{-2\eta T}\left((q+\delta)\sum_{|\alpha|\le 4}\|D^\alpha \vp(T,\cdot)\|^2_{L^2(\Omega)}+\frac{1}{q\eta^2}\|\mathcal{L}(\partial_{z_1}\partial_{z_2}^2\vp)(T,\cdot)\|^2_{L^2(\Omega)}\right)\nonumber\\
	&\quad+\left(q\eta +\frac{1}{q\eta}\right)\int_0^T e^{-2\eta z_0}\|\partial_{z_0z_1}\partial_{z_2}^2\vp(z_0,\cdot)\|^2_{L^2(\Omega)}dz_0\nonumber\\
	&\quad+ \|f|_{z_0=0}\|^2_{H^3(\Omega)}.\label{3.98}
\end{align}
In the above estimates, the Cauchy inequality and inequality \eqref{eq:low norm bounded by higher norm} are empoyed, and $\delta$ comes from the $L^4$-norm (which can be bounded by the $H^1$-norm) of $D^2r_{ij}$. 
By a similar argument, we can also deduce that
\begin{align}
	\mathcal{J}^2_{5,2}&\lesssim \eta\int_0^T e^{-2\eta z_0}\left(\frac{1}{q\eta^2}\|\mathcal{L}(\partial_{z_1}\partial_{z_2}^2\vp)(z_0,\cdot)\|^2+q\eta^2\|\partial_{z_1}\partial_{z_2}^2\vp(z_0,\cdot)\|^2_{L^2(\Omega)}\right)dz_0\nonumber\\
	&\quad+\eta \int_0^T\left(\frac{1}{q}\|\partial_{z_1}\partial_{z_2}^2\vp(z_0,\cdot)\|^2_{L^2(\Omega)}+(q+\delta)\sum_{|\alpha|\le 4}\|D^\alpha\vp(z_0,\cdot)\|^2_{L^2(\Omega)}\right)dz_0\nonumber\\
	&\lesssim  \int_0^T e^{-2\eta z_0}\left(\frac{1}{q\eta}\|\mathcal{L}(\partial_{z_1}\partial_{z_2}^2\vp)(z_0,\cdot)\|^2_{L^2(\Omega)}+\eta(q+\delta)\sum_{|\alpha|\le 4}\|D^\alpha\vp(z_0,\cdot)\|^2_{L^2(\Omega)}\right)dz_0\nonumber\\
	&\quad+\frac{1}{q\eta}\int_0^T e^{-2\eta z_0}\sum_{|\alpha|\le 4}\|D^\alpha\vp(z_0,\cdot)\|^2_{L^2(\Omega)}dz_0.
\label{3.99}
\end{align}
By the Cauchy inequality, it is not difficult to see that
\begin{align}
&\mathcal{J}_{5,2}^3+\mathcal{J}_{5,2}^4\nonumber\\
	&\lesssim \int_0^T e^{-2\eta z_0}\left(\frac{1}{q\eta} \|\mathcal{L}(\partial_{z_1}\partial_{z_2}^2\vp)(z_0,\cdot)\|^2_{L^2(\Omega)}+q\eta \|\partial_{z_0z_1}\partial_{z_2}^2\vp(z_0,\cdot)\|^2_{L^2(\Omega)}\right)dz_0.\nonumber\\
	&\quad+\delta \int_0^T e^{-2\eta z_0}\sum_{|\alpha|\le 4}\|D^\alpha\vp(z_0,\cdot)\|^2_{L^2(\Omega)}dz_0.\label{3.100}
\end{align}
Combining \eqref{3.97}-\eqref{3.100}, we are able to conclude that
\begin{align}
	\mathcal{J}_5&\lesssim \frac{1}{q\eta}\int_0^T e^{-2\eta z_0} \sum_{|\alpha|\le 3}\|\mathcal{L}(D^\alpha\vp)(z_0,\cdot)\|^2_{L^2(\Omega)}dz_0\nonumber\\
	&\quad+\left(\delta+(q+\delta)\eta+\frac{1}{q\eta}+1\right)\int_0^T e^{-2\eta z_0} \sum_{|\alpha|\le 4}\|D^\alpha\vp(z_0,\cdot)\|^2_{L^2(\Omega)}dz_0\nonumber\\
	&\quad+e^{-2\eta T}\left((q+\delta)\sum_{|\alpha|\le 4}\|D^\alpha \vp(T,\cdot)\|^2_{L^2(\Omega)}+\frac{1}{q\eta^2}\|\mathcal{L}(\partial_{z_1}\partial_{z_2}^2\vp)(T,\cdot)\|^2_{L^2(\Omega)}\|^2\right)\nonumber\\
	&\quad+\left(q\eta +\frac{1}{q\eta}\right)\int_0^T e^{-2\eta z_0}\|\partial_{z_0z_1}\partial_{z_2}^2\vp(z_0,\cdot)\|^2_{L^2(\Omega)}dz_0+ \|f|_{z_0=0}\|^2_{H^3(\Omega)}.
\end{align}
Collecting the estimates of $\mathcal{J}_1,\cdots,\mathcal{J}_5$, we obtain the estimate of $\mathcal{I}_{4,2}$:
\begin{align}
	\mathcal{I}_{4,2}&\lesssim \frac{1}{q\eta}\sum_{|\alpha|\le 3}\int_{\Omega_T}e^{-2\eta z_0}|\mathcal{L}(D^\alpha\vp)|^2dzdz_0+(\frac{1}{q\eta}+q\eta)\sum_{|\alpha|\le 4}\int_{\Omega_T}e^{-2\eta z_0}|D^\alpha\vp|^2dzdz_0\nonumber\\
	&\quad+\|f|_{z_0=0}\|^2_{H^3(\Omega)}+(1+q\eta^2)\|\vp_0\|^2_{H^4(\Omega)}+\|\vp_1\|^2_{H^3(\Omega)}\nonumber\\
	&\quad+\frac{1}{q\eta^2}e^{-2\eta T}\left(\|\vp(T,\cdot)\|^2_{H^4(\Omega)}+\sum_{|\alpha|\le 3}\|\mathcal{L}(D^\alpha)\vp(T,\cdot)\|^2_{L^2(\Omega)}\right)\nonumber\\
	&\quad+q e^{-2\eta T}\|\vp(T,\cdot)\|^2_{H^4(\Omega)}\label{3-80}.
\end{align}
For $\mathcal{I}_{4,3}$, via \eqref{3-59}, the trace theorem, and the Cauchy inequality, one has
\begin{align}
	 \mathcal{I}_{4,3}&\lesssim \int_0^T e^{-2\eta z_0}\left(\frac{1}{q\eta}\|\mathcal{L}(\partial_{z_0z_2}\vp)\|^2_{H^1(\Omega)}+q\eta \|\partial_{z_1}\partial_{z_2}^2\vp\|^2_{H^1(\Omega)}\right)dz_0\nonumber\\
	 &\lesssim \frac{1}{q\eta }\sum_{|\alpha|\le 3}\int_{\Omega_T}e^{-2\eta z_0}|\mathcal{L}(D^\alpha\vp)|^2dzdz_0+\left(\frac{1}{q\eta}+q\eta\right)\sum_{|\alpha|\le 4}\|e^{-2\eta z_0}D^\alpha\vp\|^2_{L^2(\Omega_T)}.\label{3-81}
\end{align}
Now the sum of the estimate of $\mathcal{I}_{4,1}$, $\mathcal{I}_{4,2}$, and $\mathcal{I}_{4,3}$ yields 
\begin{align}
	|\mathcal{I}_4|&\lesssim \frac{1}{q\eta}\sum_{|\alpha|\le 3}\int_{\Omega_T}e^{-2\eta z_0}|\mathcal{L}(D^\alpha\vp)|^2dzdz_0\nonumber\\
	&\quad+(\frac{1}{q\eta}+(q+\delta)\eta+\delta+1)\sum_{|\alpha|\le 4}\int_{\Omega_T}e^{-2\eta z_0}|D^\alpha\vp|^2dzdz_0\nonumber\\
	&\quad+e^{-2\eta T}\left((q+\delta)\|\vp(T,\cdot)\|^2_{H^4(\Omega)}+\frac{1}{q\eta^2}\sum_{|\alpha|\le 3}\|\mathcal{L}(D^\alpha\vp)(T,\cdot)\|^2_{L^2(\Omega)}\right)\nonumber\\
&\quad+\|f|_{z_0=0}\|^2_{H^3(\Omega)}.	
\end{align}
For $\mathcal{I}_5$, by \eqref{3-59} and the trace theorem
\begin{align}
    |\mathcal{I}_5|&=\left|-\int_0^T\!\!\!\int_{\mathbb{R}^+} e^{-2\eta z_0}\partial_{z_0}\partial_{z_2}^3\vp\sum_{i=0}^2r_i\partial_{z_i}\partial_{z_2}^2\vp dz_1dz_0\right|\nonumber\\
   &=\left|\int_0^T\!\!\!\int_{\mathbb{R}^+} e^{-2\eta z_0}\frac{1}{r_{22}}\mathcal{L}(\partial_{z_0z_2}\vp)\sum_{i=0}^2r_i\partial_{z_i}\partial_{z_2}^2\vp dz_1dz_0\right|\nonumber\\
   &\lesssim \int_0^T e^{-2\eta z_0}\left(\frac{1}{q\eta}\|\mathcal{L}\partial_{z_0z_2}\vp(z_0,\cdot)\|^2_{H^1(\Omega)}+q\eta \|\vp(z_0,\cdot)\|^2_{H^4(\Omega)}\right)dz_0\nonumber\\
   &\quad+\int_0^T e^{-2\eta z_0}\|\partial_{z_0}\vp(z_0,\cdot)\|^2_{H^3(\Omega)}dz_0\nonumber\\
   &\lesssim \int_0^Te^{-2\eta z_0}\frac{1}{q\eta}\left(\sum_{|\alpha|\le 4}\|D^\alpha\vp\|^2_{L^2(\Omega)}+\sum_{|\alpha|\le 3}\|\mathcal{L}(D^\alpha\vp)\|^2_{L^2(\Omega)}\right)dz_0\nonumber\\
   &\quad+(q\eta+1)\sum_{|\alpha|\le 4}\int_0^T e^{-2\eta z_0}\|D^\alpha\vp\|^2_{L^2(\Omega)}dz_0,
\end{align}
where in the first inequality, we have used the Cauchy inequality.
Gathering the estimates of $\mathcal{I}_1$, $\mathcal{I}_2$, $\cdots$, $\mathcal{I}_5$,
we finally have 
\begin{align}
	&\int_0^T\!\!\!\int_{\mathbb{R}^+}e^{-2\eta z_0}r_{22}\partial_{z_0}\partial_{z_2}^3\vp\partial_{z_2}^4\vp|_{z_2=0} dz_1dz_0\nonumber\\
	&\quad\lesssim  q\left(\eta \int_{0}^T e^{-2\eta z_0}\sum_{|\alpha|\le 4}\|D^\alpha\vp\|^2_{L^2(\Omega)}dz_0+e^{-2\eta T}\sum_{|\alpha|\le 4}\|D^\alpha\vp(T,\cdot)\|^2_{L^2(\Omega)}\right)\nonumber\\
	&\qquad+\frac{1}{q\eta}\int_0^T\!\!e^{-2\eta z_0}\left(\sum_{|\alpha|\le 4}\|D^\alpha\vp(z_0)\|^2_{L^2(\Omega)}+\sum_{|\alpha|\le 3}\|\mathcal{L}(D^\alpha\vp)(z_0)\|^2_{L^2(\Omega)}\right)dz_0\nonumber\\
	&\qquad+\frac{1}{q\eta}\int_{0}^T\!\!e^{-2\eta z_0}\left(\sum_{|\alpha|\le 3}\|\mathcal{L}(D^\alpha\vp)\|^2_{L^2(\Omega)}+\sum_{|\alpha|\le 4}\|D^\alpha\vp\|^2_{L^2(\Omega)}\right)dz_0\nonumber\\
	&\qquad+\frac{1}{q\eta^2}e^{-2\eta T}\left(\|\vp(T,\cdot)\|^2_{H^4(\Omega)}+\sum_{|\alpha|\le 3}\|\mathcal{L}(D^\alpha\vp)(T,\cdot)\|^2_{L^2(\Omega)}\right)\nonumber\\
	&\qquad+q e^{-2\eta T}\|\vp(T,\cdot)\|^2_{H^4(\Omega)}+\frac{1}{q}\|\partial_{z_2}^2f|_{z_0=0}\|^2_{L^2(\Omega)}).\label{3-84}
\end{align}
By \eqref{3-52}, \eqref{3-57}, and \eqref{3-84}, we obtain the estimate of $D\partial_{z_2}^3\vp$, i.e.,
\begin{align}
	&\eta\int_{\Omega_T}\!\!e^{-2\eta z_0}|D\partial_{z_2}^3\vp|^2 dzdz_0+e^{-2\eta T}\!\int_{\Omega}|D\partial_{z_2}^3\vp|_{z_0=T}|^2dz_1dz_2\nonumber\\
	&\quad\lesssim \delta\left(\eta\int_0^T e^{-2\eta z_0}\|\vp(z_0,\cdot)\|^2_{H^4(\Omega)}dz_0+e^{-2\eta T}\|\vp(T,\cdot)\|^2_{H^4(\Omega)}\right)\nonumber\\
	&\qquad+\delta \int_0^T \!\!\!e^{-2\eta z_0}(\|\vp(z_0,\cdot)\|^2_{H^3(\Omega)}+\|\partial_{z_0}\vp(z_0,\cdot)\|^2_{H^3(\Omega)})dz_0\nonumber\\
	&\qquad+ (q\eta+1)\sum_{|\alpha|\le 4}\int_{0}^T e^{-2\eta z_0}\|D^\alpha\vp\|^2_{L^2(\Omega)}dz_0+q e^{-2\eta T}\sum_{|\alpha|\le 4}\|D^\alpha\vp(T,\cdot)\|^2_{L^2(\Omega)}\nonumber\\
	&\qquad+\frac{1}{q\eta}\int_0^T\!\!e^{-2\eta z_0}\left(\sum_{|\alpha|\le 4}\|D^\alpha\vp(z_0)\|^2_{L^2(\Omega)}+\sum_{|\alpha|\le 3}\|\mathcal{L}(D^\alpha\vp)(z_0)\|^2_{L^2(\Omega)}\right)dz_0\nonumber\\
	&\qquad+\frac{1}{q\eta^2}e^{-2\eta T}\left(\|\vp(T,\cdot)\|^2_{H^4(\Omega)}+\sum_{|\alpha|\le 3}\|\mathcal{L}(D^\alpha\vp)(T,\cdot)\|^2_{L^2(\Omega)}\right)\nonumber\\
	&\qquad+\frac{1}{q}\|\partial_{z_2}^2f|_{z_0=0}\|^2_{L^2(\Omega)}.\label{3-87}
\end{align}
By applying the equation, 
\begin{align}
	&|\partial_{z_1}^2\partial_{z_2}^2\vp|\lesssim |\mathcal{L}(\partial_{z_2}^2\vp)|+|D\partial_{z_2}^3\vp|+\sum_{|\alpha|\le 3}|D^\alpha\partial_{z_0}\vp|,\\
	&|\partial_{z_1}^3\partial_{z_2}\vp|\lesssim |\mathcal{L}(\partial_{z_1z_2}\vp)|+|\partial_{z_1}^2\partial_{z_2}^2\vp|+|\partial_{z_1}\partial_{z_2}^3\vp|+\sum_{|\alpha|\le 3}|D^\alpha\partial_{z_0}\vp|\nonumber\\
	&\qquad\qquad\!\lesssim |\mathcal{L}(\partial_{z_1z_2}\vp)|+|\mathcal{L}(\partial_{z_2}^2\vp)|+|D\partial_{z_2}^3\vp|+\sum_{|\alpha|\le 3}|D^\alpha\partial_{z_0}\vp|,\\
	&\quad{\ }|\partial_{z_1}^4\vp|\lesssim |\mathcal{L}(\partial_{z_1z_2}\vp)|+|\mathcal{L}(\partial_{z_2}^2\vp)|+|\mathcal{L}(\partial_{z_1}^2\vp)|+|D\partial_{z_2}^3\vp|+\sum_{|\alpha|\le 3}|D^\alpha\partial_{z_0}\vp|.	
\end{align}
It is clear that $\partial_{z_1}^2\partial_{z_2}^2\vp$ can be bounded by the estimated terms, namely, the derivatives of $\partial_{z_0}\vp$ and $\partial_{z_2}\vp$, and the commutator. Then $\partial^3_{z_1}\partial_{z_2}\vp$ and $\partial_{z_1}^4\vp$ can be also bounded by the controlled terms. In fact, one has
\begin{align}
	&\eta\int_{\Omega_T}\!\!e^{-2\eta z_0}|\partial_{z_1}^2\partial_{z_2}^2\vp|^2 dzdz_0+e^{-2\eta T}\!\int_{\Omega}|\partial_{z_1}^2\partial_{z_2}^2\vp|_{z_0=T}|^2dz_1dz_2\nonumber\\
	&\quad+\eta\int_{\Omega_T}\!\!e^{-2\eta z_0}|\partial_{z_1}^3\partial_{z_2}\vp|^2 dzdz_0+e^{-2\eta T}\!\int_{\Omega}|\partial_{z_1}^3\partial_{z_2}\vp|_{z_0=T}|^2dz_1dz_2\nonumber\\
	&\quad+\eta\int_{\Omega_T}\!\!e^{-2\eta z_0}|\partial_{z_1}^4\vp|^2 dzdz_0+e^{-2\eta T}\!\int_{\Omega}|\partial_{z_1}^4\vp|_{z_0=T}|^2dz_1dz_2\nonumber\\
	&\lesssim \sum_{|\alpha|\le 2}\left(\eta\int_{\Omega_T}\!\!e^{-2\eta z_0}|\mathcal{L}(D^\alpha\vp)|^2 dzdz_0+e^{-2\eta T}\!\int_{\Omega}|\mathcal{L}(D^\alpha\vp)|_{z_0=T}|^2dz_1dz_2\right)\nonumber\\
	&\quad+\eta\int_{\Omega_T}\!\!e^{-2\eta z_0}|D\partial_{z_2}^3\vp|^2 dzdz_0+e^{-2\eta T}\!\int_{\Omega}|D\partial_{z_2}^3\vp|_{z_0=T}|^2dz_1dz_2\nonumber\\
	&\quad+\sum_{|\alpha|\le 3}\left(\eta\int_{\Omega_T}\!\!e^{-2\eta z_0}|D^\alpha\partial_{z_0}\vp|^2 dzdz_0+e^{-2\eta T}\!\int_{\Omega}|D^z\alpha\partial_{z_0}\vp|_{z_0=T}|^2dz_1dz_2\right)\label{3-91}.
\end{align}

By \eqref{eq:low norm bounded by higher norm}, we have
\begin{align}
	&\sum_{|\alpha|\le 2}\left(\eta\int_{\Omega_T}\!\!e^{-2\eta z_0}|\mathcal{L}(D^\alpha\vp)|^2 dzdz_0+e^{-2\eta T}\!\int_{\Omega}|\mathcal{L}(D^\alpha\vp)|_{z_0=T}|^2d
	z_1dz_2\right)\nonumber\\
   &\quad\lesssim\sum_{|\alpha|\le 2}\left(\frac{1}{\eta }\int_{\Omega_T}e^{-2\eta z_0}|\partial_{z_0}\mathcal{L}(D^\alpha\vp)|^2 dzdz_0+\|\mathcal{L}(D^\alpha \vp)|_{z_0=0}\|^2_{L^2(\Omega)}\right)\nonumber\\
   &\quad\lesssim \sum_{|\alpha|\le 3}\frac{1}{\eta }\int_{\Omega_T}e^{-2\eta z_0}|\mathcal{L}(D^\alpha\vp)|^2 dzdz_0+\sum_{|\alpha|\le 4}\frac{1}{\eta }\int_{\Omega_T}e^{-2\eta z_0}|D^\alpha\vp|^2 dzdz_0.\label{3-92}
\end{align}

It is clear that the left hand sides of \eqref{eq:3rd order estimate of partial_z0 vp}, \eqref{3-87}, and \eqref{3-91} cover all the fourth order derivatives of $\vp$. Hence by adding up \eqref{eq:3rd order estimate of partial_z0 vp}, \eqref{3-87}, and \eqref{3-91}, we obtain a estimate of all the fourth order derivatives, but still with $D^\alpha\partial_{z_0}\vp$, $D\partial_{z_2}^3\vp$ and $\mathcal{L}(D^\alpha\vp)$ $(|\alpha|\le 3)$ (which have already been estimated) on the right hand side of the estimate. The resultant inequality is too long, so we omit to write it down. Then we substitute the estimate of $D^\alpha\partial_{z_0}\vp$ $(|\alpha|\le 3)$ in \eqref{eq:3rd order estimate of partial_z0 vp}, the estimate of $D\partial_{z_2}^3\vp$ in \eqref{3-87}, and the estimate in \eqref{3-92} into the right hand side of the resultant estimate. Next, one firtstly chooses $q$ and $\delta$ properly small, then chooses $\eta$ appropriately large, one deduces that
\begin{align}
&\sum_{|\alpha|\le| 4}\left(\eta\int_{\Omega_T}\!\!e^{-2\eta z_0}|D^\alpha\vp|^2 dzdz_0+e^{-2\eta T}\!\int_{\Omega}|D^\alpha\vp|_{z_0=T}|^2d
z_1dz_2\right)\nonumber\\
&\quad\lesssim 	\frac{1}{\eta^2}\sum_{\alpha|\le 3}\left(\eta\int_{\Omega_T}\!\!e^{-2\eta z_0}|\mathcal{L}(D^\alpha\vp)|^2 dzdz_0+e^{-2\eta T}\!\int_{\Omega}|\mathcal{L}(D^\alpha\vp)|_{z_0=T}|^2d
z_1dz_2\right)\nonumber\\
&\qquad+ \|f|_{z_0}=0\|^2_{H^3(\Omega)}.
\end{align}
\end{proof}
Combining lemma \ref{third order regularity} and lemma \ref{fourth order regularity}, it is easy to see that lemma \ref{higher regularity} holds.

\section{The nonlinear problem, proof of theorem \ref{main theorem}}\label{sec:5}
In this section, based on the well-posedness of linear problem \eqref{linear problem} in Proposition \ref{well-posedness}, we will establish the existence of the non-linear problem by constructing an iteration scheme. The iteration scheme admits an approximate sequence of the solutions. Then by showing the sequence is bounded in the higher order norm and contracted in the lower order norm, one shows that the sequence converges to the desired solution. Hence theorem \ref{main theorem} is proved.
First from Proposition $\ref{well-posedness}$, we have the following theorem:
\begin{thm}\label{thm 4.1}
Under assumptions $(\romannumeral1)-(\romannumeral4)$, there exists a smooth solution to $\eqref{linear problem}$ and there exists a constant $\bar{\eta}>0$, such that for $\eta\ge\bar{\eta}$ and any $T>0$, 
   \begin{align}
   &\sum_{|\alpha|\le| 4}\left(\eta\int_{\Omega_T}\!\!e^{-2\eta z_0}|D^\alpha\vp|^2 dzdz_0+e^{-2\eta T}\!\int_{\Omega}|D^\alpha\vp|_{z_0=T}|^2d
   z_1dz_2\right)\nonumber\\
   &\quad\lesssim \frac{1}{\eta^2}\sum_{|\beta|\le 4}\left(\eta\int_0^T e^{-2\eta z_0}\|D^\beta f(z_0,\cdot)\|^2_{L^2(\Omega)}dz_0+e^{-2\eta T}\|D^\beta f(T,\cdot)\|^2_{L^2(\Omega)}\right)\nonumber\\
   &\qquad+ \|f|_{z_0}=0\|^2_{H^3(\Omega)}.
   \end{align}
\end{thm}

\begin{proof}
Note that
\begin{align}\label{4.2}
\mathcal{L}(D^\alpha\vp)=-[D^\alpha,\mathcal{L}]\vp+D^\alpha\mathcal{L}\vp
=-[D^\alpha,\mathcal{L}]\vp+D^\alpha f.
\end{align}
For the commutator $[D^\alpha,\mathcal{L}]\vp$, it is clear that
\begin{align}
    [D^\alpha,\mathcal{L}]\vp=\sum_{i,j=0}^2\left(D^\alpha(r_{ij}\partial_{ij}\vp)-r_{ij}D^\alpha\partial_{ij}\vp\right).
\end{align}

By the Sobolev embedding theorem and assumption $(\romannumeral2)$, one has
\begin{align}
    &\|[D^\alpha,\mathcal{L}]\vp\|^2_{L^2(\Omega)}\lesssim \|\partial_{ij} \vp\|^2_{L^\infty(\Omega)}\cdot\left(\sum_{|\beta\le 4|}\|D^\beta\Phi\|^2_{L^2(\Omega)}\right)+ \|D\partial_{ij}\vp\cdot D^3\Phi\|^2_{L^2(\Omega)}\nonumber\\
    &\qquad\qquad\qquad+\|D^2\partial_{ij}\vp\|^2_{L^2(\Omega)}\cdot\|D^2\Phi\|^2_{L^\infty(\Omega)}\nonumber\\
    &\quad\lesssim \delta(\|\vp\|^2_{H^4(\Omega)}+\|\partial_{z_0}\vp\|^2_{H^3(\Omega)}+\|\partial_{z_0}^2\vp\|^2_{H^2(\Omega)})+\|D\partial_{ij}\vp\|^2_{L^4(\Omega)}\cdot\|D^3\Phi\|^2_{L^4(\Omega)}\nonumber\\
    &\quad \lesssim \delta(\|\vp\|^2_{H^4(\Omega)}+\|\partial_{z_0}\vp\|^2_{H^3(\Omega)}+\|\partial_{z_0}^2\vp\|^2_{H^2(\Omega)})+\|D\partial_{ij}\vp\|^2_{H^1(\Omega)}\cdot\|D^3\Phi\|^2_{H^1(\Omega)}\nonumber\\
    &\quad\lesssim \delta(\|\vp\|^2_{H^4(\Omega)}+\|\partial_{z_0}\vp\|^2_{H^3(\Omega)}+\|\partial_{z_0}^2\vp\|^2_{H^2(\Omega)}+\|\partial_{z_0}^3\vp\|^2_{H^1(\Omega)})\nonumber\\
    &\quad \lesssim \delta\sum_{|\beta|\le 4}\|D^\beta\vp\|^2_{L^2(\Omega)},
    \end{align}
where we have used the H{\"o}lder's inequality in the second inequality.
Hence,
\begin{align}
    &\eta\int_0^Te^{-2\eta z_0} \|[D^\alpha,\mathcal{L}]\vp(z_0,\cdot)\|_{L^2(\Omega)}^2dz_0+e^{-2\eta T} \|[D^\alpha,\mathcal{L}]\vp(T,\cdot)\|_{L^2(\Omega)}^2\nonumber\\
    &\le \delta\sum_{|\beta|\le 4}\left(\eta\int_0^T e^{-2\eta z_0}\|D^\beta\vp(z_0,\cdot)\|^2_{L^2(\Omega)}dz_0+e^{-2\eta T}\|D^\beta\vp(T,\cdot)\|^2_{L^2(\Omega)}\right).
\end{align}
This together with \eqref{4.2} gives 
\begin{align}
    &\eta\int_0^T e^{-2\eta z_0}\|\mathcal{L}(D^\alpha\vp)(z_0,\cdot)\|_{L^2(\Omega)}dz_0+e^{-2\eta T}\|\mathcal{L}(D^\alpha\vp)(T,\cdot)\|_{L^2(\Omega)}\nonumber\\
    &\quad\lesssim \delta\sum_{|\beta|\le 4}\left(\eta\int_0^T e^{-2\eta z_0}\|D^\beta\vp(z_0,\cdot)\|^2_{L^2(\Omega)}dz_0+e^{-2\eta T}\|D^\beta\vp(T,\cdot)\|^2_{L^2(\Omega)}\right)\nonumber\\
    &\qquad+ \sum_{|\beta|\le 3}\left(\eta\int_0^T e^{-2\eta z_0}\|D^\beta f(z_0,\cdot)\|^2_{L^2(\Omega)}dz_0+e^{-2\eta T}\|D^\beta f(T,\cdot)\|^2_{L^2(\Omega)}\right).\label{estimate of commutator}
\end{align}
Therefore for properly small $\delta$ and appropriately large $\eta$, we deduce that
\begin{align}
&\sum_{|\alpha|\le| 4}\left(\eta\int_{\Omega_T}\!\!e^{-2\eta z_0}|D^\alpha\vp|^2 dzdz_0+e^{-2\eta T}\!\int_{\Omega}|D^\alpha\vp|_{z_0=T}|^2d
z_1dz_2\right)\nonumber\\
&\quad\lesssim \frac{1}{\eta^2}\sum_{|\beta|\le 3}\left(\eta\int_0^T e^{-2\eta z_0}\|D^\beta f(z_0,\cdot)\|^2_{L^2(\Omega)}dz_0+e^{-2\eta T}\|D^\beta f(T,\cdot)\|^2_{L^2(\Omega)}\right)\nonumber\\
&\qquad+ \|f|_{z_0}=0\|^2_{H^3(\Omega)}.
\end{align}
\end{proof}
\begin{lem}
    For any smooth function $v$, we have
    \begin{align}
     \|e^{-\eta z_0}v\|^2_{H^s(\Omega_T)}\le \sum_{|\alpha|\le s} \|e^{-\eta z_0}D^\alpha v\|^2_{L^2(\Omega_T)}\label{relation of norms}
\end{align}
provided that $\partial^j_{z_0}v|_{z_0=0}=0$ for $j=1,\cdots,s-1$.
\end{lem}
\begin{proof}
For $T>0$, let
\begin{align}
    A(T)=\int_{\Omega_T}e^{-2\eta z_0}v^2dz.
\end{align}
Then one has
\begin{align}
    A(T)&=-\frac{1}{2\eta}\int_{\Omega_T}(e^{-2\eta z_0})_{z_0}v^2dz\nonumber\\
    &=-\frac{1}{2\eta}\int_{\Omega_T}(e^{-2\eta z_0}v^2)_{z_0}-e^{-2\eta z_0}2vv_{z_0}dz\nonumber\\
    &\le-\frac{1}{2\eta}e^{-2\eta T}\int_{\Omega}v^2dz_1dz_2+\frac{1}{2\eta}\int_{\Omega_T}e^{-2\eta z_0}(\eta v^2+\frac{1}{\eta}v_{z_0}^2)dz\nonumber\\
    &\le \frac{1}{2}A(T)+\frac{1}{2\eta^2}\int_{\Omega_T}e^{-2\eta z_0}v_{z_0}^2dz.
\end{align}
Hence we have
\begin{align}
    A(T)\le \frac{1}{\eta^2}\int_{\Omega_T}e^{-2\eta z_0}v_{z_0}^2dz.\label{4.11}
\end{align}
Now we show $\eqref{relation of norms}$ by the induction on $s$. For $s=1$,
    \begin{align}
    \|e^{-\eta z_0}v\|^2_{H^1(\Omega_T)}&=\eta^2\|e^{-\eta z_0}v\|^2_{L^2(\Omega_T)}+\|e^{-\eta z_0}Dv\|^2_{L^2(\Omega_T)}\nonumber\\
    &\le \|e^{-\eta z_0}v_{z_0}\|^2_{L^2(\Omega_T)}+\|e^{-\eta z_0}Dv\|^2_{L^2(\Omega_T)}\nonumber\\
    & \le\sum_{|\alpha|\le 1} \|e^{-\eta z_0}D^\alpha v\|^2_{L^2(\Omega_T)}.
\end{align}
Now for $k\in\mathbb{N}$, assume
\begin{equation}\label{4.31}
\|e^{-\eta z_0}v\|^2_{H^k(\Omega_T)}\le \sum_{|\alpha|\le k} \|e^{-\eta z_0}D^\alpha v\|^2_{L^2(\Omega_T)}.
\end{equation}

We are going to show
\begin{equation}\label{4.32}
\|e^{-\eta z_0}v\|^2_{H^{k+1}(\Omega_T)}\le \sum_{|\alpha|\le k+1} \|e^{-\eta z_0}D^\alpha v\|^2_{L^2(\Omega_T)}.
\end{equation}

Repeating the process for estimate \eqref{4.11} above $m$ times where $|v|^2$ in $A(T)$ is replaced by $|D^nv|^2$, we have
\begin{equation}\label{recursive formula}
\int_0^{T}e^{-2\eta t}\|D^{n}v\|^2_{L^2(\Omega)}\mathrm{d}t\le \eta^{-2m}\int_0^{T}e^{-2\eta t}\|D^{m+n}v\|^2_{L^2(\Omega)}\mathrm{d}t,
\end{equation}
provided that $\partial^l_tv|_{t=0}=0,\, l=0,1,2,\cdots, m+n-1.$
Note that
\begin{equation}\label{induction1}
\|e^{-\eta z_0}v\|^2_{H^{k+1}(\Omega_T)}= \|e^{-\eta z_0}v\|^2_{H^{k}(\Omega_T)}+ \sum_{|\alpha|=k+1}\|\mathrm{D}^{\alpha}(e^{-\eta t}v)\|_{L^2(\Omega_T)}^2
\end{equation}
and
\begin{align}\label{induction2}
\sum_{|\alpha|= k+1}\|\mathrm{D}^{\alpha}(e^{-\eta t}v)\|^2_{L^2(\Omega_T)}&=\sum_{l_1+l_2=k+1}\|(-\eta)^{l_1}e^{-\eta t}\mathrm{D}^{l_2}v\|^2_{L^2(\Omega_T)}\nonumber\\
&=\sum_{l_1+l_2=k+1}(\eta)^{2l_1}\|e^{-\eta t}\mathrm{D}^{l_2}v\|^2_{L^2(\Omega_T)}.
\end{align}
So by $\eqref{recursive formula}$, we have
\begin{align}\label{induction3}
\sum_{l_1+l_2=k+1}(\eta)^{2l_1}\|e^{-\eta t}\mathrm{D}^{l_2}v\|^2_{L^2(\Omega_T)}&\le \sum_{l_1+l_2=k+1}\|e^{-\eta t}\mathrm{D}^{l_1+l_2}v\|^2_{L^2(\Omega_T)}\nonumber\\
&=\sum_{|\alpha|=k+1}\|e^{-\eta t}\mathrm{D}^{\alpha}v\|^2_{L^2(\Omega_T)}.
\end{align}
From \eqref{4.31}, $\eqref{induction1}$--$\eqref{induction3}$, we obtain \eqref{4.32}.
Therefore, we derive the estimate \eqref{relation of norms} for any $s\in \mathbb{N}$ by the induction method.
\end{proof}

Let $\psi(z_0,z_1,z_2)=\sum_{k=0}^3\frac{\vp_k}{k!}z_0^k$, where $\vp_k=\partial^k_{z_0}\hat{\Phi}|_{z_0=0}$, \emph{i.e.}, $\vp_0=\hat{\Phi}_0$ and $\vp_1=\hat{\Phi}_1$ by the initial conditions in \eqref{eq:nolinear in z-coordinate}, and $\vp_k$ for $k=2$ or $3$ is defined by equation $\eqref{eq:nolinear in z-coordinate}_1$ and the initial conditions.
Then one defines a approximation sequence in the following manner. 

Let $\tilde{\vp}_0=0$ and suppose $\tilde{\vp}_m$ is given. Then $\tilde{\vp}_{m+1}$ is defined as the solution to the following initial boundary value problem:
\begin{equation}\label{4.14}
    \begin{cases}
    \mathcal{L}(\tilde{\vp}_m+\psi)\tilde{\vp}_{m+1}=F_m,&\mbox{in\ }\Omega_T,\\
    \mathcal{B}\tilde{\vp}_{m+1}=0,
    &\mbox{on\ }\Gamma_1,\\
    \partial_{z_2}\tilde{\vp}_{m+1}=0,&\mbox{on\ }\Gamma_2,\\
    \tilde{\vp}_{m+1}=0,\quad\partial_{z_0}\tilde{\vp}_{m+1}=0,&\mbox{on\ }\Gamma_0,
\end{cases}
\end{equation}
where
\begin{align}
    \mathcal{L}(\Psi')\Psi''\defs\sum_{i,j=0}^2\alpha_{ij}(D\Psi')\partial_{ij}\Psi'',
\end{align}
and
\begin{align}\label{4.15}
	F_m=-\mathcal{L}(\tilde{\vp}_m+\psi)\psi-\sum_{i=0}^2\alpha_i(\tilde{\vp}_m+\psi)\partial_{z_i}\tilde{\vp}_m.
\end{align}

By the compatibility conditions, we have $\mathcal{B}\psi=0$ on $\Gamma_1$ and $\partial_{z_2}\psi=0$ on $\Gamma_2$. Via Theorem \ref{thm 4.1}, the sequence $\{{\tilde{\vp}_m}\}_{m=0}^\infty$ is well-defined. Now, we will show $\tilde{\vp}_m$ converges to some function $\tilde{\vp}$, and then $\tilde{\vp}+\psi$ is a solution to the non-linear problem $\eqref{eq:nolinear in z-coordinate}$.


\begin{prop}[Boundness in the higher order norm]\label{high norm boundedness}
 There exist three constants $\delta_0>0$, $\eta_0\ge 1$, and $T_0>0$ such that for $\eta \ge \eta_0$ and $0<T\le T_0$, and for all $n\ge 0$, it holds that
\begin{align}\label{High norm boundedness}
    \|e^{-\eta z_0}\tilde{\vp}_{n}\|_{H^4(\Omega_T)}+e^{-2\eta T}\sum_{k=0}^4\sup_{0\le z_0\le T} \|\partial_{z_0}^{k}\tilde{\vp}_n(z_0,\cdot)\|_{H^{4-k}(\Omega)}\le \delta_0.
\end{align}
\end{prop}
\begin{proof}
We will prove it by the induction. 
It is easy to see $\eqref{High norm boundedness}$ is true when $n=0$, since $\epsilon$ can be selected sufficiently small.
Assume $\eqref{High norm boundedness}$ holds for $n=m\ge 0$. We will show $\eqref{High norm boundedness}$ holds for $n=m+1$. By \eqref{4.14} and Proposition $\ref{well-posedness}$, one has
   \begin{align}\label{4.17}
   	&\|e^{-\eta z_0}\tilde{\vp}_m\|^2_{H^4(\Omega_T)}+e^{-2\eta T}\sum_{k=0}^4\sup_{0\le z_0\le T} \|\partial_{z_0}^{k}\tilde{\vp}_{m+1}(z_0,\cdot)\|_{H^{4-k}(\Omega)}\nonumber\\
   	&\quad\le \frac{C}{\eta^2}\left(\|e^{-\eta z_0}F_m\|^2_{H^3(\Omega_T)}+e^{-2\eta T}\sum_{k=0}^3 \|\partial_{z_0}^{k}F_m(T,\cdot)\|^2_{H^{3-k}(\Omega)}\right).
   \end{align}
By \eqref{4.15} and $\eqref{High norm boundedness}$ holds for $n=m$, one has
   \begin{align}\label{4.18}
   	\|e^{-\eta z_0} F_m\|_{H^3(\Omega_T)}^2&\le C'\|e^{-\eta z_0}\psi\|^2_{H^5(\Omega_{T})}e^{2\eta T}\left(\delta_0^2+\|e^{-\eta z_0}\psi\|^2_{H^5(\Omega_T)}\right)\nonumber\\
   	&\quad+C' e^{2\eta T}\delta_0^2(\delta_0^2+\|e^{-\eta z_0}\psi\|^2_{H^5(\Omega_T)}).
   \end{align}
 Similarly for $k=0,1,2,3$, we have
   \begin{align}\label{4.19}
   	\|\partial_{z_0}^kF_m(T,\cdot)\|^2_{H^{3-k}(\Omega)}\le C'\epsilon^2 e^{2\eta T}\left(C\delta_0^2+\epsilon^2\right)
  +C' e^{2\eta T}\delta_0^2(C\delta_0^2+\epsilon^2).
   \end{align}
  Select $\eta_0\ge 1$ such that  $C C'\eta_0^{-2}\le \frac18$ and let $T_0$ be small such that $e^{2\eta_0 T_0}\le 2$. Then for $\delta_0\le \sqrt{\frac{1}{C}}$, one sets $0<\epsilon\le \delta_0$ in Theorem \ref{main theorem} small such that $\|e^{-\eta_0 z_0}\psi(z_0,\cdot)\|^2_{H^5(\Omega_{T_0})}\le \delta_0^2$. So it follows from \eqref{4.17}-\eqref{4.19} that,
  \begin{align}\label{20}
\|e^{-\eta z_0}\tilde{\vp}_m\|^2_{H^4(\Omega_T)}+e^{-2\eta T}\sum_{k=0}^4\sup_{0\le z_0\le T} \|\partial_{z_0}^{k}\tilde{\vp}_{m+1}(z_0,\cdot)\|_{H^{4-k}(\Omega)}
\le \delta_0^2.
  \end{align}
\end{proof}
Let $v_m=\tilde{\vp}_{m+1}-\tilde{\vp}_m$ for $m\ge 0$, then $v_m$ satisfies the following initial boundary value problem:
\begin{equation}\label{4.21}
\begin{cases}
\mathcal{L}_1(\tilde{\vp}_m+\psi)v_{m}=G_m,&\mbox{in\ }\Omega_T,\\
\mathcal{B}v_{m}=0,
&\mbox{on\ }\{z_1=0\},\\
\partial_{z_2}v_{m}=0,&\mbox{on\ }\{z_2=0\},\\
v_{m}=0,\quad\partial_{z_0}v_{m}=0,&\mbox{on\ }\{z_0=0\},
\end{cases}
\end{equation}
where
\begin{align}
	G_m&=-[\mathcal{L}(\tilde{\vp}_m+\psi)-\mathcal{L}(\tilde{\vp}_{m-1}+\psi)]\psi\nonumber\\
	&\quad-[\alpha_2(\tilde{\vp}_m+\psi)-\alpha_2(\tilde{\vp}_{m-1}+\psi)]\partial_{z_2}v_{m-1}\nonumber\\
	&\quad-[\mathcal{L}_1(\tilde{\vp}_{m}+\psi)-\mathcal{L}_1(\tilde{\vp}_{m-1}+\psi)]\tilde{\vp}_m.
\end{align}

\begin{prop}[Contraction in the lower order norm]\label{Low norm contraction}
Under the same assumptions in Theorem \ref{main theorem}, there exist two constants $\eta_*\ge \eta_0$ and $T_*\le T_0$ such that, for some $\sigma \in (0,1)$ for all $m\ge 1$ and for $\eta\ge \eta_*$ and $T\le T_*$, it holds that:
\begin{align}
&\|e^{-\eta z_0}v_{m}\|_{H^1(\Omega_T)}+e^{-2\eta T}\sum_{k=0}^1\sup_{0\le z_0\le T} \|\partial_{z_0}^{k}v_{m}(z_0,\cdot)\|_{H^{1-k}(\Omega)}\nonumber\\
&\le \sigma\left(\|e^{-\eta z_0}v_{m-1}\|_{H^1(\Omega_T)}+e^{-2\eta T}\sum_{k=0}^1\sup_{0\le z_0\le T} \|\partial_{z_0}^{k}v_{m-1}(z_0,\cdot)\|_{H^{1-k}(\Omega)}\right).\label{4.23}
\end{align}
\end{prop}
\begin{proof}
Note that
	\begin{align}
		&[\mathcal{L}(\tilde{\vp}_m+\psi)-\mathcal{L}(\tilde{\vp}_{m-1}+\psi)]\tilde{\vp}_{m}\nonumber\\
		&\quad=\sum_{i,j=0}^2[\alpha_{ij}(\tilde{\vp}_m+\psi)-\alpha_{ij}(\tilde{\vp}_{m-1}+\psi)]\partial_{ij}\tilde{\vp}_{m}\nonumber\\
		&\quad=\sum_{i,j,k=0}^3\left(\partial_k v_{m-1}\int_0^1\frac{\partial\alpha_{ij}}{\partial(\partial_k\tilde{\vp})}(\tilde{\vp}_{m-1}+\psi+\theta v_{m-1}) d\theta\right)\partial_{ij}\tilde{\vp}_{m}.\nonumber
	\end{align}
So for $\eta \ge \eta_0$ and $T\le T_0$, we have
	\begin{align}
		\|e^{-\eta z_0}[\mathcal{L}(\tilde{\vp}_m+\psi)-\mathcal{L}(\tilde{\vp}_{m-1}+\psi)]\tilde{\vp}_{m}\|_{L^2(\Omega_T)}\le C(\delta_0)\|e^{-\eta z_0}v_{m-1}\|_{H^1(\Omega_T)}.\label{4.24}
	\end{align}
By a similar argument, we also derive that
	\begin{align}
		&\|[\alpha_i(\tilde{\vp}_m+\psi)-\alpha_i(\tilde{\vp}_{m-1}+\psi)]\partial_{z_2}v_{m-1}\|_{L^2(\Omega_T)}\le C(\delta_0)\|e^{-\eta z_0}v_{m-1}\|_{H^1(\Omega_T)}\label{4.25}
	\end{align}
	and
	\begin{align}
		&\|e^{-\eta z_0}[\mathcal{L}(\tilde{\vp}_{m}+\psi)-\mathcal{L}(\tilde{\vp}_{m-1}+\psi)]\psi\|_{L^2(\Omega_T)}\le C(\delta_0)\|e^{-\eta z_0}v_{m-1}\|_{H^1(\Omega_T)}.\label{4.26}
	\end{align}
	By \eqref{4.24}-\eqref{4.26} and the first order energy estimate \eqref{first order estimate: 2nd step}, one has
	\begin{align}
		\|e^{-\eta z_0}v_{m}\|_{H^1(\Omega_T)}+e^{-2\eta T}&\sum_{k=0}^1\sup_{0\le z_0\le T} \|\partial_{z_0}^{k}v_{m}(z_0,\cdot)\|_{H^{1-k}(\Omega)}\nonumber\\
&\quad\le C(\delta_0)\eta^{-1}\|e^{-\eta z_0}v_{m-1}\|_{H^1(\Omega_T)}.
	\end{align}
Choose $\eta_*\ge \eta_0$ such that $\sigma\defs C(\delta_0)\eta_*^{-1}<1$. Note that $\delta_0$ does not depend on the weight $\eta$. Then select $T_*\le T_0$ such that $e^{2\eta_1 T_*}\le 2$, we obtain \eqref{4.23} for $\eta\ge \eta _*$ and $T\le T_*$. This completes the proof of this proposition.
\end{proof}

Now, we are ready to conclude this paper by showing the main theorem.
\begin{proof}[Proof of Theorem \ref{main theorem}] Proposition \ref{Low norm contraction} implies that $\{\tilde{\vp}_m\}_m^\infty$ is a Cauchy sequence. Hence there exists a function $\tilde{\vp}\in H^1(\Omega_T)$ such that for $\eta\ge \eta_*$ and $T\le T_*$,
\begin{align}
	\lim_{m\rightarrow \infty}\left(\|e^{-\eta z_0}(\tilde{\vp}_m-\tilde{\vp})\|_{H^1(\Omega_T)}+e^{-2\eta T}\sum_{k=0}^1 \|\partial_{z_0}^{k}(\tilde{\vp}_m-\tilde{\vp})\|_{L^\infty(0,T;H^{1-k}(\Omega))}\right)=0.\nonumber
\end{align}
Moreover, Proposition \ref{high norm boundedness} implies that $\tilde{\vp}\in H^4(\Omega_T)$. Hence by passing the limit $m\rightarrow \infty$ in the approximation problem \eqref{4.14}, one deduces that $\tilde{\vp}+\psi$ is a smooth solution to the non-linear problem \eqref{nonlinear initial boundary value problem}. Moreover, by \eqref{High norm boundedness} and the assumptions in Theorem \ref{main theorem}, one has
$\|e^{-\eta z_0}(\tilde{\vp}+\psi)\|_{H^4(\Omega_T)}\le C\delta_0$, for $\eta\ge \eta_*$ and $T\le T_*$. This completes the proof of Theorem \ref{main theorem}.
\end{proof}

\bigskip
\section*{Acknowledgements}

The research of Beixiang Fang was supported in part by the National Key R\&D Program of China (No. 2020YFA0712000) and National Natural Science Foundation of China (No. 11971308 and 11631008).
The research of Wei Xiang was supported in part by the Research Grants Council of the HKSAR, China (Project No. CityU 11303518, CityU 11304820, CityU 11300021, and CityU 11311722). The research of Feng Xiao was supported in part by the National Natural Science Foundation of China (No.12201209).


\begin{thebibliography}{10}
\bibitem{AlinhacCPDE}S. Alinhac, Existence d'ondes de raréfaction pour des syst\`emes quasi-lin\'eaires hyperboliques multidimensionnels.(French) [Existence of rarefaction waves for multidimensional hyperbolic quasilinear systems]\textit{ Comm. Partial Differential Equations} 14:173--230, 1989.

 \bibitem{Alinhac1989}S. Alinhac, Unicit\'e d'ondes de raréfaction pour des syst\`emes quasi-lin\'eaires hyperboliques multidimensionnels. (French) [Uniqueness of rarefaction waves for multidimensional hyperbolic quasilinear systems]\textit{ Indiana Univ. Math. J.}, 38:345--363, 1989. 

\bibitem{BCF} M. Bae, G.-Q. Chen, and M. Feldman, Regularity of solutions to regular shock reflection for potential flow, \textit{Ivent. Math.} 175:505--543, 2009.

\bibitem{BCFMemoris}M. Bae, G.-Q. Chen, and M. Feldman, Prandtl-Meyer reflection configurations, transonic shocks, and free boundary problems, to appear in \textit{Memoirs of the AMS}, {\em Preprint at arXiv:1901.05916}, 2019.

\bibitem{BenzoniSerre2007} S. Benzoni-Gavage and D. Serre, \textit{Multidimensional hyperbolic partial differential equations. First-order systems and applications}, \textit{Oxford Mathematical Monographs}. The Clarendon Press, Oxford University Press, Oxford, 2007.

\bibitem{Bressan}
A. Bressan, \textit{Hyperbolic systems of conservation laws: The one dimensional
Cauchy problem}, Oxford University Press: Oxford, 2000.

\bibitem{Chen2003Multidimensional}
G.-Q. Chen and M. Feldman,
\newblock Multidimensional transonic shocks and free boundary problems for
nonlinear equations of mixed type,
\newblock {\em J. Amer. Math. Soc.}, 16:461--494, 2003.

\bibitem{CFeldman2010} G.-Q. Chen and M. Feldman, Global solutions of shock reflection by large-angle wedges for potential flow, \textit{Ann. of Math.}, 171:1067--1182, 2010. 

\bibitem{CFeldman2018} G.-Q. Chen and M. Feldman,
\textit{Mathematics of Shock Reflection-Diffraction and Von Neumann's Conjecture}, Annals of Mathematics Studies, Princeton, 2018.

\bibitem{CFHX} G.-Q. Chen, M. Feldman, J. Hu, and W. Xiang, Loss of regularity of solutions of shock diffraction problem by a convex cornered wedge to the potential flow equation, {\em SIAM J. Math. Anal.}, 52:1096--1114, 2020.


\bibitem{CFeldmanX} G.-Q. Chen, M. Feldman, and W. Xiang, Convexity of transonic shocks in self-similar coordinates,  \textit{Arch. Ration. Mech. Anal.}, 238:47--124,  2020.




\bibitem{Chen5} S.-X. Chen, Mach configuration in pseudo-stationary compressible flow, \textit{J. Amer. Math. Soc.}, 21:63--100, 2008.


 

\bibitem{CouSecchi2004} J.-F. Coulombel and P. Secchi, The stability of compressible vortex sheets in two space dimensions, \textit{Indiana Univ. Math. J.} 53:941--1012, 2004.
 
\bibitem{CouSecchi} J.-F. Coulombel and P. Secchi, Non-linear compressible vortex sheets in two space dimensions, \textit{Annales scientifiques de l'\'ecole Normale Sup\'erieure, S\'erie 4}, 41:85--139, 2008.

\bibitem{CouSecchi2009}
 J.-F. Coulombel and P. Secchi, Uniqueness of 2-D compressible vortex sheets, \textit{Commun. Pure Appl. Anal.}, 8:1439--1450, 2009. 
\bibitem{Dafermos}
C. M. Dafermos, \textit{Hyperbolic conservation laws in continuum physics}, Springer-Verlag: Berlin, 2010.

\bibitem{GS} F. Gazzola and P. Secchi, Inflow-outflow problems for Euler equations in a rectangular cylinder, \textit{Non-linear Differ. Equ. Appl.}, 8:195--217, 2001.

\bibitem{FHXX}
B.-X. Fang, F.-M. Huang, W. Xiang and F. Xiao,
\newblock Persistence of the steady planar normal shock structure in 3-D unsteady potential flows,
\newblock {\em Preprint at arXiv:2108.09207}, 2021.

\bibitem{FXX}
B.-X. Fang, W. Xiang, and F. Xiao,
\newblock Persistence of the steady normal shock structure for the unsteady
potential flow,
\newblock {\em SIAM J. Math. Anal.}, 52:6033--6104, 2020.

\bibitem{Godin}P. Godin, On the breakdown of 2D compressible Eulerian flows in bounded impermeable regions with corners, \textit{J. Math. Pures Appl.} 137:178--212, 2020.

\bibitem{GodinMAMS}P. Godin, The 2D compressible Euler equations in bounded impermeable domains with corners, \textit{Mem. Amer. Math. Soc.} 269, v+72 pp, 2021.

\bibitem{IKAWA1968}
M. Ikawa,
\newblock Mixed problems for hyperbolic equations of second order,
\newblock {\em J. Math. Soc. Japan}, 20:580--608,
1968.

\bibitem{IKAWA1970}
M. Ikawa,
\newblock On the mixed problem for hyperbolic equations of second order with
the neumann boundary condition,
\newblock {\em Osaka J. Math.}, 7:203--223, 1970.


\bibitem{EL} V. Elling and T.-P. Liu, Supersonic flow onto a solid wedge, \textit{Comm. Pure. Appl. Math.}, 61:1347--1448, 2008.

\bibitem{R.Courant1948}
K.O.~Friedrichs and {R. Courant},
\newblock {\em {Supersonic flow and shock waves}},
\newblock Springer-Verlag, New York, 1948.

\bibitem{LiZheng2009}
J.-Q. Li and Y.-X. Zheng, Interaction of rarefaction waves of the two-dimensional self-similar Euler equations, \textit{Arch. Ration. Mech. Anal.}, 193:623--657, 2009.

\bibitem{LiZheng2010}
J.-Q. Li and Y.-X. Zheng, Interaction of four rarefaction waves in the bi-symmetric class of the two-dimensional Euler equations, \textit{Comm. Math. Phys.}, 296:303--321, 2010.



\bibitem{Li-Witt-Yin2018}
J. Li, I. Witt, and H.-C. Yin, Global multidimensional shock waves of 2-dimensional and 3-dimensional unsteady potential flow equations, \textit{SIAM J. Math. Anal.}, 50:933--1009, 2018. 

\bibitem{Majda1983S} A. Majda, The stability of multidimensional shock fronts, \textit{Mem. Amer. Math. Soc.} 41, 1983.

\bibitem{Majda1983E} A. Majda, The existence of multidimensional shock fronts, \textit{Mem. Amer. Math. Soc.}, 43, 1983.

\bibitem{Majda1987}
A. Majda and E. Thomann,
\newblock {Multi-dimensional shock fronts for second order wave equations}.
\newblock {\em Comm. Partial Differential Equations},
12:777--828, 1987.

\bibitem{Osher} S. Osher, An ill posed problem for a hyperbolic equation near a corner, \textit{Bull. Amer. Math. Soc.}, 79:1043--1044, 1973.

\bibitem{Osher1} S. Osher, An ill-posed problem for a strictly hyperbolic equation in two unknowns near a corner, \textit{Bull. Amer. Math. Soc.}, 80:705--708, 1974.
\bibitem{Smoller} J. Smoller, \textit{Shock Waves and Reaction--Diffusion Equations}, Springer, New York, 1994.
\bibitem{XuYin2020}
G. Xu and H.-C. Yin, On global smooth solutions of 3-D compressible Euler equations with vanishing density in infinitely expanding balls, \textit{Discrete Contin. Dyn. Syst.}, 40:2213--2265, 2020.

\bibitem{Yuan2012}
H.-R. Yuan,
\newblock {Persistence of shocks in ducts},
\newblock {\em Nonlinear Anal.},
75:3874--3894, 2012.
\bibitem{Zheng_Book}
Y.-X. Zheng, \textit{Systems of Conservation Laws: Two-Dimensional Riemann Problems}, Birkh\"auser, Boston, 2001.
\end{thebibliography}
\end{document}